\documentclass[11pt]{amsart}
\usepackage{amssymb,latexsym, amsmath, amsxtra}
\usepackage[dvips]{graphics}
\usepackage{epsfig}

\theoremstyle{plain}
\newtheorem{theorem}{Theorem}[section]

\newtheorem{conjecture}[theorem]{Conjecture}
\newtheorem{conj}[theorem]{Conjecture}

\newtheorem{lemma}[theorem]{Lemma}

\theoremstyle{definition}
\newtheorem{definition}[theorem]{Definition}
\theoremstyle{remark}

\numberwithin{equation}{section}

\DeclareMathOperator*{\Res}{Res}
\DeclareMathOperator*{\Det}{Det}

\begin{document}

\title{Orthogonal and symplectic $n$-level densities}

\abstract
In this paper we apply to the zeros of families of $L$-functions with orthogonal or symplectic symmetry the method that Conrey and Snaith~\cite{conrey2008ce} used to calculate the $n$-correlation of the zeros of the Riemann zeta function. This method uses the Ratios Conjectures~\cite{conrey2007autocorrelatio} for averages of ratios of zeta or $L$-functions.  Katz and Sarnak~\cite{katz1999zz} conjecture that the zero statistics of families of $L$-functions have an underlying symmetry relating to one of the classical compact groups $U(N)$, $O(N)$ and $USp(2N)$. Here we complete the work already done with $U(N)$~\cite{conrey2008ce} to show how new methods for calculating the $n$-level densities of eigenangles of random orthogonal or symplectic matrices can be used to create explicit conjectures for the $n$-level densities of zeros of $L$-functions with orthogonal or symplectic symmetry, including all the lower order terms.  We show how the method used here results in formulae that are easily modified when the test function used has a restricted range of support, and this will facilitate comparison with rigorous number theoretic $n$-level density results.
\endabstract

\author {A.M. Mason}
\address{NDM Experimental Medicine, University of Oxford, John Radcliffe Hospital, Oxford, OX3 9DU, United Kingdom} \email{amy.mason@ndm.ox.ac.uk}

\author{N.C. Snaith}
\address{School of Mathematics,
University of Bristol, Bristol, BS8 1TW, United Kingdom}
\email{N.C.Snaith@bris.ac.uk}

\thanks{
The first author was supported by an EPSRC Doctoral Training Grant.  The second author was partially supported by an EPSRC
Advanced Research Fellowship and subsequently sponsored by the Air
Force Office of Scientific Research, Air
  Force Material Command, USAF, under grant number FA8655-10-1-3088.} \maketitle

\tableofcontents

\section{Introduction}

\subsection{Motivation}
\label{sec:intro}

Long-standing results in random matrix theory show that the $n$-point correlation functions (in some situations these are also called $n$-level densities) of eigenvalues from ensembles of random matrices such as $U(N)$, $SO(N)$ and $USp(2N)$ can be written concisely as $n$-dimensional determinants  of matrices whose elements are the kernel belonging to the particular ensemble (see, for example, \cite{conrey2005note} where the kernels and correlation functions are written down for these groups). These results are elegant and exact, and their determinantal form is very useful for calculations within random matrix theory (RMT). However, the corresponding quantity in number theory, the $n$-point correlation function (or level density)  of the complex zeros of an $L$-function, or of families of $L$-functions, does not seem to take this concise determinantal form, except in a suitable asymptotic limit where the statistics are expected to agree with RMT (see \cite{kn:conrey1,kn:keasna03,kn:snaith10} for reviews on the connection between number theory and random matrix theory).  This leads us to a question: is there a different, albeit less elegant, way to write the random matrix eigenvalue correlation functions that might help to make precise conjectures about the form of the correlations of zeros?  Furthermore, given that the determinantal form is not available in the number theory case, is there a natural form in which to write the $n$-correlations of the zeros that is useful for further applications?

The first question has been answered in the case of the Riemann zeta function, which has zeros displaying the statistics of eigenvalues of matrices from $U(N)$, with Haar measure, in the appropriate limit.  The $n$-point correlation functions of the Riemann zeros, including the lower order correction terms to the limiting random matrix theory result, were first studied by Bogomolny and Keating \cite{kn:bk96,kn:bk13a} and then by Conrey and Snaith \cite{kn:consna07,conrey2008ce}.  In particular, Conrey and Snaith derive the $n$-point correlation function of $U(N)$ eigenvalues using average values of ratios of characteristic polynomials in a method that is precisely reproducible in the number theory case using a conjecture on average values of ratios of the Riemann zeta function (see \cite{conrey2007autocorrelatio} and also \cite{conrey2007applications}).

It is this first question in the case of matrices from $SO(2N)$ and from $USp(2N)$ and the families of $L$-functions associated with them that is the main focus of this current paper. The $n$-level density functions of eigenvalues from matrices in $SO(2N)$ are calculated in Section \ref{sec:ortho} using averages of ratios of characteristic polynomials, with the final result appearing in Theorem \ref{thm:orthoNcorr2}.  The steps are mirrored in Section \ref{sec:LF1} using the Ratios Conjecture for a family of $L$-functions associated with elliptic curves.  The zeros of this family are expected to behave statistically like the eigenvalues of $SO(2N)$. The Ratios Conjecture is used to derive the $n$-level desities of zeros of the $L$-functions in this family, complete with lower order terms. This result can be found in Theorem \ref{thm:LFNcorr}. In Section \ref{sub:symp} eigenvalues of matrices from $USp(2N)$ are considered in the same way, see Theorem \ref{thm:Uspcorr}, and Section \ref{sec:dir} demonstrates the method with a family of $L$-functions showing symplectic symmetry. The result can be found in Theorem \ref{thm:DLNcorr}.

As reflected in the second question, the motivation for this work is not simply to derive the detailed conjectures for the $n$-correlation of zeros of families of $L$-functions, but to obtain formulae that are useful for further applications.  For the past 40 years, since Montgomery's pioneering work on the statistics of the two-point correlation function of the Riemann zeros \cite{montgomery1973pc}, there has been a catalogue of work where a rigorous result on $n$-correlation functions (often called $n$-level densities in number theory literature) of the zeros of an $L$-function or family of $L$-functions is derived and then an attempt is made to match it with either the limiting RMT correlation function (for example \cite{rudnick1996zp,rubinstein2001low,kn:gao05,kn:young,levinson2012n,kn:entrodrud}), or a conjecture derived using a Ratios Conjecture that includes lower order terms (for example \cite{kn:mil07,kn:mil09,kn:Goesetal,kn:milmon,kn:huymilmor}).   For general $n$ this comparison is difficult, and is made more so by the fact that all the rigorous number theoretical results hold only for test functions with restricted support (of their Fourier transform) and by the fact that if comparison is made with the determinantal RMT form, this has nothing like the shape of the number theoretical results.  By test function we mean the function $f$ in the definition of the $n$-level density (\ref{eq:generaln-level}), which is sampled at the positions of an $n$-tuple of zeros.  There is an equivalent test function in the definition of the $n$-level density of eigenvalues, for example in (\ref{thm:orthoNcorr1}), that is sampled at the positions of an $n$-tuple of eigenvalues.  In \cite{conrey2012restricted} Conrey and Snaith demonstrate that in the $U(N)$ random matrix theory case, the form of the $n$-point correlation function resulting from the method using the average of ratios of characteristic polynomials allows for immediate simplification when the support of the test function is restricted - something that is certainly not true of the determinantal form.  This allows them to apply to the $n$-point correlation function for eigenvalues from $U(N)$ the same restriction to the support of the test function that was used by Rudnick and Sarnak in \cite{rudnick1996zp} when looking at the $n$-point correlation function for zeros of a general $L$-function. Then the identical structure of the two expressions reveals that they coincide in their respective asymptotic limits.  This addresses the second question above in the case of the zeros of the Riemann zeta function.  The second question in the case of $SO(2N)$ or $USp(2N)$ matrices is approached in Section \ref{sec:restricted} where it is shown that the form of the random matrix theory $n$-level density given in this paper leads easily to a simpler result when the support of the test function is restricted. It will be the subject of future work to show whether this allows more straightforward confirmation of the limiting random matrix form for rigorous calculations of the $n$-level density for associated families of $L$-functions.

\subsection{Background to the Ratios Conjectures}
\label{sec:background}

The connection between RMT and number theory became apparent in the 1970s when Montgomery \cite{montgomery1973pc}, after a conversation with Dyson \cite{kn:dyson}, conjectured that in the appropriate scaling limit the pair correlation of the zeros of the Riemann zeta function have the same form as the equivalent statistic for eigenvalues of unitary matrices of large dimension drawn at random from the group $U(N)$ endowed with Haar measure.  There followed work, both of an analytical and a numerical nature, that provided more evidence that zeros of the Riemann zeta function high on the critical line and eigenvalues of large random unitary matrices show the same local statistics when both sets of points are scaled to have a mean spacing of unity (see for example \cite{kn:bogkea95,kn:bogkea96,kn:hejhal94,
kn:odlyzko89,rudnick1996zp}).  Even decades later, Montgomery's conjecture is proven only for test functions with restricted support.

Katz and Sarnak \cite{kn:katzsarnak99a,katz1999zz} then considered statistics of zeros of $L$-functions near the critical point (the point at which the real axis crosses the critical line on which the complex zeros are expected to lie).  They predict that in a natural family of $L$-functions these zeros would, when averaged across the family, display the same statistical behaviour as the eigenvalues near 1 of matrices chosen at random with respect to Haar measure from the classical compact groups $U(N)$, $O(N)$ or $USp(2N)$.  As in the case of the Riemann zeta function, this correspondence is expected to occur when the zeros and eigenvalues being considered are scaled to have a mean spacing of unity and in an appropriate asymptotic limit: normally the zero statistics tend to the eigenvalue statistics of large matrices as a parameter intrinsic to the definition of the family tends to infinity.  The correspondence with random matrix theory has been proven for specific families and for test functions with restricted support (see for example \cite{duenez2006low,FouIwa03,kn:gao05,Gul05,HugMil07,HugRud03b,ILS99,Mil04_a,kn:milpec12,ozluk1999distribution,RicottaRoyer09,royer2001petits,rubinstein2001low,kn:young}).

It is a long-standing conjecture that the moments of the Riemann zeta function
 grow asymptotically like
\begin{equation}\label{eq:zetamom}
\frac{1}{T}\int_0^T |\zeta(1/2+it)|^{2\lambda} dt \sim a_\lambda
\frac{g_\lambda}{\Gamma(1+\lambda^2)} (\log T)^{\lambda^2}
\end{equation}
for large $T$, where we have chosen to write the coefficient in this way for ease of comparison later on with random matrix theory.  Here $a_\lambda$ is a product over primes
\begin{eqnarray}
a_\lambda& =  &\prod_p \left(1-\tfrac{1}{p}\right)^{\lambda^2}
\sum_{m=0}^{\infty} \left( \frac{\Gamma(m+\lambda)} {m!
\Gamma(\lambda)} \right)^2 p^{-m}, \label{eq:zetaak}
\end{eqnarray}
but $g_{\lambda}$ is very difficult to calculate.  Moments of the Riemann zeta function have been the subject of study for one hundred years, but still $g_1=1$
\cite{kn:harlitt18} and  $g_2=2$ \cite{kn:ing26} are the only
proven values.  Conjectures have risen from number theoretical methods, for example \cite{kn:congho92,kn:congon,kn:dgh03}, but at the start of this millennium there was no
viable conjecture for any values beyond $k=4$ and random matrix theory provided the only available  method for predicting higher moments.

The characteristic polynomial, $\Lambda_A(s)$, of a random unitary matrix, $A$, is used to model the Riemann
zeta function.  The analogous quantity to (\ref{eq:zetamom}) is

\begin{equation}
\label{eq:momgen} M_N(\lambda)=
\int_{U(N)}|\Lambda_A(1)|^{2\lambda}dA_{\rm Haar},
\end{equation}
where the integration is with respect to Haar measure.
This average can be computed using Selberg's integral (see \cite{kn:mehta3}),
giving
\begin{eqnarray}
\label{eq:M} M_N(\lambda)&=&\prod_{j=1}^{N} \frac{\Gamma(j)
\Gamma(j+2\lambda)}{(\Gamma(j+\lambda))^{2}}\\
\nonumber &\sim& f(\lambda)N^{\lambda^2},
\end{eqnarray}
where the final line is the asymptotic for large $N$.  The Barnes double gamma function \cite{kn:barnes00} can be used
to express $f(\lambda)$ concisely:
\begin{equation}\label{eq:fGbarnes}
f(\lambda)=\frac{(G(1+\lambda))^2}{G(1+2\lambda)}.
\end{equation}

It is conjectured \cite{kn:keasna00a} that
\begin{equation}
f(\lambda)=\frac{g_\lambda}{\Gamma(1+\lambda^2)}.
\end{equation}
This
conjecture agrees with established or independently
conjectured results for $g_{k}$ for $k=1,2,3,4$.

We note that when we compare (\ref{eq:zetamom}) and (\ref{eq:M}) we see $N$ is playing the role of $\log T$, an equivalence that is also observed by equating the density of zeta zeros to the density
of eigenvalues.   This observation is key, allowing comparison of random matrix results
to number theoretic quantities even when the limits $N\rightarrow \infty$ and $T\rightarrow \infty$ have not been reached.

For a family of $L$-functions, a similar quantity is the average, over the family, of the values
of the $L$-functions at the critical point.  If we denote by $\mathcal{F}$ a family of  $L$-functions $L_f(s)$, indexed by $f$, and $\mathcal{F}^*$ is the number of members of the family, then the moment

\begin{equation}
\frac{1}{\mathcal{F}^*} \sum_{f\in \mathcal{F}} (L_f(1/2))^\lambda
\end{equation}
can be modelled by a quantity similar to (\ref{eq:momgen}), only with the average
performed over the appropriate subgroup of $U(N)$ comprised of unitary, orthogonal or symplectic matrices \cite{kn:confar00,kn:keasna00b}: e.g.
\begin{equation}
\int_{SO(2N)}(\Lambda_A(1))^\lambda dA_{Haar}, \:\: A\in SO(2N).
\end{equation}

It was described in \cite{conrey2005integra} that more precise conjectures for the moments of zeta and $L$-functions, including lower-order
terms, can be obtained from the ``shifted" moments.  These are quantities of
the form

\begin{eqnarray}\label{eq:shiftedNT}
&&\frac{1}{T}\int_0^T
\zeta(1/2+it+\alpha_1)\cdots\zeta(1/2+it+\alpha_k)\nonumber
\\
&&\qquad\qquad\times
\zeta(1/2-it-\alpha_{k+1})\cdots\zeta(1/2-it-\alpha_{2k}) dt.
\end{eqnarray}
In that paper Conrey,
Farmer, Keating, Rubinstein and Snaith propose a ``recipe" for conjecturing a precise expression for such moments and for similar averages
over families of $L$-functions.  One argument in support of the validity of these conjectures, is that the expression produced
has an identical structure the analogous (and rigourous) moment calculation in random
matrix theory for a quantity such as \cite{kn:cfkrs1} (see also \cite{kn:brezhik00}):

\begin{eqnarray}\label{eq:shiftedRMT}
&&\int_{U(N)}
\Lambda_A(e^{-\alpha_1})\cdots\Lambda_A(e^{-\alpha_k})
\Lambda_{A^*} (e^{\alpha_{k+1}})\Lambda_{A^*}(e^{\alpha_{2k}})
dA_{\rm Haar}.
\end{eqnarray}
Here $A^*$ is the conjugate transpose of the matrix $A\in U(N)$.

A further generalization was made by Conrey, Farmer and Zirnbauer
to averages of ratios of the Riemann zeta function,
\begin{equation}\label{eq:zetaratio}
\frac{1}{T}\int_0^T
\frac{\prod_{k=1}^K\zeta(1/2+it+\alpha_k)\;\prod_{\ell=K+1}^{K+L}\zeta(
1/2-it-\alpha_{\ell})}
{\prod_{q=1}^Q\zeta(1/2+it+\gamma_q)\;\prod_{r=1}^R\zeta(1/2-it-\delta_r)}dt.
\end{equation}
and to averages of ratios of $L$-functions
\begin{equation}
\frac{1}{\mathcal{F}^*} \sum_{f\in \mathcal{F}} \frac{ \prod_{k=1}^K L_f(1/2+\alpha_k)} {\prod_{q=1}^Q L_f(1/2+\gamma_q)}.
\end{equation}
In \cite{conrey2007autocorrelatio} they develop conjectures for these ratios that are informed by analogous  random matrix ratio calculations
\cite{ConreyPreprint,kn:bumgam06,conrey2005ar}.

The ``recipe" for creating shifted moment or ratio conjectures involves several steps where we neglect terms that may be of the same size as the main term that emerges at the end of the procedure.  The miracle is that these neglected terms appear to cancel out.  The recipe appears to produce a correct conjecture down the the level of the constant term or below.

The main steps of the recipe are as follows; these are reproduced from \cite{conrey2007autocorrelatio}, although somewhat simplified as we don't need their full generality for the families with orthogonal or symplectic symmetry considered here.  For particular cases where a Ratios Conjecture is worked though, see for example \cite{conrey2007applications} or \cite{huynh2009lowe}.

{\bf The recipe:}  Here we assume we have a family $\mathcal{F}$ of $L$-functions, indexed by $f\in \mathcal{F}$.    Each $L$-function in the family has a functional equation of the form (assuming we have normalised so that the critical line has real part 1/2)
\begin{equation}
\mathcal{L}_f(s)=\varepsilon_f \mathcal{X}_f(s) \overline{\mathcal{L}_f}(1-s).
\end{equation}
As described in \cite{conrey2005integra}, we measure the ``size" of an $L$-function by $c(f)=|\mathcal{X}'_f(\tfrac{1}{2})|$.  This is the logarithm of the quantity which for certain families of $L$-functions is commonly referred to as the ``conductor", and by the argument principle $ c(f)/(2\pi)$ gives the density of zeroes near the critical point.  

The $L$-function also has an approximate functional equation of the form
\begin{equation}\label{eq:afe}
\mathcal{L}_f(s) = \sum \frac{a_n(f)}{n^s} +\varepsilon_f \mathcal{X}_f(s) \sum\frac{{\overline{a_n(f)}}}{n^{1-s}} +{\rm remainder},
\end{equation}
for appropriate ranges of summation over $n$.
In addition, for each $L$-function we can write its reciprocal as a Dirichlet series
\begin{equation}\label{eq:rds}
\frac{1}{\mathcal{L}_f(s)} = \sum_{n=1}^\infty \frac{\mu_{f}(n)}{n^s},
\end{equation}
for appropriate coefficients $\mu_{f}(n)$; for example, in the case of the Riemann zeta function the Dirichlet series for the reciprocal of zeta features $\mu(n)$, the M\"obius function, which is multiplicative and is equal to -1 when $n=p$ is prime and is equal to 0 when $n=p^\ell$ with $\ell >1$.

We are interested in understanding the ratio
\begin{equation}
\mathcal{L}_f(s;{\boldsymbol \alpha}_K;{\boldsymbol\gamma}_Q) = \frac{ \mathcal{L}_f(s+\alpha_1)\cdots \mathcal{L}_f(s+\alpha_K)} {\mathcal{L}_f(s+\gamma_1)\cdots \mathcal{L}_f(s+\gamma_Q)},
\end{equation}
averaged over $f\in \mathcal{F}$.  For the $L$-functions we encounter in this paper, $\mathcal{L}_f(s)=\overline{\mathcal{L}_f}(s)$.

The first step of the recipe is to replace each $L$-function in the numerator with the two terms from its approximate functional equation (\ref{eq:afe}), leaving out the remainder term.  Then each $L$-function in the denominator is replaced by the series (\ref{eq:rds}). Multiplying out the $K$ approximate functional equations in the numerator results in $2^K$ terms.  Each of those terms can be written in the form
\begin{equation}
({\rm product \; of\;} \varepsilon_f {\rm \;factors})({\rm product \;of\;} \mathcal{X}_f {\rm \;factors}) \sum_{n_1,\ldots,n_{K+Q}} ({\rm summand}).
\end{equation}
In this expression the recipe indicates we replace each product of $\varepsilon_f$'s by its expected value when averaged over the family.  Similarly, the summand is replaced by its average over the family.  The summations that resulted from the approximate functional equation are over a restricted range of integers.  The next step is to extend the sums to the full range from 1 to infinity.   If the result of these steps is denoted as $M_f(s;\boldsymbol{\alpha}_K;\boldsymbol{\gamma}_Q)$, then the corresponding Ratio Conjecture can be written as
\begin{equation}\label{eq:ratioconjformulation}
\frac{1}{\mathcal{F}^*}\sum_{f\in \mathcal{F}} L_f(\tfrac{1}{2};\boldsymbol{\alpha}_K;\boldsymbol{\gamma}_Q)=\frac{1}{\mathcal{F}^*}\sum_{f\in \mathcal{F}} M_f(\tfrac{1}{2};\boldsymbol{\alpha}_K,\boldsymbol{\gamma}_Q) (1+O(e^{-\delta c(f))})),
\end{equation}
for some $\delta>0$.  The exponent $-\delta c(f)$ in the error term here indicates a power saving over the size of the main term.  There are specific instances in which more is known about this size of this error term for averages of ratios of $L$-functions in various families - these are described in Section \ref{sec:error}, where there is a full discussion of this error term.

In \cite{conrey2008ce} Conrey and Snaith demonstrate how to extract information about the $n$-point correlation function of zeros from the $n$-over-$n$ ratio of the Riemann zeta function.  They use Cauchy's theorem to express a sum over the heights of the zeros of the zeta function, $\gamma_j$,  as
\begin{eqnarray}\label{eq:nfold}
&&\sum_{0<  \gamma_1,\dots ,\gamma_n\leq T}
f(\gamma_1,\dots,\gamma_n)\nonumber
\\
&&\qquad\qquad=\frac{1}{(2\pi i)^n} \int_{\mathcal C}\dots
\int_{\mathcal C} f(-iz_1,\dots,-iz_n)\prod_{j=1}^n
\frac{\zeta'}{\zeta}(1/2+z_j)~dz_1\dots dz_n,
\end{eqnarray}
where $\mathcal C$ is a positively oriented contour which encloses
a subinterval of the imaginary axis from zero to $T$.  The product of logarithmic derivatives of the zeta function can be obtained from a Ratios Conjecture by differentiating with respect to the $\alpha$ parameters in the numerator in an expression like (\ref{eq:zetaratio}). This connection between ratios of $L$-functions and zero statistics is what we draw on in the current paper.  We will not go further into the details of \cite{conrey2008ce} as all steps of the method are covered explicitly in the present paper for matrices or $L$-functions with orthogonal or symplectic symmetry.

We can see this brings us in a full circle back to the statistics of zeros of zeta and $L$-functions, as information about correlations of zeros are encoded in these Ratios
Conjectures.  We will see how this information is extracted in this present paper.  First, in Section \ref{sec:ortho} we will use ratios of characteristic polynomials to derive a new form of the $n$-level density for the orthogonal group $SO(2N)$.  Section \ref{sub:symp} follows, illustrating the same calculation for unitary symplectic matrices in $USp(2N)$.  Then the Ratios Conjectures of Conrey, Farmer and Zirnbauer are used in Sections \ref{sec:zerostatselliptic} and \ref{sec:dir} to conjecture the $n$-level densities for zeros of specific families of $L$-functions.  In Section \ref{sec:restricted} we demonstrate why this new form of the random matrix theory $n$-level density simplifies immediately when the support of the Fourier transform of the test function (the equivalent of $f$ in (\ref{eq:nfold})) is restricted.  Section \ref{app:12} gives examples of 1- and 2-level densities explicitly.

\section{Eigenvalue Statistics of Orthogonal Matrices}
\label{sec:ortho}

Let $X$ be a $2N \times 2N$ matrix with real entries $X = \left(x_{jk}\right)$. We define the transpose matrix, $X^{*} = (x_{kj})$ and call $X$ \emph{orthogonal} if $X{X}^{*} = I$. We let $SO\left(2N\right)$ denote the group of all such orthogonal matrices with $\Det X=1$.  All the eigenvalues of a matrix $X \in SO(2N)$ have absolute value 1 and can be written as $e^{\pm i\theta_1 }, \cdots,
e^{\pm i\theta_N}$ with $0\leq \theta_1, \cdots,
\theta_N \leq \pi $.

The result we are proving in this section is:

\begin{theorem}\label{thm:orthoNcorr1} The $n$-level density for eigenvalues of matrices from $SO(2N)$ can be written as
\begin{align}
\begin{split}
&\int_{SO(2N)}\sum^{N}_{\substack{j_1, \cdots, j_n =1 \\ j_i \neq j_k \forall i,k}}f(\theta_{j_1}, \cdots, \theta_{j_n}) dX\\
&= \frac{1}{(2 \pi i)^n}\sum_{K \cup L \cup M= \{1, \cdots, n\}}{(2N)^{|M|}\left(\int_{0}^{\pi}\right)^n J^{*}(-iz_K \cup i z_L)}\\
& \qquad \qquad \times f(z_1, \cdots, z_n)dz_1 \cdots dz_n,
\end{split}
\end{align}
where $J^{*}(-iz_K\cup i z_L)$ will be defined in (\ref{def:Jstar})  and the sum in the right hand side involving $K$, $L$ and $M$ is over disjoint subsets of $\{1,\ldots,n\}$.
\end{theorem}

\subsection{Integral Theorem}\label{sub:integral}

In this section we will express the $n$-level density of eigenvalues of matrices from $SO(2N)$ as contour integrals on the complex plane.  We will see in the subsequent sections that moving the contours onto the imaginary axis (or, by a change of variables, the real axis) to arrive at Theorem \ref{thm:orthoNcorr1} is non-trivial and will take some work.

\begin{theorem}[Integral theorem for $SO(2N)$ matrices]\label{thm:RMortho1}
Let $C_{-}$ denote the path from $ - \delta - \pi i$ up to $- \delta + \pi i$ and let $C_{+}$
denote the path from $  \delta - \pi i$ up to $ \delta + \pi i$. Let $f$ be a $2\pi$-periodic,
holomorphic function of $n$ variables such that
\begin{align}
f\left(\theta_{j_1}, \cdots, \theta_{j_n}\right) = f\left(\pm \theta_{j_1}, \cdots, \pm
\theta_{j_n}\right)
\end{align}

Then
\begin{align}\label{nostar}
\begin{split}
2^n &\int_{SO\left(2N\right)} {\sum_{j_1, \cdots, j_n =1}^N {f\left(\theta_{j_1}, \cdots,
\theta_{j_n}\right) dX_{SO(2N)}} } \\
&= \frac{1}{\left(2 \pi i \right)^n} \sum_{K \cup L \cup M = \left\{1, \cdots,
n\right\}}{(2N) ^{\left|M\right|}} \\
& \qquad \times \int_{C_+^K} {\int_{C_-^{L \cup M}}{J\left(z_K \cup - z_L\right)f\left(iz_1, \cdots,
iz_n\right)dz_1 \cdots d z_n}}
\end{split}
\end{align}
where $z_K = \left\{z_k: k \in K\right\}$ and $-z_L = \left\{-z_l: l \in L\right\}$,
$\int_{C_+^K} {\int_{C_-^{L \cup M}}}$ means we are integrating all the variables in $z_K$ along the
$C_+$ path and all others up the $C_-$ path and \\
\begin{align} \label{def:J}
J\left(A\right) = \int_{SO(2N)}{\prod_{\alpha \in
A}{(-e^{-\alpha})\frac{\Lambda_X^{'}}{\Lambda_X}(e^{-\alpha})} dX_{SO(2N)}}.
\end{align}

Here $\Lambda_X(e^{\alpha})=\det(I-e^{\alpha}X^*)$ is the characteristic polynomial of $X$; $K,L,M$ are finite sets of distinct integers; $A$ is a finite set of complex numbers; and $dX_{SO(2N)}$ indicates integration with respect to the Haar measure.
\end{theorem}

Note that the sum over the eigenangles in \eqref{nostar} is not restricted to a sum over distinct indices. These extra terms, which do not figure in Theorem \ref{thm:orthoNcorr1},  will cancel out the residues caused by moving the contours onto the imaginary axis in Section~\ref{sub:ncorr}.

$X$ is an orthogonal matrix, so it has $2N$ eigenvalues $e^{\pm i \theta_1 }, \ldots, e^{\pm i\theta_N }$ such that $0 \leq \theta_1, \ldots, \theta_N \leq \pi$.  The proof follows by considering $g(z) = \Lambda_X(e^z)$, the characteristic polynomial of $X$. Using Cauchy's theorem we can show that
\begin{align}
\sum_{j_1, \cdots, j_n =1}^{N}{f\left(\theta_{j_1}, \cdots, \theta_{j_n}\right)} = \frac{1}{(2\pi
i)^N} \int_{C^n}{\frac{g^{'}}{g}(z_1, \cdots, z_n)f(\frac{z_1}{i}, \cdots \frac{z_n}{i})dz_1 \cdots
dz_n}
\end{align}
where $C$ is a positively oriented rectangular contour around a interval of the imaginary axis of
length $2\pi$ e.g. with corners $\pm i\pi \pm \delta$. By the periodicity of $f$, the horizontal
integrals cancel and we
use the functional equation and a change of variables from $z \rightarrow -z$ to get the result.
\begin{proof}[Proof of Theorem \ref{thm:RMortho1}]
The characteristic polynomial of $X$, \begin{align}
g \left(z \right)&= \Lambda_X\left(e^z\right)\\
  &=\prod_{j=1}^{N}{(1-e^ze^{i \theta_j})(1-e^ze^{- i \theta_j})},
\end{align}
has zeros at $z_j = \pm i \theta_j + 2 \pi m i, m \in \textbf{Z}, 1\leq j\leq N$.  Then by
Cauchy's Theorem $\frac{g^{'}\left(z\right)}{g\left(z\right)} f\left(\frac{z}{i}\right)$ has poles
at $z_j = \pm i \theta_j + 2 \pi m i$ with residue
\begin{align}
\frac{{g^{'} \left({z_j}\right)}}{{g^{'}\left({z_j}\right)}} f\left(\frac{z_j}{i}\right)= f\left(\frac{z_j}{i}\right)= f(\theta_j).
\end{align}
\\

So let $C$ be the positively-orientated rectangle with vertices $\pm \delta \pm \pi i$, where
$\delta$ is a small positive number. We can express the sum
\begin{align}
\begin{split}
2^n &\sum_{j_1, \cdots , j_n=1}^{N}{f\left(\theta_{j_1}, \cdots, \theta_{j_n}\right)}\\
&=\sum_{j_1, \cdots, j_n =1}^{N}\sum_{\epsilon \in \{1,-1\}^n}{f\left(\epsilon_1 \theta_{ j_1}, \cdots, \epsilon_n \theta_{j_n}\right)}
\end{split}\\
&= \left(\frac{1}{2 \pi i}\right)^n \int_{C} \cdots \int_{C} \prod_{j=1}^{n}{\frac{g^{'}\left(z_j\right)}{g\left(z_j\right)} f\left(\frac{z_1}{i}, \cdots,\frac{z_n}{i}\right) d z_1 \cdots d z_n}.
\end{align}

The $2^{n}$ term in front of the sum comes from counting all the possible combinations of $+\theta_i$ and $-\theta_j$. \\

We average this over $X \in SO \left(2N\right)$, and apply a change of variable from $z_j
\rightarrow -z_j$. Let $dG$ denote $f\left(iz_1, \cdots, iz_n\right) d z_1, \cdots dz_n$
\begin{align}
\begin{split}
2^n &\int_{SO\left(2N\right)} {\sum_{j_1, \cdots, j_n =1}^N {f\left(\theta_{j_1}, \cdots,
\theta_{j_n}\right) dX} }\\
&= \frac{1}{\left(2 \pi i\right)^n} \int_{C^n} {J\left(\{z_1, \cdots, z_n\}\right) dG}
\end{split}
\end{align}
where
\begin{align}
J\left(A\right) = \int_{SO\left(2N\right)}{\prod_{\alpha \in A}
{-e^{-\alpha}\frac{\Lambda^{'}_X}{\Lambda_{X}}\left(e^{-\alpha}\right) dX}}. \label{def:Jortho}
\end{align}

By the periodicity of the function $f$, the horizontal segments of the contour cancel. Only the
vertical paths need to be considered.
\begin{align}
&\int_{C^n} {J\left(\left\{z_1, \cdots, z_n \right\}\right) dG}= \left( \int_{C_+}{} - \int_{C_-}{}\right)^n J\left(\left\{z_1, \cdots, z_n\right\}\right)dG.
\end{align}

So the expression is a sum of $2^n$ terms, where each term is a $n$-fold integral each variable
being integrated over $C_-$ or $C_+$. Another way to write this is
\begin{align}
\begin{split}
&\int_{C^n} {J\left(\left\{z_1, \cdots, z_n\right\}\right) dG}\\
&= \sum_{K\cup L = \left\{1, \cdots, n\right\}}{(-1)^{\left|L\right|}\int_{C_+^K} \int_{C_-^L}{J\left(\{z_1, \cdots,
z_n\}\right)dG}}
\end{split}\\
\begin{split}
&= \int_{SO\left(2N\right)} \sum_{K \cup L = \left\{1, \cdots, n\right\}}{(-1)^{\left|L\right|}\prod_{j \in K}
{\int_{C_+}{-e^{-z_j}\frac{\Lambda^{'}_X}{\Lambda_X} \left(e^{-z_j}\right)}}} \\
& \qquad \times \prod_{j \in L}{\int_{C_-}{-e^{-z_j}\frac{\Lambda^{'}_X}{\Lambda_X} \left(e^{-z_j}\right) dG dX}}
\end{split}
\end{align}
where the sum over $K \cup L$ is a sum over disjoint sets.

For each variable $z_j$ on $C_-$, we replace
$-e^{-\alpha}\frac{\Lambda^{'}_X}{\Lambda_{X}}\left(e^{-\alpha}\right)$ with $ -2N +
e^{\alpha}\frac{\Lambda^{'}_X}{\Lambda_{X}}\left(e^{\alpha}\right)$ (obtained by differentiating the functional equation $\Lambda_X(e^z)=(e^z)^{2N}\Lambda_X(e^{-z})$).
\begin{align}
\begin{split}
&\int_{C^n} {J\left(\left\{z_1, \cdots, z_n\right\}\right) dG}\\
&= \int_{SO\left(2N\right)}{ \sum_{\substack {K\cup L  \\ =\left\{1, \cdots, n\right\}}}{(-1)^{\left|L\right|}\prod_{j \in K}{\int_{C_+}{-e^{-z_j}\frac{\Lambda^{'}_X}{\Lambda_X} \left(e^{-z_j}\right)}}}} \\
& \qquad \times\prod_{j \in L} {\int_{C_-}{ \left(-2N+ e^{z_j}\frac{\Lambda^{'}_X}{\Lambda_X} \left(e^{z_j}\right)\right)dG}}
\end{split}\\
\begin{split}
&= \int_{SO\left(2N\right)} {\sum_{\substack{ K \cup L \cup M \\ = \left\{1, \cdots, n\right\}}}{(-1)^{\left|L \cup M\right|}\prod_{j \in K}{\int_{C_+}{-e^{-z_j}\frac{\Lambda^{'}_X}{\Lambda_X} \left(e^{-z_j}\right)}}}} \\
& \qquad \times \prod_{j \in L}
{\int_{C_{-}}{ \left( e^{z_j}\frac{\Lambda^{'}_X}{\Lambda_X} \left(e^{z_j}\right)\right)}} \prod_{j\in M} \int_{C_{-}} \left(-2N\right) dG dX
\end{split} \\
&= \sum_{\substack{K \cup L \cup M \\ = \left\{1, \cdots, n\right\}}}{ (2N)^{\left|M\right|} \int_{C_+^K}{\int_{C_-^{L \cup M}}{J\left(\{z_K \cup-z_L\}\right)dG}}}
\end{align}
where the sums over $K \cup L \cup M$ are sums over disjoint sets.

By the definition of $J(A)$ in \eqref{def:Jortho} this gives us the required result.
\end{proof}

We can then use the Ratios Theorem to find alternative ways of expressing $J\left(A\right)$.

\subsection{Ratios Theorem} \label{sub:ratios}

Here we write down the average of ratios of characteristic polynomials of the even orthogonal group. The following form is rewritten into set notation from the result of Conrey, Forrester and Snaith \cite{conrey2005ar} - they credit the original result to a preprint by Conrey, Farmer and Zirnbauer \cite{ConreyPreprint}.

Take finite sets $A$ and $B$ of complex numbers. Let $N \geq |B|$ and $\Re(\beta) >0 \forall \beta \in B$ and consider
\begin{align} \label{def:orthoR}
R(A;B) &= \int_{SO(2N)}{\frac{\prod_{\alpha \in A}{\Lambda_{X}(e^{-\alpha})}}{\prod_{\beta \in
B}{\Lambda_{X}(e^{- \beta})}}dX}\\
 &= \sum_{D \subseteq A}{e^{-2N\sum\limits_{\delta \in D}{\delta}}\sqrt{\frac{Z(D^-\cup (A \backslash D),D^-\cup (A \backslash D))Z(B,B)Y(B)}{Z(D^-\cup (A \backslash D),
B)^2Y(A \backslash D)Y(D^-)}}}
\end{align}
where $D^{-}=\{-\alpha: \alpha \in D\}$, $A \backslash D=\{\alpha \in A, \alpha \not\in D\}$ and
\begin{align}
z(x) &= \frac{1}{1-e^{-x}},\\
Y(A)&= \prod_{\alpha \in A}{z(2\alpha)}\label{def:orthoY},\\
Z(A,B) &= \prod_{\substack{\alpha \in A \\ \beta \in B}}{z(\alpha + \beta)}\label{def:orthoZ}
\end{align}

We can express $J(A)$ as defined in (\ref{def:Jortho}) in terms of $R(A;B)$
\begin{align}
J(A) = \left.{\prod_{\alpha \in A}{\frac{d}{d\alpha} R(A; B)}}\right|_{B=A}.
\end{align}

So by differentiating the
Ratios Theorem we can obtain a theorem about averages of logarithmic derivatives. \\

\begin{theorem}\label{thm:ortho2}

Let $A$ be a finite set of complex numbers where $\Re(\alpha)>0$ for $\alpha \in A$ and $|A| \leq N$, then
$J(A) = J^{*}(A)$ where
\begin{align}
J(A) &= \int\limits_{SO(2N)}{\prod_{\alpha \in A}{(-e^{-\alpha})\frac{\Lambda_X^{'}}{\Lambda_X}(e^{-\alpha})} dX}\\
\begin{split}
J^*(A) &= \sum_{D\subseteq A}{e^{-2N \sum\limits_{\delta \in D}{\delta}}(-1)^{|D|} \sqrt{\frac{Z(D,D)Z(D^{-},D^{-})Y(D)}{Y(D^{-})Z^{\dag}(D^{-},D)^2}}}\label{def:Jstar}\\
& \qquad \times \sum_{\substack{A \backslash D = W_1 \cup \cdots \cup W_R \\ |W_r| \leq 2}}{\prod_{r=1}^{R}{H_D(W_r)}}
\end{split}
\end{align}
where the sum over $W_i$ is a sum over all distinct set partitions of $A\backslash D$.
\begin{align} \label{def:HD}
H_D(W) =  \begin{cases}
\left( \sum\limits_{\delta \in D}{\frac{z^{'}(\alpha -\delta)}{z(\alpha - \delta)}-\frac{z^{'}(\alpha +\delta)}{z\left( \alpha + \delta\right)}}\right) - \frac{z^{'}}{z}(2\alpha) & W = \{\alpha\} \subset A \backslash D\\
\left(\frac{z^{'}}{z}\right)^{'}(\alpha + \widehat{\alpha}) & W= \{\alpha, \widehat{\alpha}\} \subset A \backslash D\\
1 & W= \emptyset
    \end{cases}
\end{align}
and
\begin{align}
z(x) &= \frac{1}{1-e^{-x}},\\
Y(A)&= \prod_{\alpha \in A}{z(2\alpha)}, \\
Z(A,B) &= \prod_{\substack{\alpha \in A \\ \beta \in B}}{z(\alpha + \beta)}. \label{def:Z2}
\end{align}
The dagger adds a
restriction that a factor $z(x)$ is omitted if its argument is zero.
\end{theorem}

The main steps in the proof are pulling the differentiation inside the sum over subsets $D \subseteq A$, and then separating the differentiations in cases by whether $\alpha \in D$ or $\alpha \not\in D$. The differentiation by $\alpha \in D$ are relatively straightforward, but to do the remaining differentiations we need to use logarithmic differentiation, as in the unitary case \cite{conrey2007applications}.  We do not include the details here as it is a simpler case of the proof of Theorem~\ref{LFJconj} presented later.

\subsection{Residue Lemma}\label{sub:residue}
In this section, we will locate the poles of $J^{*}(A)$ and calculate the residue
of these poles.

It is clear the only possible poles of $J^{*}(A)$ are when $\alpha = - \beta$ for some $\alpha, \beta \in A$, or when $\alpha = 0$. We need to know what the residues are at these poles.
\begin{theorem}[Residue Theorem for $SO(2N)$ matrices] \label{thm:residuelemma}
Let A be a finite set of complex numbers, let $\alpha^{*},\beta^{*} \in A$ and $A^{'}=A\backslash \{\alpha^{*},\beta^{*}\}$ and let $J^{*}(A)$ as defined in (\ref{def:Jstar}). Then there is a simple pole when $\alpha^{*}=-\beta^{*}$ and the poles cancel at $\alpha^{*}=0$.
\begin{align}\label{eq:resAB}
\Res_{\alpha^{*} =-\beta^{*}} \left(J^{*}(A)\right) &= J^{*}\left(A^{'} \cup \{\beta^{*}\}\right) + J^{*}\left(A^{'} \cup \{-\beta^{*}\}\right)+2N J^{*}\left(A^{'}\right)\\
\Res_{\alpha^{*} =0} \left(J^{*}(A)\right) &= 0.\label{eq:resA0}
\end{align}
\end{theorem}

We will consider the cases of (\ref{eq:resAB}) and (\ref{eq:resA0})  separately. First we will prove that the pole at $\alpha^{*}= -\beta^{*}$ is simple and show that (\ref{eq:resAB}) holds.

Since proofs of this structure are going to come up again, we define
\begin{align}
  A^{'} &= A\backslash \{\alpha^{*}, \beta^{*}\}\\
  D^{'} &= D \cap A^{'}\\
  (A \backslash D)^{'} &= (A \backslash D) \cap A^{'}
  \end{align}
where $A \backslash D = \{a \in A, a \not\in D\}$.
\begin{definition}\label{def:PropP}
We say that the meromorphic functions $Q$ and $H$ of several variables have property $\mathbf{P_f}$ if the following four conditions hold for a continuous single variable function $f$.
\end{definition}
\begin{description}

\item[P1]If $\alpha^*,\beta^*\in A\backslash D$, then $Q(D)$ is
independent of $\alpha^*$ and $\beta^*$ and
\begin{eqnarray}
H(D,W)=\left\{ \begin{array}{ll}
\frac{1}{(\alpha^*+\beta^*)^2}+O(1)
& \mbox{ if $W=\{\alpha^*,\beta^*\}$ }\\
O(1) & \mbox{ otherwise }
\end{array} \right.
\end{eqnarray}
\item[P2] If $\alpha^*\in D$ and $\beta^*\in A\backslash D$, then $Q(D)$ is
regular when $\alpha^*=-\beta^* $ and
\begin{eqnarray}
H(D,W)=\left\{ \begin{array}{ll}
\frac{1}{\alpha^*+\beta^*}+O(1)
& \mbox{ if $W=\{\beta^*\}$ }\\
O(1) & \mbox{ otherwise }
\end{array} \right.
\end{eqnarray}
\item[P3] If $\alpha^*\in A\backslash D$ and $\beta^*\in D$, then $Q(D)$ is
regular when $\alpha^*=-\beta^*$ and
\begin{eqnarray}
H(D,W)=\left\{ \begin{array}{ll}
\frac{1}{\alpha^*+\beta^*}+O(1)
& \mbox{ if $W=\{\alpha^*\}$ }\\
O(1) & \mbox{ otherwise }
\end{array} \right.
\end{eqnarray}
\item[P4] If $\alpha^*\in D$ and $\beta^*\in D$, then
$Q(D)=\big(\frac{-1}{(\alpha^*+\beta^*)^2}+O(1)\big)Q_1(D)$
where
\begin{eqnarray}
&&Q_1(D)=Q(D')\big(1-(\alpha^*+\beta^*)\big(f(\beta^*)+
H(D',\{\alpha^*\})|_{\alpha^*=-\beta^*}+
H(D',\{\beta^*\})\big)\nonumber
\\
&&\qquad\qquad\qquad\qquad\qquad\qquad\qquad\qquad\qquad\qquad
\qquad\qquad+O(|\alpha^*+\beta^*|^2)\big)
\end{eqnarray}
and
\begin{eqnarray}
&&H(D,W)=H(D',W)-(\alpha^*+\beta^*)\big(H(D',W \cup \{\alpha^*\})_{\alpha^*=-\beta^*}
+H(D',W \cup \{\beta^*\})\big)\nonumber
\\
&&\qquad\qquad\qquad\qquad\qquad\qquad\qquad\qquad\qquad\qquad
\qquad\qquad\quad+O(|\alpha^*+\beta^*|^2).
\end{eqnarray}

\end{description}
We will choose our $f(x)$ dependant on what family of matrices we are considering.
\begin{lemma}\label{lem:P}
If $Q(D)$ and $H(D,W)$ have property $\mathbf{P_f}$ and
\begin{align}
J^{*}(A) &= \sum_{D \subseteq A}{P_D(A \backslash D)}\\
 &=\sum_{D \subseteq A}{Q(D) \sum_{\substack{ A \backslash D= \bigcup_r{W_r}\\ |W_r|\leq 2}}{\prod_{r}{H(D,W_r)}}},\label{eq:line2P}
\end{align}

(where $P_D$ is defined by comparison with (\ref{eq:line2P})),  then
\begin{align}
\Res_{\alpha^{*} \rightarrow \beta^{*}}(J^{*}(A)) &= f(\beta^*) J^{*}\left(A^{'}\right)+ J^{*}\left(A^{'}\cup \{\beta^{*}\}\right)+ J^{*}\left(A^{'}\cup \{-\beta^{*}\}\right).\end{align}
\end{lemma}

\begin{proof}
If $\alpha^*,\beta^* \in A\backslash D$ then
\begin{align}
P_D(A \backslash D) &= Q(D) \sum_{\substack{A \backslash D= \bigcup_r{W_r}\\ |W_r|\leq 2}}{\prod_{r}{H(D,W_r)}}\\
&= \frac{1}{(\alpha^{*}+\beta^{*})^2}Q(D)\left(\sum_{\substack{ (A \backslash D)^{'}= \bigcup_r{W_r}\\ |W_r|\leq 2}}{\prod_{r}{H(D,W_r)}}\right) + O(1)\\
&=  \frac{1}{(\alpha^{*}+\beta^{*})^2} P_D((A \backslash D)^{'}) +O(1),
\end{align}
by {\bf P1}.  So $\Res\limits_{\alpha^{*} \rightarrow -\beta^{*}}(P_D(A \backslash D))=0$.\\

If $\alpha^*\in D$ and $\beta^*\in A\backslash D$, then
\begin{align}
P_D(A \backslash D) &= Q(D)\sum_{\substack{A \backslash D = \bigcup_r{W_r} \\ \left|W_r\right|\leq 2}}{\prod_r{H(D,W_r)}}\\
&= \frac{1}{\alpha^{*} + \beta^{*}}Q(D)\sum_{\substack{(A \backslash D)^{'} = \bigcup_r{W_r} \\ \left|W_r\right|\leq 2}}{\prod_r{H(D,W_r)}} + O(1)\\
&= \frac{1}{\alpha^{*} + \beta^{*}} P_{D^{'}\cup\{-\beta^{*}\}}((A \backslash D)^{'})+ O(1),
\end{align}
by {\bf P2}.  Thus
\begin{align}
\Res_{\alpha^{*} \rightarrow \beta^{*}}(P_D(A \backslash D))=P_{D^{'}\cup\{-\beta^{*}\}}((A \backslash D)^{'}).
\end{align}
\\
If $\alpha^*\in A\backslash D$ and $\beta^*\in D$, then, similarly to the previous case,
\begin{align}
P_D(A \backslash D) &= \frac{1}{\alpha^{*} + \beta^{*}}Q(D)\sum_{\substack{(A \backslash D)^{'} = \bigcup_r{W_r} \\ |W_r|\leq 2}}{\prod_r{H(D,W_r)}} + O(1)\\
&=\frac{1}{\alpha^{*} + \beta^{*}}P_{D}((A \backslash D)^{'})+O(1),
\end{align}
by {\bf P3}.  This allows us to conclude that  $\Res\limits_{\alpha^{*} \rightarrow -\beta^{*}}(P_D(A \backslash D))=P_{D}((A \backslash D)^{'})$.\\

If $\alpha^*\in D$ and $\beta^*\in D$, then
\begin{align}
P_D(A \backslash D)&= Q(D)\sum_{\substack{A \backslash D=\bigcup_r{W_r}\\ |W_r|\leq 2}}{\prod_r{H(D,W_r)}}\nonumber\\
\begin{split}
&= -\frac{Q(D^{'})}{(\alpha^{*}+\beta^{*})^2}\sum_{\substack{A \backslash D= \bigcup_r{W_r}\\ |W_r|\leq 2}}{\prod_r{H(D^{'},W)}}\\
    & \qquad+ \frac{Q(D^{'})}{\alpha^{*}+\beta^{*}}\sum_{\substack{A \backslash D= \bigcup_r{W_r}\\ |W_r|\leq 2}}{\prod_r{H(D^{'},W_r)} \left(\vphantom{\frac{H(D^{'},W_r)}{H(D^{'},W_r)}}f(\beta^*)+ H(D^{'},\{\beta^{*}\})+ H(D^{'},\{-\beta^{*}\})\right.}\\
            &\qquad \qquad \qquad +\left. \sum_{r}{\frac{H(D^{'},W_r \cup \{\beta^{*}\})+H(D^{'},W_r \cup \{-\beta^*\})}{H(D^{'},W_r)}}\right)\\
    &\qquad+ O(1).
\end{split}\\
 \begin{split}&= -\frac{1}{(\alpha^{*}+\beta^{*})^2}P_{D^{'}}(A \backslash D)
 \\
    & \qquad+ \frac{1}{\alpha^{*}+\beta^{*}}\left(f(\beta^*) P_D\left(A \backslash D\right)+P_{D^{'}}\left((A \backslash D) \cup \{\beta^{*}\}\right)+P_{D^{'}}\left((A \backslash D) \cup \{-\beta^{*}\}\right)\right)
           \\
    &\qquad+ O(1), \label{eq:bob}
\end{split}
\end{align}
by {\bf P4}.  Therefore
\begin{align}
\Res_{\alpha^{*} \rightarrow \beta^{*}}\left(P_D(A \backslash D)\right) = f(\beta^*) P_D\left(A \backslash D\right)+P_{D^{'}}\left((A \backslash D) \cup \{\beta^{*}\}\right)+P_{D^{'}}\left((A \backslash D) \cup \{-\beta^{*}\}\right).
\end{align}
We obtain the result of this Lemma by combining the results for the four cases.
\begin{align}
\begin{split}
\Res_{\alpha^{*} \rightarrow \beta^{*}}(J^{*}(A)) &= \sum_{\substack{D\subseteq A\\ \alpha^{*}, \beta^{*} \in D}}{f(\beta^*) P_D(A \backslash D)+P_{D^{'}}\left((A \backslash D) \cup \{\beta^{*}\}\right)+P_{D^{'}}\left((A \backslash D)\cup \{-\beta^{*}\}\right)}\\
&\qquad+ \sum_{\substack{D\subseteq A\\ \alpha^{*} \in D, \beta^{*} \in A \backslash D}}{P_{D^{'} \cup \{-\beta^{*}\}}\left((A \backslash D)^{'}\right)} \\
& \qquad +\sum_{\substack{D\subseteq A\\ \beta^{*} \in D, \alpha^{*} \in A \backslash D}}{P_{D}\left((A \backslash D)^{'}\right)}
\end{split}\\
\begin{split}
&= f(\beta^*) \sum_{D \subseteq A^{'}}{P_D(A \backslash D)} + \sum_{D \subseteq A^{'}\cup \{\beta^{*}\}}{P_D(A \backslash D)}\\
& \qquad + \sum_{D \subseteq A^{'}\cup \{-\beta^{*}\}}{P_D(A \backslash D)}
\end{split}\\
&= f(\beta^*) J^{*}\left(A^{'}\right)+ J^{*}\left(A^{'}\cup \{\beta^{*}\}\right)+ J^{*}\left(A^{'}\cup \{-\beta^{*}\}\right).
\end{align}
Note that the $\frac{1}{(\alpha^{*}+\beta^{*})^2}$ term in (\ref{eq:bob}) features $P_{D^{'}}(A \backslash D)$ which is equal to $P_{D}\left((A \backslash D)^{'}\right)$ when $\alpha^{*}, \beta^{*} \in A \backslash D$, so the term cancels with the $\frac{1}{(\alpha^{*}+\beta^{*})^2}$ term from the first case, confirming that the pole at $\alpha^{*} \rightarrow \beta^{*}$ is simple.

\end{proof}

It is straightforward to show that if $f(\beta^*)=2N$,
\begin{equation}\label{eq:QD}
Q(D)=\sum_{D\subseteq A}{e^{-2N \sum\limits_{\delta \in D}{\delta}}(-1)^{|D|} \sqrt{\frac{Z(D,D)Z(D^{-},D^{-})Y(D)}{Y(D^{-})Z^{\dag}(D^{-},D)^2}}}
\end{equation}
and
\begin{equation}
H(D,W)=H_D(W)= \begin{cases}\label{eq:HDW}
\left( \sum\limits_{\delta \in D}{\frac{z^{'}(\alpha -\delta)}{z(\alpha - \delta)}-\frac{z^{'}(\alpha +\delta)}{z\left( \alpha + \delta\right)}}\right) - \frac{z^{'}}{z}(2\alpha) & W = \{\alpha\} \subset A \backslash D\\
\left(\frac{z^{'}}{z}\right)^{'}(\alpha + \widehat{\alpha}) & W= \{\alpha, \widehat{\alpha}\} \subset A \backslash D\\
1 & W= \emptyset
    \end{cases}
\end{equation}
then {\bf P1} to {\bf P4} hold, and so we have proved (\ref{eq:resAB}).

Now we will prove (\ref{eq:resA0}).  Again we define a general property.
\begin{definition}\label{def:PropR}
We say that the meromorphic functions $Q$ and $H$ of several variables have property $\mathbf{R}$ if the following two conditions hold.
\end{definition}
\begin{description}

\item[R1]If $\alpha^*\in A\backslash D$, then $Q(D)$ is
independent of $\alpha^*$  and
\begin{eqnarray}
H(D,W)=\left\{ \begin{array}{ll}
\frac{1}{2\alpha^*}+O(1)
& \mbox{ if $W=\{\alpha^*\}$ }\\
O(1) & \mbox{ otherwise }
\end{array} \right.
\end{eqnarray}
\item[R2] If $\alpha^*\in D$ then
\begin{equation}H(D,A \backslash D)|_{\alpha^{*}=0}= H(D^{'},A \backslash D)
\end{equation}
 and
 \begin{equation}
 Q(D)=\frac{-1}{2\alpha^*} Q(D') +O(1).
 \end{equation}
 \end{description}

 \begin{lemma}\label{lem:R}
If $Q(D)$ and $H(D,W)$ satisfy these properties and
\begin{align}
J^{*}(A) &= \sum_{D \subseteq A}{P_D(A \backslash D)}\\
 &=\sum_{D \subseteq A}{Q(D) \sum_{\substack{ A \backslash D= \bigcup_r{W_r}\\ |W_r|\leq 2}}{\prod_{r}{H(D,W_r)}}}
\end{align}
(where $P_D$ is defined by the above equation), then
\begin{align}
\Res_{\alpha^{*} \rightarrow0}(J^{*}(A)) &= 0.\end{align}
\end{lemma}

\begin{proof}
This is proved in a similar manner to Lemma \ref{lem:P}.  We have that if $\alpha^*\in A\backslash D$ then by {\bf R1}
\begin{align}
\Res_{\alpha^{*} \rightarrow 0}\left(P_D(A \backslash D)\right) = \frac{1}{2}P_D\left((A \backslash D)^{'}\right).
\end{align}
If $\alpha^*\in D$ then by {\bf R2}
\begin{align}
\Res_{\alpha^{*} \rightarrow 0}(P_D(A \backslash D))&= \frac{-1}{2}P_{D^{'}}\left(A \backslash D\right).
\end{align}
Combining the two cases gives us
\begin{align}
\Res_{\alpha^{*} \rightarrow 0}(J^{*}(A)) &= \sum_{\substack{D\subseteq A \\ \alpha^{*} \in A \backslash D}}{\frac{1}{2}P_D\left((A \backslash D)^{'}\right)} - \sum_{\substack{D\subseteq A \\ \alpha^{*} \in D}}{\frac{1}{2}P_{D^{'}}(A \backslash D)}\\
&= \sum_{D\subseteq A^{'} }{\frac{1}{2}P_D\left((A \backslash D)^{'}\right)} - \sum_{D\subseteq A^{'}}{\frac{1}{2}P_{D^{'}}(A \backslash D)}\\
&=0.
\end{align}

\end{proof}

With the definitions (\ref{eq:QD}) and (\ref{eq:HDW}) it is straightforward to show that properties {\bf R1} and {\bf R2} hold, thus proving (\ref{eq:resA0}).

\subsection{$n$-level Density of Orthogonal Matrices}\label{sub:ncorr}

So far we have expressed the $n$-level density of eigenangles of orthogonal matrices in terms of contour integrals. Recalling that $C_{-}$ denotes the path from $ - \delta - \pi i$ up to $- \delta + \pi i$ and $C_{+}$ the path from $  \delta - \pi i$ up to $ \delta + \pi i$, we take a  $f$ be a $2\pi$-periodic,
holomorphic function of $n$ variables such that
\begin{align}
f\left(\theta_{j_1}, \cdots, \theta_{j_n}\right) = f\left(\pm \theta_{j_1}, \cdots, \pm
\theta_{j_n}\right)
\end{align}
Then we have shown
\begin{align}
\begin{split}
2^n &\int_{SO\left(2N\right)} {\sum_{j_1, \cdots, j_n =1}^N {f\left(\theta_{j_1}, \cdots,
\theta_{j_n}\right) dX_{SO(2N)}} } \\
&= \frac{1}{\left(2 \pi i \right)^n} \sum_{K \cup L \cup M = \left\{1, \cdots,
n\right\}}{ (2N) ^{\left|M\right|}} \\
& \qquad \times \int_{C_+^K} {\int_{C_-^{L \cup M}}{J\left(z_K \cup - z_L\right)f\left(iz_1, \cdots,
iz_n\right)dz_1 \cdots d z_n}}
\end{split}
\end{align}
where $z_K = \left\{z_k: k \in K\right\}$ and $-z_L = \left\{-z_l: l \in L\right\}$ and
$\int_{C_+^K} {\int_{C_-^{L \cup M}}}$ means we are integrating all the variables in $z_K$ along the
$C_+$ path and all others down the $C_-$ path. The sum over $K \cup L \cup M$ is over disjoint unions and $J(A)$ defined as in (\ref{def:J}).

We can then deduce the following.
\begin{lemma} \label{thm:RMorthoJstar}
Let $n \leq N$
\begin{align}
\begin{split}
2^n &\int_{SO\left(2N\right)} {\sum_{j_1, \cdots, j_n =1}^N {f\left(\theta_{j_1}, \cdots,
\theta_{j_n}\right) dX} } \\
&= \frac{1}{\left(2 \pi i \right)^n} \sum_{K \cup L \cup M = \left\{1, \cdots,
n\right\}}{ (2N) ^{\left|M\right|}} \\
& \qquad \times \int_{C_+^K} {\int_{C_-^{L \cup M}}{J^{*}\left(z_K \cup - z_L\right)f\left(iz_1, \cdots,
iz_n\right)dz_1 \cdots d z_n}}
\end{split}
\end{align}
\end{lemma}
\begin{proof}
We know that $J(A) =J^{*}(A)$ when $|A| \leq N$ and $\Re(\alpha)>0$ $\forall \alpha \in A$ by Theorem \ref{thm:ortho2}. This condition is clearly met by $\left(z_K \cup - z_L\right)$.
\end{proof}

The next step is to move these contour integrals onto the imaginary axis. In order to do this, we first need some new notation.

For given $n$, $0 \leq R \leq n$, let
\begin{align} \label{def:sum notation}
\sideset{}{^{n,R}}\sum = \sum^{N}_{\substack{j_1, \cdots, j_n =1 \\ j_i \neq j_k \forall i,k>R}}.
\end{align}

For fixed sets $K,L,M$ such that $K \cup L \cup M = \{1, \cdots, n\}$ let $I^{n,R}_{f,K,L,M}$ be the integral in Lemma \ref{thm:RMorthoJstar} with $N-R$ of the integrals shifted onto the imaginary axis. All the integrals on the imaginary axis are principal value integrals.
\begin{align}
\begin{split}
I^{n,R}_{f,K,L,M}=& \int^{i\pi}_{-i\pi}{ \cdots \int^{i\pi}_{-i\pi}{ \int_{C^{K \cap \{1, \cdots, R\}}_+}{\int_{C^{(L \cup M) \cap \{1, \cdots, R\}}_-}{J^{*}(z_K \cup -z_L)}}}}\\
& \qquad \qquad \times f(iz_1, \cdots, iz_n) dz_1 \cdots dz_R d z_{R+1} \cdots dz_n.
\end{split}
\end{align}
We can express Lemma \ref{thm:RMorthoJstar} in the new notation.
\begin{align}
\begin{split}
(2\pi i)^n 2^n &\int_{SO(2N)}{\sideset{}{^{n,n}}\sum{f(\theta_{j_1}, \cdots, \theta_{j_n}) dX}}\\
 & = \sum_{K \cup L \cup M = \{1, \cdots, n\}}{(2N)^{|M|}I^{n,n}_{f,K,L,M}}
 \end{split}
\end{align}

We will be using Sokhotski-Plemelj Theorem \cite{Sokhot1873} \cite{plemelj1908erganzungssatz}, which states that we can shift a contour of integration onto a line passing through a pole of the integrand. We interpret the resulting integral as a principal value integral and gain half the residue of the pole. For further details see any standard text on complex analysis, or a section on functions of a complex variable in a text such as \cite{kn:arfken}.

We will now prove
\begin{theorem} \label{thm:orthomain}
With the notation defined above, $0\leq R \leq n$
\begin{align}
\begin{split}
(2 \pi i)^n 2^n &\int_{SO(2N)}{\sideset{}{^{n,R}}\sum{f(\theta_{j_1}, \cdots, \theta_{j_n}) dX}} \\
&= \sum_{K\cup L \cup M = \{1, \cdots, n\}}{(2N)^{|M|}I^{n,R}_{f,K,L,M}}. \label{statement}
\end{split}
\end{align}
\end{theorem}
\begin{proof}
We will prove this by induction, noting that we have already proved the case when $R=n$ for all values of $n$.

Firstly we will prove the base case when $n=1$. We have already proven it is true for $n=1$ and $R=1$, so we just need to show Equation \eqref{statement} holds for $n=1$ and $R=0$\\
\begin{align}
\begin{split}
 (2\pi i)2 & \int_{SO(2N)}{\sideset{}{^{1,0}}\sum{f(\theta_{j_1}) dX}}\\
 &=(2 \pi i)2 \int_{SO(2N)}{\sideset{}{^{1,1}}\sum{f(\theta_{j_1}) dX}}
 \end{split}\\
&= \sum_{K \cup L \cup M = \{1\}}{(2N)^{|M|}I^{1,1}_{f,K,L,M}}\\
&= \int_{C^+}{J^{*}(\theta)f(i\theta) d \theta}+ \int_{C^-}{J^{*}(-\theta)f(i\theta) d \theta}+ \int_{C^-}{2N f(i\theta)d \theta}
\end{align}
By Theorem \ref{thm:residuelemma} the only pole of $J^{*}(\theta)$ is at 0 and has residue 0.
\begin{align}
&= \int_{-i\pi}^{i \pi}{(J^{*}(\theta)+J^{*}(-\theta)+2N)f(i\theta) d \theta}\\
&=\sum_{K \cup L \cup M = \{1\}}{(2N)^{|M|}I^{1,0}_{f,K,L,M}}.\label{eq:n1R0}
\end{align}

We now move on to the inductive step. Assume Equation ~\eqref{statement} is true for $n=p-1$ and $0 \leq R \leq p-1$ and consider the case when $n=p$, $0 \leq R \leq p$.

We will proceed by induction on $R$. We have already proved that Equation~\eqref{statement} holds if $R=p$ so take that as the base case. Assume that Equation~\eqref{statement} holds if $n=p$, and $R \geq S$ so that
\begin{align}
\begin{split}
(2 \pi i)^p 2^p & \int_{SO(2N)}{\sideset{}{^{p,S}}\sum{f(\theta_{j_1}, \cdots, \theta_{j_p}) dX}} \\
&= \sum_{K \cup L \cup M = \{1, \cdots, p\}}{(2N)^{|M|}I^{p,S}_{f,K,L,M}}
\end{split}\\
\begin{split}
&=\sum_{K \cup L \cup M = \{1, \cdots, p\}}{(2N)^{|M|}}\\
 &\qquad \times \int^{i\pi}_{-i\pi}{ \cdots \int^{i\pi}_{-i\pi}{ \int_{C^{K \cap \{1, \cdots, S\}}_+}{\int_{C^{(L \cup M) \cap \{1, \cdots, S\}}_-}{J^{*}(z_K \cup -z_L)}}}}\\
 & \qquad \qquad \times f(iz_1, \cdots, iz_p) dz_1, \cdots, dz_S, d z_{S+1} \cdots dz_p. \label{induct}
 \end{split}
\end{align}

Consider moving the $z_S$ contour on the right hand side of this equation, from just off the imaginary axis onto the imaginary axis. It picks up residues from the variables whose contours have already been moved to the imaginary axis. Take $t>S$, then there are residues when $z_S = -z_t$ and $S,t \in K$  or $S,t \in L$ and also residues when $z_S = z_t$ if $t>S$ and $S \in L,t \in K$  or $S \in K,t \in L$.

Consider fixed $K,L,M$. If $S \in K$, then the residue of $J^{*}(z_K \cup -z_L)f(iz_1, \cdots, iz_p)$  at $z_S=z_t$ ($t\in L$) or $z_S=-z_t$ ($t\in K$) is
\begin{align}
\begin{split}
i \pi &\left(2NJ^{*}\left(z_{K^{'}}\cup -z_{L^{'}}\right)+J^{*}\left(z_{K^{'}}\cup -z_{L^{'} \cup \{t\}}\right)+J^{*}\left(z_{K^{'}\cup \{t\}}\cup -z_{L^{'} }\right)\right)\\
& \qquad \times f(iz_1, \cdots,iz_{S-1}, iz_t, iz_{S+1}, \cdots, iz_p) \label{eq:residue}
\end{split}
\end{align}
where $K^{'} = K \cap (A - \{S,t\}), L^{'} = L \cap (A - \{S,t\})$. The residue is multiplied by $i \pi$ rather than $2 \pi i$ because the $z_S$ contour is moving precisely onto the imaginary axis, so it only gives half the residue of a contour going completely around the pole by the Sokhotski-Plemelj Theorem \cite{Sokhot1873} \cite{plemelj1908erganzungssatz}. As $f$ is symmetric in all its variables, the residue is the same when $t \in L$ and $ t \in K$ as the residue is symmetric in $t$.

The other residues are when $S \in L$. Then $z_S$ appears in $J^{*}(z_K \cup -z_L)f(iz_1, \cdots, iz_p)$  with an additional minus sign and is on the $C_{-}$ contour and hence being integrated around the pole in a clockwise direction. These two minus signs cancel leaving the same residue as in Equation~\eqref{eq:residue}.

Returning to our calculation in Equation~\eqref{induct}, we note $S$ is fixed and we split the sum into three different cases: $S\in K$,$S \in L$ and $S \in M$. We then consider the residues that result when $t>S$ and $t \notin M$ as $t$ varies from $S+1$ to $p$.
\begin{align}
\begin{split}
&\sum_{K \cup L \cup M = \{1, \cdots, p\}}(2N)^{|M|}I^{p,S}_{f,K,L,M}\\
& = \sum_{K \cup L \cup M = \{1, \cdots, p\}}{(2N)^{|M|}}I^{p,S-1}_{f,K,L,M}\\
& \qquad + \sum_{t=S+1}^{p}{\sum_{\substack{K \cup L \cup M = \{1, \cdots, p\} \\ S \in K,t \not \in M}}}{(2N)^{|M|}}\\ 
& \qquad \qquad \int^{i\pi}_{-i\pi}{ \cdots \int^{i\pi}_{-i\pi}{ \int_{C^{K \cap \{1, \cdots, S-1\}}_+}{\int_{C^{(L \cup M) \cap \{1, \cdots, S-1\}}_-}{i \pi \left( 2NJ^{*}(z_{K^{'}}\cup -z_{L^{'}})\right.}}}}\\
& \qquad \qquad \qquad \qquad \left.+J^{*}(z_{K^{'}}\cup -z_{L^{'} \cup \{t\}})+J^{*}(z_{K^{'}\cup \{t\}}\cup -z_{L^{'} })\right)\\
& \qquad \qquad \qquad \times f(iz_1, \cdots,iz_{S-1}, iz_t, iz_{S+1}, \cdots, iz_p) dz_1, \cdots, dz_{S-1}, d z_{S+1} \cdots dz_p\\
& \qquad + \sum_{t=S+1}^{p}{\sum_{\substack{K \cup L \cup M = \{1, \cdots, p\} \\ S \in L,t \not \in M}}}{(2N)^{|M|}}\\ 
& \qquad \qquad \int^{i\pi}_{-i\pi}{ \cdots \int^{i\pi}_{-i\pi}{ \int_{C^{K \cap \{1, \cdots, S-1\}}_+}{\int_{C^{(L \cup M) \cap \{1, \cdots, S-1\}}_-}{i \pi \left( 2NJ^{*}(z_{K^{'}}\cup -z_{L^{'}})\right.}}}}\\
& \qquad \qquad \qquad \qquad \left.+J^{*}(z_{K^{'}}\cup -z_{L^{'} \cup \{t\}})+J^{*}(z_{K^{'}\cup \{t\}}\cup -z_{L^{'} })\right)\\
& \qquad \qquad \qquad \times f(iz_1, \cdots,iz_{S-1}, iz_t, iz_{S+1}, \cdots, iz_p) dz_1, \cdots, dz_{S-1}, d z_{S+1} \cdots dz_p
\end{split}
\end{align}

This expression can then be simplified to
\begin{align} \label{eq:repeatf}
\begin{split}
&\sum_{K \cup L \cup M = \{1, \cdots, p\}}{(2N)^{|M|}}I^{p,S-1}_{f,K,L,M}\\
& \qquad + 4 \pi i\sum_{t=S+1}^{p}{\sum_{K^{'}\cup L^{'}\cup M= \{1, \cdots, p\}-\{S,t\}}}{(2N)^{|M|}}\\
& \qquad \qquad \int^{i\pi}_{-i\pi}{ \cdots \int^{i\pi}_{-i\pi}{ \int_{C^{K^{'} \cap \{1, \cdots, S-1\}}_+}{\int_{C^{(L^{'} \cup M) \cap \{1, \cdots, S-1\}}_-}{ \left( 2NJ^{*}(z_{K^{'}}\cup -z_{L^{'}})\right.}}}}\\
& \qquad \qquad \qquad \qquad \left.+J^{*}(z_{K^{'}}\cup -z_{L^{'} \cup \{t\}})+J^{*}(z_{K^{'}\cup \{t\}}\cup -z_{L^{'} })\right)\\
& \qquad \qquad \qquad \times f(iz_1, \cdots,z_{S-1}, z_t, z_{S+1}, \cdots iz_p) dz_1, \cdots, dz_{S-1}, d z_{S+1} \cdots dz_p
\end{split}
\end{align}
where the factor of four comes from the fact that there are 4 different combinations of $S,t \notin M$ (i.e. $S \in L$ and $t \in K$, $S \in L$ and $t \in L$ etc). All of the unions over sets are disjoint unions.

Notice that  $f(z_1, \cdots,z_{S-1},z_t, z_{S+1}, \cdots, z_t, \cdots z_p)$ has a repeated variable. It is one of the terms we want to remove from \ref{nostar} so that we have a sum over distinct indices. With this in mind we relabel the variables $z_1, \cdots,z_{S-1}, z_{S+1}, \cdots z_p$ as $\widehat{z}_1, \cdots, \widehat{z}_{p-1}$.  We define a function
\begin{align}
g_{t,S}(\widehat{z}_1, \cdots, \widehat{z}_{p-1}) &= f(\widehat{z}_1, \cdots,\widehat{z}_{S-1},\widehat{z}_{t-1}, \widehat{z}_{S}, \cdots, \widehat{z}_{t-1}, \cdots ,\widehat{z}_{p-1})\\
 &= f(z_1, \cdots,z_{S-1},z_t, z_{S+1}, \cdots, z_t, \cdots z_p).\label{eq:g}
\end{align}
It is obvious that $g_{t,S}$ is the function $f$ with $z_S = z_t$, for some $t>S$.

We also note that for functions $h$ of sets $K,L$ and $M$
\begin{align}
\begin{split}
&\sum_{K \cup L \cup M=\{1,\cdots, m-1\}}{h(K\cup \{m\},L,M)+h(K,L\cup \{m\},M)+h(K,L,M\cup \{m\})}\\
& \qquad \qquad =\sum_{\mathcal{K} \cup \mathcal{L} \cup \mathcal{M}=\{1,\cdots, m\}}{h(\mathcal{K},\mathcal{L},\mathcal{M})}.
\end{split}
\end{align}
This allows us to rewrite Equation~\eqref{eq:repeatf}
\begin{align}
\begin{split}
&\sum_{K \cup L \cup M = \{1, \cdots, p\}}(2N)^{|M|}I^{p,S}_{f,K,L,M}\\
& = \sum_{K \cup L \cup M = \{1, \cdots, p\}}{(2N)^{|M|}}I^{p,S-1}_{f,K,L,M}\\
& \qquad + 4 \pi i\sum_{t=S+1}^{p}{\sum_{K \cup L \cup M = \{1, \cdots, p-1\}}}{(2N)^{|M|}I^{p-1,S-1}_{g_{t,S},K,L,M}}.
\end{split}
\end{align}

By our inductive hypothesis on $n$, Equation ~\eqref{induct} holds for $n=p-1$ and $R=S-1$.
\begin{align}
\begin{split}
&\sum_{K \cup L \cup M = \{1, \cdots, p\}}(2N)^{|M|}I^{p,S}_{f,K,L,M}\\
& \qquad = \sum_{K \cup L \cup M = \{1, \cdots, p\}}{(2N)^{|M|}}I^{p,S-1}_{f,K,L,M}\\
& \qquad \qquad + 4 \pi i\sum_{t=S+1}^{p}{(2 \pi i)^{p-1}2^{p-1} \int_{SO(2N)}{\sideset{}{^{p-1,S-1}}\sum {g_{t,S}(\theta_{j_1}, \cdots,\theta_{j_{p-1}} ) dX_{SO(2N)}}}}. \label{four1}
\end{split}
\end{align}
It is clear from the definition of $\sum^{n,R}$ in Equation ~\eqref{def:sum notation} that
\begin{align}
\sideset{}{^{n,R}}\sum{f(\theta_1, \cdots, \theta_n)} = \sideset{}{^{n,R-1}}\sum{f(\theta_1, \cdots, \theta_n)}+ \sum_{t=R+1}^{n}{\sideset{}{^{n-1,R-1}}\sum{g_{t,R}(\theta_1, \cdots, \theta_n)}}.
\end{align}

Considering now the left hand side of Equation~\eqref{induct}, we see that
\begin{align}
\begin{split}
(2 \pi i)^p 2^p & \int_{SO(2N)}{\sideset{}{^{p,S}}\sum{f(\theta_{j_1}, \cdots, \theta_{j_p}) dX_{SO(2N)}}}\\
&= (2 \pi i)^p 2^p  \int_{SO(2N)}{\sideset{}{^{p,S-1}}\sum{f(\theta_1, \cdots, \theta_p) dX_{SO(2N)}}} \\
& \qquad +  (2 \pi i)^p 2^p  \int_{SO(2N)}{\sum_{t=S+1}^{p}{\sideset{}{^{p-1,S-1}}\sum{g_{t,S}(\theta_1, \cdots, \theta_{p-1}) dX_{SO(2N)}}}}. \label{four2}
\end{split}
\end{align}
Equating Equations ~\eqref{four1} and \eqref{four2}, we see that
\begin{align}
\begin{split}
&\sum_{K \cup L \cup M = \{1, \cdots, p\}}{(2N)^{|M|}}I^{p,S-1}_{f,K,L,M}\\
&\qquad = (2 \pi i)^p 2^p  \int_{SO(2N)}{\sideset{}{^{p,S-1}}\sum{f(\theta_1, \cdots, \theta_p) dX_{SO(2N)}}}.
\end{split}
\end{align}

Thus we have completed the inductive step in $R$ by showing that Equation~\eqref{statement} holds for $n=p, R=S-1$ if it holds for $n=p, R=S$. As we have already shown Equation~\eqref{statement} holds for $n=p, R=p$, we have now shown that it holds for all $R$ when $n=p$.

This completes the inductive step in $n$, as we have shown that Equation ~\eqref{statement} holds for all $0 \leq R \leq n$ when $n=p$ if it holds for all $0 \leq R \leq n$ and $n=p-1$. We already proved that it holds for all $0 \leq R \leq 1$ when $n=1$, so by induction Equation~\eqref{statement} is true for all $n$ and $0 \leq R \leq n$.
\end{proof}

Using Theorem \ref{thm:orthomain} we can prove our main theorem on $n$-level density for orthogonal random matrices
\begin{theorem} [$n$-level Density of Orthogonal Matrices]\label{thm:orthoNcorr2}
\begin{align}
\begin{split}
&\int_{SO(2N)}{\sideset{}{^{n,0}}\sum f(\theta_{j_1}, \cdots, \theta_{j_n}) dX}\\
&= \frac{1}{(2 \pi i)^n}\sum_{K \cup L \cup M= \{1, \cdots, n\}}{(2N)^{|M|}\left(\int_{0}^{\pi}\right)^n J^{*}(-iz_K \cup i z_L)}\\
& \qquad \qquad \times f(z_1, \cdots, z_n)dz_1 \cdots dz_n,
\end{split}
\end{align}
where $J^{*}(-iz_K\cup i z_L)$ is defined in (\ref{def:Jstar}) and the sum notation is defined at (\ref{def:sum notation}) and indicates that the sum is over distinct indices.
\end{theorem}
\begin{proof}
We have proven that
\begin{align}
\begin{split}
(2 \pi i)^n 2^n &\int_{SO(2N)}{\sideset{}{^{n,0}}\sum{f(\theta_{j_1}, \cdots, \theta_{j_n}) dX}} \\
&= \sum_{K \cup L \cup M = \{1, \cdots, n\}}{(2N)^{|M|}I^{n,0}_{f,K,L,M}}
\end{split}
\\
\begin{split}
&= \sum_{K \cup L \cup M = \{1, \cdots, n\}}{(2N)^{|M|}}\\
& \qquad \qquad \times \left(\int^{-i \pi}_{i \pi}\right)^n J^{*}(z_K \cup -z_L)f(iz_1, \cdots, iz_n)
\end{split}
\end{align}
where the integral here is a principal value integral. Changing variable to $\theta_1=iz_1$ etc
\begin{align}
\begin{split}
(2 \pi i)^n 2^n &\int_{SO(2N)}{\sideset{}{^{n,0}}\sum{f(\theta_{j_1}, \cdots, \theta_{j_n}) dX}} \\
&= \sum_{K \cup L \cup M = \{1, \cdots, n\}}(2N)^{|M|}\left(\int^{\pi}_{-\pi}\right)^n J^{*}(-iz_K \cup iz_L)f(z_1, \cdots, z_n)
\end{split}\\
&= \left(\int^{\pi}_{-\pi}\right)^n \sum_{K \cup L \cup M = \{1, \cdots, n\}}(2N)^{|M|} J^{*}(-iz_K \cup iz_L)f(z_1, \cdots, z_n).
\end{align}
Now
\begin{align}
\begin{split}
&\sum_{K \cup L \cup M = \{1, \cdots, n\}}(2N)^{|M|} J^{*}(-iz_K \cup iz_L) \\
& \qquad = \sum_{K \cup L \cup M = \{1, \cdots, n\}}(2N)^{|M|} J^{*}(iz_K \cup -iz_L)
\end{split}
\end{align}
 and $f\left(\theta_{j_1}, \cdots, \theta_{j_n}\right) = f\left(\pm \theta_{j_1}, \cdots, \pm
\theta_{j_n}\right)$ so each integral from $- \pi$ to $\pi$ is double the integral from $0$ to $\pi$. We change into this form to allow easier comparison with other  correlation functions in other contexts.
\begin{align}
\begin{split}
(2 \pi )^n 2^n &\int_{SO(2N)}{\sideset{}{^{n,0}}\sum{f(\theta_{j_1}, \cdots, \theta_{j_n}) dX_{SO(2N)}}} \\
&= 2^n \left(\int^{\pi}_{0}\right)^n \sum_{K \cup L \cup M = \{1, \cdots, n\}}(2N)^{|M|} J^{*}(-iz_K \cup iz_L)\\
& \qquad \times f(z_1, \cdots, z_n) d z_1 \cdots d z_n.
\end{split}
\end{align}

All that is now required is to show that there are no poles on the path of integration. As $f$ is holomorphic, we just need to check that
\begin{align}
\sum_{K \cup L \cup M = \{1, \cdots, n\}}(2N)^{|M|} J^{*}(-iz_K \cup iz_L)
\end{align}
has no poles at $z_1 =z_2$ (this argument applies equally well to any other pair of $z$'s by symmetry). We do not need to check for a pole at $z_1 = -z_2$ as $z_i \geq 0$, for $ 1 \leq i \leq n$ on the path of integration. A given $J^{*}(-iz_K \cup iz_L)$ has a simple pole at $z_1 =z_2$  if $1 \in K, 2 \in L$ or $1 \in L, 2 \in K$. So
\begin{align}
\begin{split}
 \Res_{z_1 = z_2}&\left(\sum_{K \cup L \cup M = \{1, \cdots, n\}}(2N)^{|M|} J^{*}(-iz_K \cup iz_L)\right)\\
& = \sum_{\mathcal{K} \cup \mathcal{L} \cup \mathcal{M} = \{3, \cdots, n\}}(2N)^{|\mathcal{M}|} \Res_{z_1 = z_2}\left(J^{*}\left(-iz_{\mathcal{K} \cup \{z_1\}} \cup iz_{\mathcal{L} \cup \{z_2\}}\right) \vphantom{\frac{J^t}{J^t}}\right.\\
& \qquad \qquad \qquad \left. \vphantom{\frac{J^t}{J^t}} + J^{*}\left(-iz_{\mathcal{K} \cup \{z_2\}} \cup iz_{\mathcal{L} \cup \{z_1\}}\right)\right)
\end{split}\\
&=0
\end{align}
as $Res_{x=y}f(x,y)=-Res_{x=y} f(-x,-y)$. Therefore if $\left(\sum_{K \cup L \cup M = \{1, \cdots, n\}}(2N)^{|M|} J^{*}(-iz_K \cup iz_L)\right)$ has a singular set, it has complex dimension less than $n-1$. However, this implies the singular set is trivial (see Corollary 7.3.2 in \cite{krantz1982function}).

\end{proof}

\section{Eigenvalue Statistics of Symplectic Matrices}\label{sub:symp}
We also want to look at families of $L$-functions with symplectic symmetry, so we will repeat the calculations we have done in Section \ref{sec:ortho} with $USp(2N)$ .

Let $Y=(y_{jk})$ be a $2N \times 2N$ matrix and define the transpose matrix  $Y^{t} = (y_{kj})$. $Y$ is a symplectic matrix if $YZY^{t}=Z$ where
\begin{align}
Z = \begin{pmatrix}
0 & I_N\\
-I_N &0
\end{pmatrix}
\end{align}
and $I_{N}$ is the identity matrix of size $N$. The group of symplectic matrices $USp(2N)$ is a subgroup of $U(2N)$ the group of unitary matrices, just like $SO(2N)$, . All the eigenvalues of a matrix $X \in USp(2N)$ have absolute value 1 and can be written as $e^{\pm i\theta_1 }, \cdots, e^{\pm i\theta_N}$ with $0\leq \theta_1, \cdots, \theta_N \leq \pi $.
The result we are proving in this section is:
\begin{theorem}\label{thm:Uspcorr1} The $n$-level density for eigenvalues of matrices from $USp(2N)$ can be written as
\begin{align}
\begin{split}
&2^n\int_{USp(2N)}\sum^{N}_{\substack{j_1, \cdots, j_n =1 \\ j_i \neq j_k \forall i,k}} f(\theta_{j_1}, \cdots, \theta_{j_n}) dX\\
&= \frac{1}{(2 \pi i)^n}\sum_{\substack{K \cup L \cup M= \\ \{1 \cdots, n\}}}{(2N)^{|M|}\left(\int_{0}^{\pi}\right)^n J_{USp(2N)}^{*}(-iz_K \cup i z_L)}\\
& \qquad \qquad \times f(z_1, \cdots, z_n)dz_1 \cdots dz_n,
\end{split}
\end{align}
where $J_{USp(2N)}^{*}(-iz_K \cup i z_L)$ is defined at (\ref{eq:JUSpstar}) and the sum in the right hand side involving $K$, $L$ and $M$ is over disjoint subsets of $\{1,\ldots,n\}$.
\end{theorem}
\subsection{Integral Theorem}  In this section we will proceed exactly analogously to the $SO(2N)$ case and express the $n$-level density of eigenvalues of matrices from $USp(2N)$ as contour integrals on the complex plane.   We will perform exactly the same manipulations to move these contours onto the imaginary axis in order to arrive at Theorem \ref{thm:Uspcorr1}.
\begin{theorem}[Integral Theorem for Symplectic Matrices]\label{thm:SympIntegral}
Let $C_{-}$ denote the path from $ - \delta - \pi i$ up to $- \delta + \pi i$ and let $C_{+}$
denote the path from $  \delta - \pi i$ up to $ \delta + \pi i$. Let $f$ be a $2\pi$-periodic,
holomorphic function of $n$ variables such that
\begin{align}
f\left(\theta_{j_1}, \cdots, \theta_{j_n}\right) = f\left(\pm \theta_{j_1}, \cdots, \pm
\theta_{j_n}\right).
\end{align}

Then
\begin{align}
\begin{split}
2^n &\int_{USp\left(2N\right)} {\sum_{j_1, \cdots, j_n =1}^N {f\left(\theta_{j_1}, \cdots,
\theta_{j_n}\right) dX_{USp(2N)}} } \\
&= \frac{1}{\left(2 \pi i \right)^n} \sum_{K \cup L \cup M = \left\{1, \cdots,
n\right\}}{(2N) ^{\left|M\right|}} \\
& \qquad \times \int_{C_+^K} {\int_{C_-^{L \cup M}}{J_{USp(2N)}\left(z_K \cup - z_L\right)f\left(iz_1, \cdots,
iz_n\right)dz_1 \cdots d z_n}}
\end{split}
\end{align}
where $z_K = \left\{z_k: k \in K\right\}$ and $-z_L = \left\{-z_l: l \in L\right\}$,
$\int_{C_+^K} {\int_{C_-^{L \cup M}}}$ means we are integrating all the variables in $z_K$ along the
$C_+$ path and all others up the $C_-$ path and \\
\begin{align} \label{def:Jsymp}
J_{USP(2N)}\left(A\right) = \int_{USp(2N)}{\prod_{\alpha \in
A}{(-e^{-\alpha})\frac{\Lambda_X^{'}}{\Lambda_X}(e^{-\alpha})} dX_{USp(2N)}}.
\end{align}

Here $K,L,M$ are finite sets of integers and $A$ is a finite set of complex numbers and $dX_{USp(2N)}$ indicates integration with respect to the Haar measure.
\end{theorem}
The proof is identical to the proof for Theorem~\ref{thm:RMortho1}

\subsection{Ratios Theorem } \label{sub:Dirratio}

Here we rewrite the theorem of Conrey, Forrester and Snaith \cite{conrey2005ar} in set notation.
 Take finite sets $A$ and $B$, where $N \geq |B|$ and consider
\begin{align} \label{def:sympR}
R_{USp(2N)}(A;B) &= \int_{USp(2N)}{\frac{\prod_{\alpha \in A}{\Lambda_{X}(e^{-\alpha})}}{\prod_{\beta \in
B}{\Lambda_{X}(e^{- \beta})}}dX}\\
\begin{split}
 &= \sum_{D \subseteq A}{e^{-2N\sum_{\delta \in D}{\delta}}\sqrt{\frac{Y(A \backslash D)Y(D^-)}{Y(B)}}}\\
 & \qquad \times \sqrt{\frac{Z(D^- \cup (A \backslash D), D^- \cup (A \backslash D))Z(B,B)}{Z(D^- \cup (A \backslash D),B)^2}} \label{eq:symp sr}
 \end{split}
\end{align}
where $D^{-}=\{-\alpha: \alpha \in D\}$ and
\begin{align}
z(x) &= \frac{1}{1-e^{-x}}\\
Y(A)&= \prod_{\alpha \in A}{z(2\alpha)}\\
Z(A,B) &= \prod_{\substack{\alpha \in A \\ \beta \in B}}{z(\alpha + \beta)}.
\end{align}
Note that the main difference is the factor $\sqrt{\frac{Y(A \backslash D) Y(D^-)}{Y(B)}}$ compared with $\sqrt{\frac{Y(B)}{Y(A \backslash D) Y(D^-)}}$ in the orthogonal case.

As in the orthogonal case, it is clear that $J_{USp(2N)}= \left.\prod_{\alpha \in A}\frac{d}{d\alpha} R_{USp(2N)} \right|_{B=A}$. So the next step is differentiating $R_{USp(2N)}(A;B)$. \\

\begin{theorem}\label{thm:symp2}
Let $A$ be a finite set of complex numbers where $\Re(\alpha)>0$ for $\alpha \in A$ and $|A| \leq N$, then
$J_{USp(2N)}(A) = J_{USp(2N)}^{*}(A)$ where
\begin{align}
J_{USp(2N)}(A) &= \int\limits_{USp(2N)}{\prod_{\alpha \in A}{(-e^{-\alpha})\frac{\Lambda_X^{'}}{\Lambda_X}(e^{-\alpha})} dX_{USp(2N)}}\\
\begin{split}\label{eq:JUSpstar}
J_{USp(2N)}^{*}(A) &= \sum_{D\subseteq A}{e^{-2N \sum\limits_{\delta \in D}{\delta}}(-1)^{|D|} \sqrt{\frac{Z(D,D)Z(D^{-},D^{-})Y(D^{-})}{Y(D)Z^{\dag}(D^{-},D)^2}}}\\
& \qquad \times \sum_{\substack{A \backslash D = W_1 \cup \cdots \cup W_R \\ |W_r| \leq 2}}{\prod_{r=1}^{R}{H_D(W_r)}}
\end{split}
\end{align}
where the sum over $A \backslash D$ is a sum over all set partitions and
\begin{align}
H_D(W) =  \begin{cases}
\left( \sum\limits_{\delta \in D}{\frac{z^{'}(\alpha -\delta)}{z(\alpha - \delta)}-\frac{z^{'}(\alpha +\delta)}{z\left( \alpha + \delta\right)}}\right) + \frac{z^{'}}{z}(2\alpha) & W = \{\alpha\} \subset A \backslash D\\
\left(\frac{z^{'}}{z}\right)'(\alpha + \widehat{\alpha}) & W= \{\alpha, \widehat{\alpha}\} \subset A \backslash D\\
1 & W= \emptyset
    \end{cases}
\end{align}
and
\begin{align}
z(x) &= \frac{1}{1-e^{-x}}\\
Y(A)&= \prod_{\alpha \in A}{z(2\alpha)}\\
Z(A,B) &= \prod_{\substack{\alpha \in A \\\beta \in B}}{z(\alpha + \beta)}.
\end{align}
The dagger adds a
restriction that a factor $z(x)$ is omitted if its argument is zero.
\end{theorem}
The proof is as in the orthogonal case.

\subsection{Residue Theorem}
We need to know what the residues are at the poles of $J_{USP(2N)}^{*}(A)$. It is clear that the only possible poles are when $\alpha^{*} = - \beta$ for some other $\beta \in A$, or when $\alpha^{*} = 0$.
\begin{theorem}[Residue Theorem for Symplectic Matrices]\label{thm:sympRes}
Let A be a finite set of complex numbers, let $\alpha^{*},\beta^{*} \in A$ and $A^{'}=A-\{\alpha^{*},\beta^{*}\}$ and let $J_{USp(2N)}^{*}(A)$ as defined in Theorem~\ref{thm:symp2}. Then
\begin{align}
\begin{split}
\Res_{\alpha^{*} =-\beta^{*}} \left(J_{USp(2N)}^{*}(A)\right) &= J_{USp(2N)}^{*}(A^{'} \cup \{\beta^{*}\}) + J_{USp(2N)}^{*}(A^{'} \cup \{-\beta^{*}\})\\
& \qquad +2N J_{USp(2N)}^{*}(A^{'})
\end{split}\\
\Res_{\alpha^{*} =0} \left(J_{USp(2N)}^{*}(A)\right) &= 0.
\end{align}
\end{theorem}
Note that this residue formula is identical to that for $J^{*}(A)$ in the orthogonal case.
\begin{proof}
The proof is very similar to that of Theorem~\ref{thm:residuelemma}. It follows by showing that the properties {\bf P1} to {\bf P4} and {\bf R1} and {\bf R2} hold and then applying Lemma \ref{lem:P} and Lemma \ref{lem:R}.\end{proof}

\subsection{$n$-level Density of Symplectic Matrices}
Using Theorem \ref{thm:symp2}, we can deduce
\begin{lemma}\label{thm:RMsympJstar}
Let $n \leq N$.  Then
\begin{align}
\begin{split}
2^n &\int_{USp\left(2N\right)} {\sum_{j_1, \cdots, j_n =1}^N {f\left(\theta_{j_1}, \cdots,
\theta_{j_n}\right) dX} } \\
&= \frac{1}{\left(2 \pi i \right)^n} \sum_{K \cup L \cup M = \left\{1, \cdots,
n\right\}}{ (2N) ^{\left|M\right|}} \\
& \qquad \times \int_{C_+^K} {\int_{C_-^{L \cup M}}{J_{USp(2N)}^{*}\left(z_K \cup - z_L\right)f\left(iz_1, \cdots,
iz_n\right)dz_1 \cdots d z_n}}.
\end{split}
\end{align}
\end{lemma}
\begin{proof}
We know that $J_{USp(2N)}(A) =J_{USp(2N)}^{*}(A)$ when $|A| \leq N$ and $\Re(\alpha)>0$ $\forall \alpha \in A$ by Theorem \ref{thm:symp2}. This condition is clearly met by $\left(z_K \cup - z_L\right)$.
\end{proof}
Clearly as $J^{*}_{USp(2N)}$ has the same poles as $J^{*}(A)$ and an identical recursive formula for the residue, we can immediately deduce the following theorem.
\begin{theorem} \label{thm:sympmain}
Let $0\leq R \leq n$. Then
\begin{align}
\begin{split}
(2 \pi i)^n 2^n &\int_{USp(2N)}{\sideset{}{^{n,R}}\sum{f(\theta_{j_1}, \cdots, \theta_{j_n}) dX}} \\
&= \sum_{K \cup L \cup M = \{1, \cdots, n\}}{(2N)^{|M|}\widehat{I^{n,R}_{f,K,L,M}}}
\end{split}
\end{align}
where
\begin{align}
\sideset{}{^{n,R}}\sum = \sum^{n}_{\substack{j_1, \cdots, j_n =1 \\ j_i \neq j_k \forall i,k>R}}
\end{align}
and \begin{align}
\begin{split}
\widehat{I^{n,R}_{f,K,L,M}}=& \int^{i\pi}_{-i\pi}{ \cdots \int^{i\pi}_{-i\pi}{ \int_{C^{K \cap \{1, \cdots, R\}}_+}{\int_{C^{(L \cup M) \cap \{1, \cdots, R\}}_-}{J_{USp(2N)}^{*}(z_K \cup -z_L)}}}}\\
& \qquad \qquad \times f(iz_1, \cdots, iz_n) dz_1 \cdots dz_R d z_{R+1} \cdots dz_n.
\end{split}
\end{align}
for fixed sets $K,L,M$ such that the disjoint $K \cup L \cup M = \{1, \cdots, n\}$.
Note: this is the integral in Theorem~\ref{thm:RMsympJstar} with $N-R$ of the integrals shifted onto the imaginary axis and all the integrals on the imaginary axis are principal value integrals.
\end{theorem}

The proof for this in the orthogonal case relies solely on the formula for the residue, so this result follows immediately from the proof for Theorem~\ref{thm:orthomain}.

We can now state the $n$-level density for symplectic random matrices.
\begin{theorem}[$n$-level Density of Symplectic Random Matrices]\label{thm:Uspcorr}
\begin{align}
\begin{split}
&2^n\int_{USp(2N)}{\sideset{}{^{n,0}}\sum f(\theta_{j_1}, \cdots, \theta_{j_n}) dX}\\
&= \frac{1}{(2 \pi i)^n}\sum_{\substack{K \cup L \cup M= \\ \{1 \cdots, n\}}}{(2N)^{|M|}\left(\int_{0}^{\pi}\right)^n J_{USp(2N)}^{*}(-iz_K \cup i z_L)}\\
& \qquad \qquad \times f(z_1, \cdots, z_n)dz_1 \cdots dz_n,
\end{split}
\end{align}
where $J_{USp(2N)}^{*}(-iz_K \cup i z_L)$ is defined at (\ref{eq:JUSpstar}) and the sum notation is defined at (\ref{def:sum notation}) and indicates that the sum is over distinct indices.
\end{theorem}
\begin{proof}
We follow exactly the proof for the equivalent theorem in the orthogonal case - Theorem~\ref{thm:orthoNcorr2}. Taking the result of Theorem~\ref{thm:sympmain}, we use a change of variable to deduce
\begin{align}
\begin{split}
(2 \pi i)^n 2^n &\int_{USp(2N)}{\sideset{}{^{n,0}}\sum{f(\theta_{j_1}, \cdots, \theta_{j_n}) dX}} \\
&= \left(\int^{\pi}_{-\pi}\right)^n \sum_{K \cup L \cup M = \{1, \cdots, n\}}(2N)^{|M|} J_{USp(2N)}^{*}(-iz_K \cup iz_L)f(z_1, \cdots, z_n).
\end{split}
\end{align}
Noting that
\begin{align}
\begin{split}
&\sum_{K \cup L \cup M = \{1, \cdots, n\}}(2N)^{|M|} J_{USp(2N)}^{*}(-iz_K \cup iz_L) \\
& \qquad = \sum_{K \cup L \cup M = \{1, \cdots, n\}}(2N)^{|M|} J_{USp(2N)}^{*}(iz_K \cup -iz_L)
\end{split}
\end{align}
 and $f\left(\theta_{j_1}, \cdots, \theta_{j_n}\right) = f\left(\pm \theta_{j_1}, \cdots, \pm
\theta_{j_n}\right)$ gives the required result.  Following precisely the same steps as in the orthogonal case, it can be shown that there are no poles on the path of integration.
\end{proof}

\section{$L$-functions}\label{sec:LF1}

\subsection{$L$-functions: a definition}

We begin by defining what we mean by $L$-function. We will use a definition which is distinct from the ``Selberg class", but is conjectured to be equivalent to it \cite{farmer2007modeling}.
\begin{definition}\label{LFcond}
Let $s \in \mathbb{C}$. An \textbf{$L$-function} is a Dirichlet series
\begin{align}
L(s)= \sum_{n=1}^{\infty} {\frac{a_n}{n^s}}
\end{align}
with $a_n= O(n^{\epsilon})$ for every $\epsilon >0$ (Ramanujan Hypothesis \cite{conrey1993selberg}) subject to three conditions: it has an Euler product, an analytic continuation and a functional equation.  We will define these properties similarly to the definitions used by  Farmer in \cite{farmer2007modeling} and Conrey et al. in \cite{conrey2005integra}.
\begin{enumerate}
\item The \emph{Euler product} is a product over primes, equal to the Dirichlet series defining the $L$-function. Formally, for $\Re(s) > 1$ we have
\begin{align}
L(s) = \prod_{p}{\prod_{j=1}^{w}{(1-\hat{\gamma}_{p,j}p^{-s})^{-1}}},
\end{align}
where the product is over the primes $p$, and each $|\hat{\gamma}_{p,j}|$ equals $1$ or $0$. $w$ is called the degree of the $L$-function.
\item
The \emph{functional equation} is an equation relating the $L$-function in one half of the complex plane to the other~\cite{kn:keasna03}. There exists $\epsilon$, such that $|\epsilon| = 1$, and a function $\Gamma_{L}(s)$ of the form
\begin{align}
\Gamma_{L}(s) = P(s) Q^{s}\prod_{j=1}^{w}{\Gamma(\frac{s}{2}+ \mu_{j})},
\end{align}
where $Q>0, \Re(\mu_j)\geq 0$ and $P$ is a polynomial whose only zeroes in $Re(s) >0$ coincide with the poles of $L(s)$ such that
\begin{align}
\xi_{L}(s) = \Gamma_{L}(s)L(s)
\end{align}
is entire and
\begin{align}
\xi_L(s) = \epsilon \overline{\xi_L(1-s)}.
\end{align}
\item
The \emph{analytic continuation } property means that $L(s)$ continues to a meromorphic function of finite order with at most finitely many poles.
\end{enumerate}
\end{definition}

It is worth noting that the Riemann Zeta function is an $L$-function with functional equation, analytic continuation and Euler product.
\subsection{Elliptic Curve $L$-functions}\label{Ellipticderivation}

An elliptic curve, $E$, over a field $K$ is a non-singular curve of genus one, defined over $K$, with a $K$-rational point, $O \in E$. $E$ has a algebraic group structure with unit element $O$. For more on elliptic curves see \cite{washington2008ellipti}, \cite{koblitz1993introductio} \cite{silverman2009arithmeti}.

$E$ can be embedded in $\mathbf{P}^2$ as a cubic curve given by a global minimum Weierstrass equation\cite{kn:young}
\begin{align}
E: y^2 + a_1xy + a_3 y = x^3 + a_2x^2 +a_4 x+ a_6
\end{align}
with coefficients $a_1, \cdots, a_6 \in K$.
Let $E$ be an elliptic curve over $\mathbf{Q}$, then the Weierstrass equation can be re-written as
\begin{align}
E: y^2 = x^3 + Ax +B
\end{align}
with $A,B \in \mathbf{Z}$ and discriminant $\Delta = -(A^3 +27B^2)$.
For prime $p$, we can look at the curve $y^2 = x^3 + Ax +B \mod p$. If $p \nmid \Delta$, then this is an elliptic curve over $\mathbf{F}_p$ and we say  $E$ has good reduction mod p.
\begin{align}
E_p: y^2 = x^3 + Ax +B \text{ mod } p
\end{align}
is a curve over the finite field $\mathbf{F}_p$. If $p | \Delta$, it is a prime of bad reduction. There are two types: multiplicative reduction if $E/\mathbf{F}_p$ has a node, and additive reduction if $E/\mathbf{F}_p$ has a cusp. Multiplicative reduction is said to be split if the slopes of the tangent lines at the node are rational, and non-split otherwise.\\

We define the conductor, M, of the elliptic curve as
\begin{align}
M = \prod_{p}{p^{f_p}}
\end{align}
where
\begin{align}
f_p = \begin{cases}
0 & \text{if $E$ has good reduction at $p$},\\
1 & \text{if $E$ has multiplicative reduction at p},\\
2 & \text{if $E$ has additive reduction at $p$, $p \neq 2,3$,}\\
2 + \delta_p & \text{if $E$ has additive reduction at $p=2$ or $3$.}\\
\end{cases}
\end{align}
Here $\delta_p$ depends on wild ramification in the action of the inertia group at $p$ of $Gal(\frac{\overline{\mathbf{Q}}}{\mathbf{Q}})$ on the Tate module $ T_p(E)$. \cite{silverman2009arithmeti} We will consider only prime $M$.\\

Then for $p$ of good reduction define the integer $a_p$ by
\begin{align}
|E(\mathbf{F}_p)| = p + 1 - a_p,
\end{align}
where $|E(\mathbf{F}_p)|$ is the number of points of $E_p$ in a finite field of cardinality $p$, including the point at infinity.
We define $a_p$ for primes of bad reduction by
\begin{align}
a_p = \begin{cases}
0 & \text{E has additive reduction at p,}\\
1 & \text{E has split multiplicative reduction at p,}\\
-1 & \text{E has non-split multiplication reduction at p.}
\end{cases}
\end{align}

We would like to use these $a_p$ to define an $L$-function in such a way that there is a $L$-function associated with each elliptic curve $E$.

The Hasse-Weil $L$-function \cite{silverman2009arithmeti} of $E/\mathbb{Q}$ is defined as
\begin{align}
L(E,s) = \prod_{ p| \Delta}{{\left(1- \frac{a_p}{p^{s}}\right)}^{-1}} \prod_{p\nmid \Delta}{{\left(1- \frac{a_p}{p^s} + p^{1-2s}\right)}^{-1}}
\end{align}
where $\Delta = -(A^3 +27B^2)$. We define $a_n$ for all $n$ by expanding this into a Dirichlet series.
\begin{align}
L(E,s) = \sum_{n=1}^{\infty}{\frac{a_n}{n^{s}}}.
\end{align}

Hasse proved this was holomorphic for $Re(s) > \frac{3}{2}$ by showing $|a_p| \leq 2 \sqrt{p}$ (Riemann Hypothesis for finite fields, proved by Hasse for elliptic curves in 1931 \cite{chahal2008riemann}). The Hasse-Weil conjecture says that $L(E,s)$ extends to a entire function. Eichler \cite{eichler1954quaternar} and
Shimura \cite{shimura1994introductio} proved that for elliptic curves $E/Q$ that have a modular parametrization, $L(E,s)$ can be extended to a entire function satisfying a functional equation
\begin{align}
\Lambda(E,s) = \omega_E \Lambda(E,2-s)
\end{align}
where $\omega_E = \pm 1$ and
\begin{align}
\Lambda(E,s) = \left(\frac{\sqrt{M}}{2 \pi}\right)^{s}\Gamma(s)L(E,s).
\end{align}

The Shimura-Taniyama-Weil conjecture says there is a cusp form whose $L$-function is equal to $L(E,s)$, and implies the Hasse-Weil conjecture. Wiles \cite{wiles1995modula} and Taylor \cite{taylor1995rin} showed that this is true for all elliptic curves with square-free conductor. \\

In 2001, Breuil, Conrad, Diamond and Taylor \cite{breuil2001modularit} then showed this was true for all elliptic curves over $\mathbb{Q}$.\\
\begin{theorem}[Iwaniec and Kowalski~\cite{iwaniec2004analytic}]
Let $E/\mathbb{Q}$ be any elliptic curve of conductor M. There exists a primitive cusp form
\begin{align}
f(z) = \sum_{1}^{\infty}{\lambda (n) n^{\frac{1}{2}}e(nz)} \in S_2(\Gamma_0(N))
\end{align}
such that $\lambda(n) = \frac{a_n}{n^{\frac{1}{2}}}$ giving
\begin{align}
L(E, s+ \frac{1}{2}) = L_E(s) = \sum_{1}^{\infty}{\frac{\lambda(n)}{n^{-s}}}.
\end{align}
\end{theorem}

From this it is clear that the Hasse-Weil $L$-function, $L(E,s)$ meets all our conditions for an $L$-function in Definition~\ref{LFcond} except the normalization of the symmetry line to $\Re(s)=\frac{1}{2}$.\\

$L_E(s) = L(E,s+\frac{1}{2})$ gives a normalized $L$-function where the functional equation relates $s$ to $1-s$, giving
\begin{align}
L_E(s) &= \sum_{n}{\frac{\lambda (n)}{n^s}}\\
&=  \prod_{p | \Delta}{{\left(1- \frac{\lambda (p)}{p^{s}}\right)}^{-1}} \prod_{p\nmid \Delta}{{\left(1- \frac{\lambda (p)}{p^s} + \frac{1}{p^{2s}}\right)}^{-1}}
\end{align}
with functional equation
\begin{align}
L_E(s) = \omega_E {\left(\frac{2 \pi}{\sqrt{M}}\right)}^{2s-1} \frac{\Gamma\left(\frac{3}{2}-s\right)}{\Gamma\left(s+ \frac{1}{2}\right)}L_E\left(1-s\right),
\end{align}
where $\omega_E$ is $+1$ or $-1$ resulting in an odd or even functional equation for $L_E$ and
\begin{align}
\xi_E(s) =  \left(\frac{2 \pi}{\sqrt{M}}\right)^{-s} \Gamma\left(\frac{3}{2}-s\right)L_E\left(s\right) = \omega(E) \xi_E(1-s)
\end{align}
is entire. \\

Clearly $L_{E}(s)$ meets all the conditions in Definition~\ref{LFcond} for our definition of an $L$-function. We have demonstrated that each elliptic curve gives rise to a well-defined $L$-function.

 \subsection{Dirichlet $L$-functions}\label{sub:dir}

We start by defining Dirichlet characters.
\begin{definition}\cite{HBPrimenumber}\label{def:dirchar}
Let $q \in \mathbf{N}$. A Dirichlet character $\chi$ to modulus $q$ is a function $\chi: \mathbf{Z} \rightarrow \mathbf{C}$ such that
\begin{enumerate}
  \item $\chi\left(mn\right) = \chi\left(m\right)\chi\left(n\right)$ for all $m,n \in \mathbf{Z}$;
  \item $\chi\left(q\right)$ has period $q$;
  \item $\chi\left(n\right)=0$ whenever $\left(n,q\right) \neq 1$; and
  \item $\chi\left(1\right)=1$.
\end{enumerate}
\end{definition}
A character is called \emph{primitive} if it cannot be induced by any character of smaller modulus. For any Dirichlet character of modulus d, $\chi_d\left(n\right)$, we can define a function \begin{align}
L\left(s, \chi_d\right) = \sum_{n=1}^{\infty}{\frac{\chi_d\left(n\right)}{n^s}}.
\end{align}
$L\left(s, \chi_d\right)$ has a Euler product
\begin{align}
L\left(s, \chi_d\right) = \prod_{p}\left(1- \frac{\chi_d(p)}{p^s}\right)^{-1}
\end{align}
and if $\chi_d$ is primitive $L\left(s, \chi_d\right)$ has an analytic continuation to the whole complex plane and has a function equation~\cite{HBPrimenumber}. Setting \begin{align}
\xi\left(s, \chi_d\right) = \left(\frac{d}{\pi}\right)^{\frac{s+a\left(\chi_d\right)}{2}}\Gamma\left(\frac{s+a\left(\chi_d\right)}{2}\right)L\left(s, \chi_d\right)
\end{align}
where
\begin{align}
a\left(\chi_d\right) = \begin{cases}
0 & \chi_d \left(-1\right) =1\\
1 & \chi_d \left(-1\right) =-1.
\end{cases}
\end{align}
Then the functional equation can be expressed as
 \begin{align}
 \xi\left(1-s, \overline{\chi_d}\right) = \frac{i^{a\left(\chi_d\right)}d^{\frac{1}{2}}}{\sum_{t=1}^{d}{\chi_d\left(t\right)e^{\frac{2 \pi i t}{d}}}}\xi\left(s, \chi_d\right).
 \end{align}
Specifically note when $\chi_d$ is real and $d>0$, the functional equation is
\begin{align}
L\left(1-s, \chi_d\right) = \left(\frac{d}{\pi}\right)^{s-\frac{1}{2}}\frac{\Gamma\left(\frac{s}{2}\right)}{\Gamma\left(\frac{1-s}{2}\right)}L\left(s, \chi_d\right).
\end{align}

Thus if we assume $\chi_d$ is a real, primitive Dirichlet character $L(s, \chi_d)$ meets all our requirements of an $L$-function in Definition~\ref{LFcond}.

\subsection{Grand Riemann Hypothesis (GRH)}

We have shown that each of these $L$-functions associated to an elliptic curve or Dirichlet character has a functional equation with a line of symmetry, similar to the Riemann zeta function and its critical line.
\begin{conj}[Grand Riemann Hypothesis \cite{iwaniec2004analytic}]\label{conj: GRH}

All non-trivial zeroes of $L$-functions $L(s)$ in the critical strip $0< \Re(s) < 1$ lie on the critical line $\Re(s) = \frac{1}{2}$.
\end{conj}

In this work, we will assume the Grand Riemann Hypothesis holds and that $L$-functions arising from elliptic curves or Dirichlet characters have all of their nontrivial zeros on their lines of symmetry i.e. the zeros have the form $\frac{1}{2}+i\omega$ with $\omega$ real.\cite{bombieri2000problems}\\

\subsection{Twisting $L$-functions into Families}\label{twistintro}
\subsubsection{Families of $L$-functions of Elliptic Curves}
Starting with a primitive $L$-function generated by an elliptic curve, we can create a family by twisting with real quadratic Dirichlet characters.  These characters can be described as follows.  We first define the Legendre symbol for prime $p$, and all integers $d$
\begin{align}
\left(\frac{d}{p}\right)_L = \begin{cases}
+1 & \text{if $p\nmid d$ and $x^2 \equiv d$ mod $p$ solvable},\\
0 & p | d,\\
-1 & \text{if $p\nmid d$ and $x^2 \equiv d$ mod  $p$ not solvable}.
\end{cases}
\end{align}
We then extend this to the Jacobi Symbol for positive odd integers $m$,
\begin{align}
\left(\frac{d}{m}\right)_J = \left(\frac{d}{p_1}\right)_L^{k_1} \cdots \left(\frac{d}{p_n}\right)_L^{k_n}
\end{align}
where $m= p_1^{k_1} \cdots p_n^{k_n}$.\\

The Kronecker symbol, $\left(\frac{d}{n}\right)_K$,  is an extension of the Jacobi symbol for all integers $d,n$ such that
\begin{align}
\left(\frac{ab}{cd}\right)_K = \bigg(\frac{a}{c}\bigg)_K \bigg(\frac{a}{d}\bigg)_K \left(\frac{b}{c}\right)_K \left(\frac{b}{d}\right)_K
\end{align}
such that $\left(\frac{d}{m}\right)_K = \left(\frac{d}{m}\right)_J$ for positive odd $m$,
\begin{align}
\left(\frac{d}{-1}\right)_K = \begin{cases}
1 & d>0,\\
-1 & d<0,
\end{cases}
\end{align}
 and
 \begin{align}
\left(\frac{d}{2}\right)_K = \begin{cases}
1 &  d \text{ odd}, d\equiv \pm 1 \text{ mod } 8,\\
0&   d \text{ even},\\
-1 &  d \text{ odd}, d\equiv \pm 3 \text{ mod } 8.
\end{cases}
\end{align}
The Kronecker symbol is a real quadratic Dirichlet character. From now on we will use $\chi_d(n) = \left(\frac{d}{n}\right)_K$.

We construct our family of $L$-functions by twisting with these Dirichlet characters. We start with a fixed primitive $L$-function $L_{E}(s)$ generated from an elliptic curves $E$ as described in Theorem~\ref{Ellipticderivation}.
\begin{align}
L_E(s) &= \sum_{n}{\frac{\lambda (n)}{n^s}}.
\end{align}

Using the quadratic Dirichlet characters we define a family parameterised by $d$.
\begin{align}
L_{E}(s, \chi_d) &= \sum_{n}{\frac{\lambda (n)\chi_d(n)}{n^s}}\\
&= \prod_{p}{\left(1- \frac{\lambda(p)\chi_d(p)}{p^s} + \frac{\Psi_M(p)\chi_d(p)^2}{p^{2s}}\right)^{-1}} \label{LFtwistfunct}
\end{align}
where
\begin{align}
\Psi_M(p)= \begin{cases}
1 & \text{if } p \nmid M,\\
0 & \text {otherwise}.
\end{cases}
\end{align}

Each $L$-function $L_E(s, \chi_{d})$ is actually the $L$-function associated with another elliptic curve $E_{d}$, the twist of $E$ by $d$~\cite{conrey2007autocorrelatio}.
Therefore it has a functional equation
\begin{align}
L_{E}(s, \chi_d) &= \chi_d(-M) \omega_E {\left(\frac{2 \pi}{\sqrt{M}|d|}\right)}^{2s-1} \frac{\Gamma\left(\frac{3}{2}-s\right)}{\Gamma\left(s+ \frac{1}{2}\right)}L_{E}\left(1-s, \chi_d\right)\\
&= \chi_d(-M) \omega_E \psi(s, \chi_d) L_{E}\left(1-s, \chi_d\right).
\end{align}
and all the other attributes we expect of $L$-functions. Of particular interest is that the Grand Riemann Hypothesis in Conjecture~\ref{conj: GRH} also applies to these twisted $L$-functions, so we are assuming that all the nontrivial zeros are on the $Re (s) = \frac{1}{2}$ line.\\

We will restrict our family of $L$-functions with the constraint $\chi_d(-M) \omega_E =1$, where $M$ is the conductor and $\omega_E$ the sign of the functional equation of the initial elliptic curve $L_E(s)$. This then gives us a family of $L$-functions with even functional equations:

\begin{align}\label{eq:E+def}
E^{+} = \{L_E(s, \chi_d): \chi_d\left(M\right)\omega_E = +1; d\;{\rm a\;fundamental\;discriminant},\;d > 0\}.
\end{align}

It is expected that the distribution of the zeros of any one of these $L$-functions  infinitely high on the critical line is like that of the eigenvalues of random unitary matrices.
But rather than fixing the $L$-function and looking at the distribution up the critical line, we can fix a height on the line and average through the family of $E^{+}$ with even functional equations.

This way we can look at the distribution of the zeros on the critical line closest to the real axis such as the distribution of the $i$th zero
\begin{align}
J^{i}_{(X,F)}[a,b]= \frac{\left|\left\{f \in F_X; \frac{t_i^f log c_f}{K_F} \in [a,b]\right\}\right|}{|F_X|}
\end{align}
where $F$ is the family of $L$-functions with conductor $c_f$, and $K_F$ is a constant depending on the symmetry of the family. In the example above the conductor is the parameter $d$. $F_X$ are the $L$-functions with $c_f<X$ and $t_i^f$ is the $i$th zero of $L$-function $f$ above the real axis.\\

Katz and Sarnak conjectured that the zero statistics of twisted elliptic curve $L$-functions with even functional equation would behave like the eigenvalues of $SO(2N)$ matrices and that zeros of families of $L$-functions with odd functional equation should behave like eigenvalues of $SO(2N+1)$~\cite{katz1999zz}. For example, for the family $F=E^+$ we expect
\begin{align}
\lim_{X \rightarrow \infty} J^{i}_{(X,F)}[a,b] = \lim_{N \rightarrow \infty } meas \left\{A \in SO(2N): \frac{N \theta_i(A)}{\pi}\in [a,b]\right\}
\end{align}
where $meas \left\{A \in SO(2N): \frac{N \theta_i(A)}{\pi}\in [a,b]\right\}$ is the distribution of the $i$th eigenvalue of a matrix $A$ varying over $SO(2N)$.

\subsubsection{Families of Dirichlet $L$-functions} \label{sub:dir2}

Recall that we defined $L\left(s, \chi_d\right)$ for some real, primitive Dirichlet character of modulus d, $\chi_d\left(n\right)$ as
\begin{align}
L\left(s, \chi_d\right) = \sum_{n=1}^{\infty}{\frac{\chi_d\left(n\right)}{n^s}}.
\end{align}
We can then define a family of $L$-functions $D^{+} = \left\{L(s, \chi_d) : d \;{\rm a\;fundamental\;discriminant},\;d>0\right\}$, where $d$ (the conductor of the $L$-function) is the ordering parameter of the family.

Katz and Sarnak \cite{katz1999zz} conjecture that the zero statistics of $D^{+} $ behave like the eigenangles of $USp(2N)$ matrices, citing evidence from Ozluk and Snyder \cite{ozluk1999distribution} and Rubinstein \cite{rubinstein1998evidence} to support this conjecture.

\section{Zero Statistics of Elliptic Curve $L$-functions}\label{sec:zerostatselliptic}

In this section we will consider a family of quadratic twists (with even functional equation) of an elliptic curve $L$-function. Their  zeros are conjectured to behave similarly to the eigenangle statistics of matrices from $SO(2N)$ with Haar measure. We will apply the method of \cite{conrey2008ce} to conjecture the $n$-level density of the zeros of these $L$-functions.  We will see that the steps mirror the calculations for orthogonal matrices in Section~\ref{sec:ortho}.  The main result to be derived in this section (stated formally in Theorem \ref{thm:LFmain}, and conditional on the Generalised Riemann Hypothesis for $L$-functions in this family and the Ratios Conjecture discussed later in this section) is the following form of the $n$-level density for zeros of elliptic curve $L$-functions in a family. For simplicity at this stage we denote the family by $\mathcal{F}$, a member of the family is labelled $\mathcal{E}$ and $\gamma_{i,\mathcal{E}}$ is the $i$th zero of  that $L$-function above the real axis.  We find the $n$-level density is:

\begin{align}\label{eq:generaln-level}
\begin{split}
&\sum_{\mathcal{E}\in \mathcal{F}}\sum_{\substack{j_1, \cdots, j_n =1 \\ j_i \neq j_k \forall i,k}}^\infty{ f(\gamma_{j_1,\mathcal{E}}, \cdots, \gamma_{j_n,\mathcal{E}}) }\\
& \qquad = \frac{1}{(2 \pi )^n}\sum_{\substack{K \cup L \cup M\\= \{1, \cdots, n\}}}{(-1)^{|M|}\left(\int_{0}^{\infty}\right)^n J_{E}^{*}(-iz_K \cup i z_L, -iz_M)}\\
& \qquad \qquad \times f(z_1, \cdots, z_n)dz_1 \cdots dz_n + O\left(X^{\frac{1}{2}+\epsilon}\right),
\end{split}
\end{align}
where $J_{E}^{*}(-iz_K \cup i z_L, -iz_M)$  is defined in (\ref{eq:JstarE2}) for the family of quadratic twists that we are interested in and the sum in the right hand side involving $K$, $L$ and $M$ is over disjoint subsets of $\{1,\ldots,n\}$.   We note how the structure of this $n$-level density mirrors that of the equivalent statistic for eigenvalues of matrices from $SO(2N)$ given in Theorem \ref{thm:orthoNcorr1}.

\subsection{Ratios of Elliptic Curve $L$-functions} \label{sec:ratios of twists}
Let $A= \{\alpha_1, \cdots, \alpha_K\}$ and $B= \{\gamma_1, \cdots, \gamma_Q\}$ be sets of complex numbers and $\epsilon \in \{1,-1\}^K$. In order to derive the $n$-level density of the zeros near the critical point $s = \frac{1}{2}$, we consider ratios of twisted $L$-functions with even functional equation averaged over the family $E^+$ defined at (\ref{eq:E+def}).  All sums of the form $0<d\leq X$ in this section include only fundamental discriminants.
\begin{align}
\begin{split}
&R_E\left(\alpha_1, \cdots, \alpha_K, \gamma_1, \cdots, \gamma_Q \right)\\
& \qquad = \sum_{\substack{0<d \leq X \\ \chi_d\left(-M\right)\omega_E = +1}}{\frac{L_E\left(\frac{1}{2}+ \alpha_1, \chi_d\right)\cdots L_E\left(\frac{1}{2}+ \alpha_K, \chi_d\right)}{L_E\left(\frac{1}{2}+ \gamma_1, \chi_d\right)\cdots L_E\left(\frac{1}{2}+ \gamma_Q, \chi_d\right)}}. \label{Lratioequation}
\end{split}
\end{align}
We use the expression given for this quantity in \cite{conrey2007autocorrelatio}; details of the derivation  when $K=Q=1$ can also be found in \cite{huynh2009lowe}.

This gives us the following conjecture, from ``Autocorrelations of ratios of $L$-functions"\cite{conrey2007autocorrelatio}, which we have written in set notation as in the random matrix calculations.
\begin{conjecture}\label{thm:LFRat2}{\rm (Ratios of $L$-functions of Elliptic Curves)}
Let $A$, $B$ be sets of complex numbers such that
\begin{align}
\begin{split}
-\frac{1}{4} < \Re\left(\alpha\right) < \frac{1}{4} &\qquad \forall \alpha \in A\\
\frac{1}{\log X} \ll \Re\left(\gamma\right)< \frac{1}{4} &\qquad \forall \gamma \in B\\
Im\left(\alpha\right), Im\left(\gamma\right)\ll X^{1 - \epsilon }& \qquad \forall \alpha \in A, \gamma \in B
\end{split}
\end{align}
Then
\begin{align}
\begin{split}\label{eq:REAB}
&R_E(A,B):= \sum_{\substack{0<d \leq X \\ \chi_d\left(-M\right)\omega_E = +1}}{\frac{\prod_{\alpha \in A}{L_E\left(\frac{1}{2} +\alpha, \chi_d\right)}}{\prod_{\gamma \in B}{L_E\left(\frac{1}{2}+\gamma, \chi_d\right)}}}\\
& \qquad = \sum_{\substack{0<d \leq X \\ \chi_d\left(-M\right)\omega_E = +1}}{\sum_{D \subseteq A}{\left(\frac{\sqrt{M}|d|}{2 \pi}\right)}^{-2 \sum_{\delta \in D}{\delta}}\prod_{\delta \in D}\frac{\Gamma\left(1- \delta\right)}{\Gamma\left(1+ \delta\right)}}\\
& \qquad \qquad \times Y_E\left(A,B,D\right) A_E\left(A,B,D\right) + O\left(X^{1/2+ \epsilon}\right)
\end{split}
\end{align}
where the sum over $d$ is over fundamental discriminants, $\zeta(s)$ is the classical Riemann zeta function and
\begin{align}
Z_{\zeta}\left(A,B\right) &= \prod_{\alpha \in A, \beta
\in B}{\zeta\left(1+\alpha+\beta \right)}\\
Y_{\zeta}\left(A\right)&=
\prod_{\alpha \in A}{\zeta\left(1+2\alpha\right)}\\
U_{-}\left(A\right) &= \prod_{\alpha \in A}\left(1-\frac{\lambda\left(p\right)}{p^{1/2+\alpha}}+\frac{1}{p^{1+2\alpha}}\right)\\
U_{+}\left(A\right) &= \prod_{\alpha \in A}\left(1+\frac{\lambda\left(p\right)}{p^{1/2+\alpha}}+\frac{1}{p^{1+2\alpha}}\right)\\
V\left(A\right) &= \prod_{\alpha \in A}\left(1-\frac{\lambda\left(p\right)}{p^{1/2+\alpha}}\right)
\end{align}
For $D \subseteq A$ we define $D^{-} = \{-\delta; \delta \in D\}$ and we can write
\begin{align}
Y_E\left(A,B,D\right) &= \sqrt{\frac{Z_{\zeta}\left(\left(A \backslash D\right)\cup D^{-}, \left(A \backslash D\right)\cup D^{-}\right) Z_{\zeta}\left(B,B\right) Y_{\zeta} \left(B\right)}{Z_{\zeta}\left(\left(A \backslash D\right)\cup D^{-},B\right)^2 Y_{\zeta}\left(A \backslash D\right)Y_{\zeta}\left( D^{-}\right)}}\\ \label{eq:LF1}
A_{E,1}\left(A,B,D\right)&= \frac{1}{Y_E\left(A,B,D\right)}\\
\begin{split}
A_{E,2}\left(A,B,D\right) &= \prod_{p \nmid M}\frac{1}{1+\frac{1}{p}}\left(\frac{U_{-}\left(B\right)}{2U_{-}\left(\left(A \backslash D\right)\cup D^{-}\right)}\right.\\
& \qquad \left.+\frac{U_{+}\left(B\right)}{2U_{+}\left(\left(A \backslash D\right)\cup D^{-}\right)}+\frac{1}{p}\right)
\end{split}\\
A_{E,3}\left(A,B,D\right) &= \prod_{p | M}{\frac{V\left(B\right)}{V\left(\left(A \backslash D\right) \cup D^{-}\right)}}\\
A_E\left(A,B,D\right) &= A_{E,1}\left(A,B,D\right)A_{E,2}\left(A,B,D\right)A_{E,3}\left(A,B,D\right).\label{eq:LF2}
\end{align}
\end{conjecture}

\subsection{Discussion of the Error Term}\label{sec:error}
Little is known in general about the error term in the Ratios Conjectures, although we list below cases where some specific instances of the conjecture are better understood. It is clear that considerable error might arise in several steps of the ratios ``recipe" described in Section \ref{sec:background}, which may be of a size similar to that of the main term. The conjecture is that these errors will cancel each other out and there are many situations where there is evidence for an error of the form $O\left(X^{1/2+ \epsilon}\right)$ as stated in \cite{conrey2007autocorrelatio} and reproduced in Conjecture \ref{thm:LFRat2}.  Note that in (\ref{eq:REAB}) the average over the family (the sum over $d$) is not divided by the number of terms in the sum as was the case at (\ref{eq:ratioconjformulation}).  Also, for the family of quadratic twists of an elliptic curve $L$-function, the log conductor $c(f)$ defined in Section \ref{sec:background} is  (asymptotically for large $d$) $2\log d$ for the $L$-function $L_E(\tfrac{1}{2}+\alpha,\chi_d)$, so the general form of the error term given at (\ref{eq:ratioconjformulation}) would here correspond to $O(X^{1-\delta})$ for some $\delta>0$. 

In several instances the error term in the Ratios Conjecture has been tested by assuming the Ratios Conjecture and deducing from it the 1-level density: a result like (\ref{eq:generaln-level}) with $n=1$.  This is then compared with a calculation of the 1-level density using rigourous methods.  These rigourous methods require a restriction on the support of the Fourier transform of the test function $f$.

Miller rigourously calculated the 1-level density for the symplectic family of quadratic Dirichlet characters arising from even fundamental discriminants $d \le X$. He found agreement with the 1-level density computed using the Ratios Conjecture up to an error term of $O(X^{1/2+\epsilon})$ for test functions supported in $(-1, 1)$\cite{kn:mil07}.

Miller and Montague tested the accuracy of the Ratios Conjecture with an orthogonal family. They considered the 1-level density of families of cuspidal newforms of constant sign and confirmed the accuracy of the conjecture up to an error equivalent to $O(X^{1/2+\epsilon})$ when the test function has support in $(-1, 1)$. When the test function has support in $(-2,2)$ they can show a power savings error \cite{kn:milmon}.

Goes, Jackson and Miller et al.~showed similar results  for the unitary family of all Dirichlet $L$-functions with prime conductor. They found agreement to the Ratios Conjecture predictions with a square-root error if the support of the Fourier transform of the test function is in $(-1,1)$ and a power saving error for support up to $(-2,2)$ \cite{kn:Goesetal}. Similarly Fiorilli and Miller \cite{kn:fiomil15} calculated the 1-level density for a family of Dirichlet $L$-functions and by explicitly computing lower order terms showed that the error term of type $O(X^{1/2+\epsilon})$ was the best possible.

Huynh, Miller and Morrison tested the Ratios Conjecture prediction for the 1-level density of a family of quadratic twists of a fixed elliptic curve with prime conductor. They found agreement up to an error term of size $X^{\frac{ 1+\sigma}{2}}$ for test functions supported in $(-\sigma, \sigma)$\cite{kn:huymilmor}.

In the case of moments (ratios with no denominator), for the family of real, quadratic Dirichlet $L$-functions, a sequence of works \cite{kn:sou00,kn:dgh03,kn:zhang05,kn:young13} indicate that in this case the third moment has a term of size $O(X^{3/4+\varepsilon})$ in place of the $O(X^{1/2+\varepsilon})$ in (\ref{eq:REAB}).  The size of this term was also extensively investigated numerically by Alderson and Rubinstein \cite{kn:aldrub12}.  Very recently Florea \cite{kn:florea15} has found explicitly a term which in this notation would be equivalent to  $O(X^{1/3+\varepsilon})$ in the first moment of quadratic Dirichlet $L$-functions in the function field setting. 

Thus we have settings where there is support for the size of the error term being the square root of the size of the main term, as stated in the original Moment and Ratios Conjectures, such as (\ref{eq:REAB}),  and there are instances where the square root error term seems too optimistic.  As we still lack a complete understanding of these terms, we will continue in this paper with the original statement of the Ratios Conjectures  with a $O\left(X^{\frac{1}{2}+\epsilon}\right)$ error term, but bearing in mind that there may be instances where this breaks down.

\subsection{Differentiating Ratios of Elliptic Curve $L$-functions}
Using Conjecture~\ref{thm:LFRat2}, we want to find the logarithmic derivative

\begin{theorem} \label{LFJconj}
Assume Conjecture~\ref{thm:LFRat2}. Let $A$ be a set of complex numbers such that
\begin{align}
\begin{split} \label{eq:restrictionsE}
\frac{1}{\log X} \ll \Re\left(\alpha\right)< \frac{1}{4} &\qquad \forall \alpha \in A\\
Im\left(\alpha\right) \ll X^{1 - \epsilon }& \qquad \forall \alpha \in A
\end{split}
\end{align}
Then $J_E\left(A\right)=J_E^{*}\left(A\right)+ O\left(X^{1/2+\epsilon}\right)$, where
\begin{align}
J_E\left(A\right)= \sum_{\substack{0<d \leq X \\ \chi_d\left(-M\right)\omega_E = +1}}{\prod_{\alpha \in A}\frac{L^{'}_E\left(\frac{1}{2}+ \alpha, \chi_d\right)}{L_E\left(\frac{1}{2}+ \alpha, \chi_d\right)}}
\end{align}
and
\begin{align}
\begin{split}
J^{*}_E\left(A\right) =& \sum_{\substack{0<d \leq X \\ \chi_d\left(-M\right)\omega_E = +1}}{\sum_{D \subseteq A}{{\left(\frac{\sqrt{M}|d|}{2 \pi}\right)}^{- \sum\limits_{\delta \in D}{2\delta}}\prod_{\delta \in D}{\frac{\Gamma\left(1- \delta\right)}{\Gamma\left(1+ \delta\right)}}}}\\
& \times \sqrt{\frac{Z_{\zeta}\left(D^{-},D^{-}\right)Z_{\zeta}\left(D,D\right)Y_{\zeta}\left(D\right)}{Z_{\zeta}^{\dag}\left(D^{-},D\right)^{2}Y_{\zeta}\left(D^{-}\right)}}\\
&\times \left(-1\right)^{|D|}A_E\left(D,D,D\right)\\
&\times \sum_{\substack{A \backslash D=W_1 \cup \cdots \cup W_R\\}}{\prod_{r=1}^{R}{\widetilde{H_D}\left(W_r\right)}}\\
\end{split}
\end{align}
where
\begin{align}
\widetilde{H_D}\left(W_r\right) &= H_D\left(W_r\right)+ A_{D,1}\left(W_r\right)+ A_{D,2}\left(W_r\right)+ A_{D,3}\left(W_r\right)\label{Hwiggle}\\
H_D\left(W_r\right) &= \begin{cases}
\sum_{\delta \in D}\left(\frac{\zeta^{'}}{\zeta}\left(1+ a- \delta \right)-\frac{\zeta^{'}}{\zeta}\left(1+ a +\delta \right)\right)- \frac{\zeta^{'}}{\zeta}\left(1+ 2a\right) & W_r = \{a\}\\
\left(\frac{\zeta^{'}}{\zeta}\right)'\left(1+ a_1 + a_2\right) & W_r= \{a_1, a_2\}\\
1 & W_r=\emptyset\\
0&|W_r|\geq 3
\end{cases}\\
A_{D,i}\left(W_r\right)&= \left.\prod_{w \in W_r}{\frac{\partial}{\partial w}}\log A_{E,i}\left(A,B,D\right)\right|_{B=A} \label{def:AD}
\end{align}
and $A_E, Z_{\zeta}, Y_{\zeta}$ etc defined as in Conjecture \ref{thm:LFRat2}. $A_E\left(A,B,D\right)$ is an analytic function. The dagger adds a
restriction that a factor $\zeta(1+x)$ is omitted if its argument is zero. The sum over $d$ is over fundamental discriminants.
\end{theorem}

The proof of this theorem involves exactly the same arguments one would use to prove Theorem~\ref{thm:ortho2}, although in Theorem~\ref{LFJconj} there is the added complication of the arithmetic terms.

A necessary step in this proof is to evaluate $\left.A_E\left(A,B,D\right)\right|_{B=A}$.
\begin{lemma}\label{lem:diffA}
Let $A,B$ be sets and $D$, a subset of $A$. Then $\left.A_E\left(A,B,D\right)\right|_{B=A} = A_E\left(D,D,D\right)$. This is an analytic function.
\end{lemma}
\begin{proof}
$A_E\left(A,B,D\right)$ is analytic as we can tell from it's
construction in Section~\ref{sec:ratios of twists}. In particular, it must be analytic when $A=B$. However it is not immediately clear that $A_{E,1}\left(A,B,D\right)$ is convergent
when A=B due to the $\zeta\left(1\right)$ terms coming from
$Z\left(A/D \cup D^{-},B\right)$. In order to avoid problems of
convergence, we will rewrite $A_{E,1}\left(A,B,D\right)$  as a product
over primes and consider the expansion of each prime separately. Recall
\begin{align}
A_{E,1}\left(A,B,D\right) = \sqrt{\frac{Z\left((A/D) \cup
D^{-},B\right)^2 Y\left(A/D\right)Y\left(D^{-}\right)}{Z\left((A/D)\cup
D^{-}, (A/D) \cup D^{-}\right)Z \left(B,B,\right)Y\left(B\right)}}
\end{align}
where
\begin{align}
\zeta\left(s\right)&=\prod_{p}{\left(1-\frac{1}{p^s}\right)^{-1}}\\
Z\left(A,B\right) &= \prod_{\alpha \in A, \beta
\in B}{\zeta\left(1+\alpha+\beta \right)}\\
Y\left(A\right)&=
\prod_{\alpha \in A}{\zeta\left(1+2\alpha\right)}.
\end{align}
Now we will define
\begin{align}
z_p\left(s\right)&=\left(1-\frac{1}{p^s}\right)^{-1}\\
Z_p\left(A,B\right)&=\prod_{\substack{\alpha \in A \\ \beta \in B}}{z_p
\left(1+\alpha+\beta \right)}\\
Y_p \left(A\right)&= \prod_{\alpha \in A}{z_p \left(1+2\alpha\right)}
\end{align}
noting that $z_p\left(x\right)$ does not have a pole at 1 and that
$\prod_{p}{Z_p\left(A,B\right)} = Z\left(A,B\right)$ and
$\prod_{p}{Y_p\left(A\right)} = Y\left(A\right)$.
This allows us to rewrite $A_{E,1}$ to get
\begin{align}
A_{E,1,p}\left(A,B,D\right) &= \sqrt{\frac{Z_p\left((A/D)\cup
D^{-},B\right)^2 Y_p\left(A/D\right)Y_p\left(D^{-}\right)}{Z_p\left((A/D)\cup
D^{-}, (A/D) \cup D^{-}\right)Z_p \left(B,B,\right)Y_p\left(B\right)}}\\
A_{E,1}\left(A,B,D\right) &= \prod_{p}{A_{E,1,p}\left(A,B,D\right) }.
\end{align}
Then if we consider the value of $A_{E,1,p}\left(A,B,D\right)$ when $A=B$
\begin{align}
\left.A_{E,1,p}\left(A,B,D\right)\right|_{A=B}&= \sqrt{\frac{Z_p\left(D^{-},D\right)^2Y_p\left(D^{-}\right)}{Z_p
\left(D^{-},D^{-}\right)Z_p\left(D,D\right)Y_P\left(D\right)}}\\
&= A_{E,1,p}\left(D,D,D\right).
\end{align}
So
\begin{align}
\left. A_{E,1}\left(A,B,D\right)\right|_{A=B}= \prod_{p}{A_{E,1,p}\left(D,D,D\right)}.
\end{align}

We will leave this as a product over primes and instead consider the second part of $A_E\left(A,B,D\right)$,
\begin{align}
\begin{split}
&\left.A_{E,2}\left(A,B,D\right)\right|_{B=A} \\
&\qquad = \left.\prod_{p \nmid M}\frac{1}{1+\frac{1}{p}}\left(\frac{U_{-}\left(B\right)}{2U_{-}\left(\left(A/D\right)\cup D^{-}\right)}+\frac{U_{+}\left(B\right)}{2U_{+}\left(\left(A/D\right)\cup D^{-}\right)}+\frac{1}{p}\right)\right|_{B=A}
\end{split}\\
& \qquad= \prod_{p \nmid M}\frac{1}{1+\frac{1}{p}}\left(\frac{U_{-}\left( D\right) }{2U_{-}\left(D^{-}\right)}+\frac{U_{+}\left(D\right)}{2U_{+}\left( D^{-}\right)}+\frac{1}{p}\right)\\
&\qquad = A_{E,2}\left(D,D,D\right)
\end{align}
where
\begin{align}
U_{-}\left(A\right) = \prod_{\alpha \in A}\left(1-\frac{\lambda\left(p\right)}{p^{1/2+\alpha}}+\frac{1}{p^{1+2\alpha}}\right)\\
U_{+}\left(A\right) = \prod_{\alpha \in A}\left(1+\frac{\lambda\left(p\right)}{p^{1/2+\alpha}}+\frac{1}{p^{1+2\alpha}}\right).
\end{align}
and finally
\begin{align}
\left.A_{E,3}\left(A,B,D\right)\right|_{B=A} &= \left.\prod_{p | M}{\frac{V\left(B\right)}{V\left(\left(A/D\right) \cup D^{-}\right)}}\right|_{B=A}\\
&=\prod_{p | M}{\frac{V\left(D\right)}{V\left(D^{-}\right)}}\\
&=A_{E,3}\left(D,D,D\right).
\end{align}
where $V\left(A\right) = \prod_{\alpha \in A}\left(1-\frac{\lambda\left(p\right)}{p^{1/2+\alpha}}\right)$.
Putting these together, and remembering that $A_E\left(A,B,D\right)$ is an analytic function, we get
\begin{align}
\left.A_E\left(A,B,D\right)\right|_{B=A}&= \prod_{p}{A_{E,1,p}\left(D,D,D\right)}A_{E,2}\left(D,D,D\right)A_{E,3}\left(D,D,D\right)\\
&= A_E\left(D,D,D\right).
\end{align}
\end{proof}
 With this lemma, we can then approach the proof of Theorem~\ref{LFJconj}.  The following steps are extremely similar to those necessary for the random matrix proof of Theorem~\ref{thm:ortho2}.
 \begin{proof}
 Using Conjecture~\ref{thm:LFRat2}, we want to find the logarithmic derivative
\begin{align}
J_E\left(A\right)&= \sum_{\substack{0<d \leq X \\ \chi_d\left(-M\right)\omega_E = +1}}{\prod_{\alpha \in A}\frac{L^{'}_E\left(\frac{1}{2}+ \alpha, \chi_d\right)}{L_E\left(\frac{1}{2}+ \alpha, \chi_d\right)}}\\
&= \left. \prod_{\alpha \in A}{\frac{\partial}{ \partial \alpha}}R_E\left(A, B\right)\right|_{B=A}
\end{align}
First we notice that if we substitute A in for B in $Y_E\left(A,B,D\right)$, we get
\begin{align}
Y_E\left(A,A,D\right)&=\sqrt{\frac{Z\left(D^{-}, D^{-}\right)Z\left(D,D\right)Y\left(D\right)}{Z\left(D^{-},D\right)^{2}Y\left(D^{-}\right)}}.
\end{align}
The $Z\left(D^{-},D\right)^{2}$ term in the denominator contains $\zeta\left(1\right)$, which is a pole. So $Y_E\left(A,A,D\right)$ is zero unless $D$ is the empty set as in the random matrices case. This means when we differentiate $R_E\left(A,B\right)$ by each $\alpha \in D$, all other terms are zero except the one where $Z\left(D^{-},D\right)$ term is differentiated by every $\alpha \in D$. Let $\beta_{\alpha}$ represent the element of $B$ that will be replaced by $\alpha$ when $B$ is substituted for $A$. Expanding $\zeta$ and $\zeta{'}$ about 1, it's clear to see that
\begin{align}
\left.\frac{\partial}{\partial \alpha}\left(\frac{1}{\zeta\left(1+ \beta_{\alpha} - \alpha\right)}\right)\right|_{\beta_{\alpha}=\alpha}= -1.
\end{align}
Using this, we can see that
\begin{align}
\left.\prod_{\alpha \in D}\frac{\partial}{\partial \alpha}\left(\frac{1}{Z\left(D^{-}, B_D\right)}\right)\right|_{B_D=D}= \frac{\left(-1\right)^{|D|}}{Z^{\dag}\left(D^{-},D\right)}
\end{align}
where $B_D$ represents the part of $B$ that will be substituted for $D$, and $Z^{\dag}$ means that we only include those terms of $Z$ that are not equal to $\zeta\left(1\right)$. All that remains now is to differentiate the remainder by each $\alpha \in A/D$.
\begin{align}
\begin{split}
J_E\left(A\right)&= \sum_{\substack{0<d \leq X \\ \chi_d\left(-M\right)\omega_E = 1}}{\sum_{D \subseteq A}{{\left(\frac{\sqrt{M}|d|}{2 \pi}\right)}^{- \sum\limits_{\delta \in D}{2\delta}}\prod_{\delta \in D}{\frac{\Gamma\left(1- \delta\right)}{\Gamma\left(1+ \delta\right)}}}}\\
    & \qquad \times (-1)^{|D|} \sqrt{\frac{Z\left(D^{-},D^{-}\right)Z\left(D,D\right)Y\left(D\right)}{Z^{\dag}\left(D^{-},D\right)^{2}Y\left(D^{-}\right)}}\\
    & \qquad \times \left(\prod_{\alpha \in A/D}{\frac{\partial}{\partial \alpha}}\exp(H^{A,B}_D)\right. \\
        &\qquad \qquad  \left.\times \sqrt{\frac{Y\left(B_D\right)Z\left(B_{A/D},B_{D}\right)^{2}Z\left(B_{A/D},B_{A/D}\right)}{Z\left(D^{-}, B_{A/D}\right)^2 }}\right)\Bigg|_{B=A}\\
    &\qquad + O\left(X^{1/2+ \epsilon}\right)
\end{split}
\end{align}
where
\begin{eqnarray}
H^{A,B}_D&=& \log \left(\sqrt{\frac{Z\left(A/D, D^{-}\right)^2 Z\left(A/D, A/D\right)}{Z\left(A/D,B_{A/D}\right)^{2}Z\left(A/D,B_D\right)^{2}Y\left(A/D \right)}}\right. \\
&& \qquad \times \left. \vphantom{\sqrt{\frac{Z}{Z}}}A_E\left(A,B,D\right)\right)\nonumber\\
&&= \sum_{\substack{\alpha \in A/D\\ \beta \in D}}{\log \zeta\left(1+ \alpha -\beta\right)}+\frac{1}{2}\sum_{\substack{\alpha \in A/D\\ \beta \in A/D}}{\log \zeta\left(1+ \alpha + \beta \right)}\nonumber\\
&&\qquad - \sum_{\substack{\alpha \in A/D\\ \beta \in B}}{\log \zeta\left(1+ \alpha +\beta\right)} - \frac{1}{2}\sum_{\alpha \in A/D}{\log \zeta\left(1+ 2\alpha \right)}\\
&&\qquad + \log A_{E,1}\left(A,B,D\right) + \log A_{E,2}\left(A,B,D\right)\nonumber\\
&&\qquad + \log A_{E,3}\left(A,B,D\right) \nonumber
\end{eqnarray}

We now note that if  $H$ is a differentiable function of $w\in
W$, then
\begin{eqnarray}\label{eq:logdiff}
\left(\prod_{w\in W}\frac{d}{dw}\right)e^H=  e^{H }
\sum_{W=W_1\cup \dots\cup W_R}H(W_1)\dots H(W_R)
\end{eqnarray}
where
\begin{eqnarray}
H(W)=\left(\prod_{w\in W} \frac{d}{dw}\right)H.
\end{eqnarray}
The sum is over all set partitions of     $W$ into disjoint sets
$W_j$.

We note that when we substitute $D$ for $B_{D}$ and $A/D$ for $B_{A/D}$, then
\begin{align}
\frac{Z\left(A/D, D^{-}\right)^2 Z\left(A/D,A/D\right)Z\left(B_{A/D},B_{D}\right)^{2}Z\left(B_{A/D},B_{A/D}\right)Y\left(B_D\right)}{Z\left(A/D,B_{A/D}\right)^{2}Z\left(A/D,B_D\right)^{2}, Z\left(D^{-}, B_{A/D}\right)^2 Y\left(A/D \right)}=1
\end{align}
and apply Lemma~\ref{lem:diffA} to see that
\begin{eqnarray}
&&\left(\prod_{\alpha \in A/D}{\frac{\partial}{\partial \alpha}}\exp(H^{A,B}_D)\right. \\
 &       &\qquad \qquad  \left.\times \sqrt{\frac{Y\left(B_D\right)Z\left(B_{A/D},B_{D}\right)^{2}Z\left(B_{A/D},B_{A/D}\right)}{Z\left(D^{-}, B_{A/D}\right)^2 }}\right)\Bigg|_{B=A}\\
 &&=A_E(D,D,D)\sum_{A/D=W_1\cup\cdots\cup W_R}{\prod_{r=1}^{R}{\widetilde{H_D}\left(W_r\right)}},
 \end{eqnarray}
 where $\widetilde{H_D}\left(W_r\right)$ is defined in Theorem~\ref{LFJconj} and plays the role of $H(W)$ in (\ref{eq:logdiff}).
This will allow us to simplify our solution after the differentiation.
\end{proof}

\subsection{Residue Theorem for Elliptic Curve $L$-functions} \label{sec:Eresidue}

Assuming the Riemann Hypothesis, the only possible poles of $J_E^{*}(A)$ are when $\alpha^{*} = - \beta^{*}$ for some $\alpha^{*}, \beta^{*} \in A$, or when $\alpha^{*} = 0$. We need to know what the residues are at these poles.
\begin{theorem}\label{LFresiduef}
Let A be a finite set of complex numbers, let $\alpha^{*},\beta^{*} \in A$ and $A^{'}=A \backslash \{\alpha^{*},\beta^{*}\}$ and let $J_E^{*}(A)$ be as defined in Theorem~\ref{LFJconj}. Then
\begin{align}
\Res_{\alpha^{*} =-\beta^{*}} \left(J_E^{*}(A)\right) &= J_E^{*}(A^{'} \cup \{\beta^{*}\}) + J_E^{*}(A^{'} \cup \{-\beta^{*}\})-\frac{\psi^{'}}{\psi}\left(\frac{1}{2}+\beta^{*},\chi_d\right) J_E^{*}(A^{'})
\end{align}
where $\psi\left(s,\chi_d\right)=\chi_d(-M) \omega_E \left(\frac{2 \pi}{\sqrt{M}|d|}\right)^{2s-1} \frac{\Gamma(\frac{3}{2}-s)}{\Gamma(s+ \frac{1}{2})}$ comes from the functional equation for $L_E(s,\chi_d)$  and
\begin{align}
\Res_{\alpha^{*} =0} \left(J_E^{*}(A)\right) &= 0.
\end{align}
\end{theorem}

\begin{proof}

Let $D\subseteq A$,  define $P_D$ and $\widetilde{Q}$ as
\begin{align}
J_E^{*}(A) &= \sum_{D \subseteq A}{P_D(A \backslash D)}\\
&=\sum_{D \subseteq A}{\widetilde{Q(D)} \sum_{A \backslash D= W_1\cup\cdots\cup W_R}}{\prod_{r=1}^R{\left(\widetilde{H_D}(W_r)\right)}},
\end{align}
where
\begin{align}
 \begin{split} \label{def:Qtilde}
\widetilde{Q(D)}&= \sum_{\substack{0<d \leq X \\ \chi_d\left(-M\right)\omega_E = +1}}{{{\left(\frac{\sqrt{M}|d|}{2 \pi}\right)}^{- \sum\limits_{\delta \in D}{2\delta}}\prod_{\delta \in D}{\frac{\Gamma\left(1- \delta\right)}{\Gamma\left(1+ \delta\right)}}}}\\
& \qquad \times \sqrt{\frac{Z_{\zeta}\left(D^{-},D^{-}\right)Z_{\zeta}\left(D,D\right)Y_{\zeta}\left(D\right)}{Z_{\zeta}^{\dag}\left(D^{-},D\right)^{2}Y_{\zeta}\left(D^{-}\right)}}\\
& \qquad\times \left(-1\right)^{|D|}A_E\left(D,D,D\right).\\
\end{split}
\end{align}
$\widetilde{H_D}$ is defined in Theorem~\ref{LFJconj} and $A_E, Z_\zeta, Y_\zeta$, etc, are defined as in Conjecture \ref{thm:LFRat2}.

Using $\widetilde{Q(D)}$ in place of $Q(D)$  and $\widetilde{H_D(W)}$ in place of $H(D,W)$, conditions {\bf P1} to {\bf P4} and {\bf R1}  and {\bf R2} (from Section \ref{sub:residue}) hold with
\begin{equation}
f(\beta^*)=-\frac{\psi^{'}}{\psi}\left(\frac{1}{2}+\beta^{*}, \chi_d\right)=2\log\Big(\frac {\sqrt{M}|d|}{2\pi} \Big) +\frac{\Gamma'}{\Gamma}(1-\beta^*)+\frac{\Gamma'}{\Gamma}(1+\beta^*).\label{eq:LFf}
\end{equation}
To show this we need to find expansions for  $A_D\left(W_r\right), H_D\left(W_r\right)$ and $A_E\left(D\right)$ around $\alpha^{*}=\beta^{*}$. These lemmas will be given without proof, but complete proofs can be found in Amy Mason's Thesis \cite{Mason2013}.
\subsubsection{Taylor Expansions for $A_{D,i}\left(W_r\right), H_D\left(W_r\right)$ and $A_E\left(D,D,D\right)$}
Recall that
\begin{align}A_{D,i}\left(W_r\right)= \left.\prod_{w \in W_r}{\frac{\partial}{\partial w}}\log A_{E,i}\left(A,B,D\right)\right|_{B=A} \mbox{ for } i=1,2,3.
\end{align}
 Then we find
\begin{lemma}\label{ADresidue}
Let $\alpha^{*}, \beta^{*} \in D$ and $D^{'} = D/\{\alpha^{*},\beta^{*}\}$. For $i=1,2,3$
\begin{align}
\begin{split}
&\left.\frac{\partial}{\partial \alpha^{*}}{A_{D,i}\left(W_r\right)}\right|_{\alpha^{*}=-\beta^{*}}\\
& \qquad = \left(-A_{D,i}\left(W_r + \{\beta^{*}\}\right) -\left. A_{D,i}\left(W_r + \{\alpha^{*}\}\right)\right)\right|_{\alpha^{*}=-\beta^{*}}
\end{split}
\end{align}
and
\begin{align}
\left.A_{D,i}\left(W_r\right)\right|_{\alpha^{*}=-\beta^{*}} = A_{D^{'},i}\left(W_r\right)
\end{align}
\end{lemma}

The second lemma we need for the Residue Theorem is to prepare for expanding $H_D^{A,B}\left(W_r\right)$.
\begin{lemma}\label{HDresidue}
Let $\alpha^{*}, \beta^{*} \in D$ and $D^{'} = D/\{\alpha^{*},\beta^{*}\}$. For $i=1,2,3$
\begin{align}
\begin{split}
\left.\frac{\partial}{\partial \alpha^{*}}H_D\left(W_r\right)\right|_{\alpha^{*}=-\beta^{*}} &= \left(-H_{D}\left(W_r + \{\beta^{*}\}\right)\right. \\
& \qquad -\left.\left. H_{D}\left(W_r + \{\alpha^{*}\}\right)\right)\right|_{\alpha^{*}=-\beta^{*}}
\end{split}
\end{align}
and
\begin{align}
\left.H_{D}\left(W_r\right)\right|_{\alpha^{*}=-\beta^{*}} = H_{D^{'}}\left(W_r\right).
\end{align}
\end{lemma}

Most importantly, these lemmas show that when $\alpha^{*}, \beta^{*} \in D$ there is a Taylor expansion of $\widetilde{H_D}\left(W_r\right)$ at $\alpha^{*} = -\beta^{*}$ where $\widetilde{H_D}\left(W_r\right) = H_D\left(W_r\right) + A_{D,1}\left(W_r\right) + A_{D,2}\left(W_r\right) + A_{D,3}\left(W_r\right)$
\begin{lemma}\label{HDtilderesidue}
Let $\widetilde{H_D}\left(W_r\right)$ etc. as defined in Equations~\ref{Hwiggle}-~\ref{def:AD}. Then
\begin{align}
\begin{split}
\widetilde{H_D}\left(W_r\right) &=  \widetilde{H_{D^{'}}}\left(W_r\right)\\
    & \qquad \left. + \left(\alpha^{*}+\beta^{*}\right)\left(-\widetilde{H_{D^{'}}}\left(W_r+\{\beta^{*}\}\right)-\widetilde{H_{D^{'}}}\left(W_r+\{\alpha^{*}\}\right)\right)\right|_{\alpha^{*} = -\beta^{*}}\\
    & \qquad +O\left(|\alpha^{*}+\beta^{*}|^2\right)
\end{split}
\end{align}
\end{lemma}

Finally, we need the expansion of $A_E\left(D,D,D\right)$ about $\alpha^{*}=-\beta^{*}$ when $\alpha^{*},\beta^{*} \in D$.
\begin{lemma} \label{AEresidue}
Let $A_E\left(D,D,D\right)$ be as defined in Equations~\ref{eq:LF1} -~\ref{eq:LF2} and $\widetilde{H_D}\left(W_r\right)$ etc. as defined in Equations~\ref{Hwiggle}-~\ref{def:AD}.
\begin{align}
\begin{split}
A_E\left(D,D,D\right) &= A_E\left(D^{'}, D^{'}, D^{'}\right)\\
    & \qquad \times \left[1\vphantom{\frac{Z}{Z}}\right.-\left(\alpha^{*}+\beta^{*}\right)\left(\widetilde{H_D}\left(\{\beta^{*}\}\right)- H_D\left(\{\beta\}\right)\right. \\
        & \qquad \qquad \left. +\widetilde{H_D}\left(\{-\beta^{*}\}\right) - H_D\left(\{-\beta\}\right)\right)\\
        & \qquad \qquad \left.\vphantom{\frac{Z}{Z}}+ O\left(|\alpha^{*}+\beta^{*}|^2\right)\right]
\end{split}
\end{align}
\end{lemma}

With these lemmas in place, we are ready to show that {\bf P1} to {\bf P4} from Definition~\ref{def:PropP} hold for $\widetilde{H_D}$, $\widetilde{Q}$ and $f$ as defined above in Equation~\ref{eq:LFf}.

{\bf P1}: Suppose $\alpha^{*},\beta^{*} \in A/D$, then $\widetilde{Q(D)}$, is independent of $\alpha^{*}$ and $\beta^{*}$.

As $\alpha^{*} \rightarrow \beta^{*}$ the only pole in $\widetilde{H_D}(W_r)$ comes when $W_r = \{\alpha^{*},\beta^{*}\}$, so when\\
\begin{align}
H_D(\{\alpha^{*},\beta^{*}\}) = { \left(\frac{\zeta^{'}}{\zeta} \right) }^{'}(1+\alpha^{*}+\beta^{*}).
\end{align}
Now $\left(\frac{\zeta^{'}}{\zeta} \right)^{'}(1+x) = \frac{1}{x^2}+ O\left(1\right)$ and $A_D\left(W_r\right)$ is analytic for all $W_r$.
So
\begin{align}
\widetilde{H_D}\left(W_r\right) =  \begin{cases}
\frac{1}{\left(\alpha^{*} + \beta^{*}\right)^2} + O(1) & W = \{\alpha^{*},\beta^{*}\}\\
O(1) & otherwise.
\end{cases}
\end{align}
Hence the requirements of {\bf P1} hold.

{\bf P2}: Now we consider the case when $\alpha^{*} \in D, \beta^{*} \in A/D$. Then $\widetilde{Q\left(D\right)}$ is independent of $\beta^{*}$ and regular as $\alpha^{*} \rightarrow \beta^{*}$. The only pole in $\widetilde{H_D}\left(W_r\right)$ comes when $W_r=\{\beta^{*}\}$. Then
\begin{align}
H_D\left(\{\beta^{*}\}\right) = \sum_{\delta \in D}{\frac{\zeta^{'}}{\zeta}\left(1+\beta^{*} - \delta\right)-\frac{\zeta^{'}}{\zeta}\left(1+\beta^{*} + \delta\right)- \frac{\zeta^{'}}{\zeta}\left(2\beta^{*}\right)}.
\end{align}
When $\delta = \alpha^{*}$, $-\frac{\zeta^{'}}{\zeta}(\beta^{*} + \delta)$ is a simple pole with a residue of $1$. $A_D\left(W_r\right)$ has no poles so
\begin{align}
\widetilde{H_D}\left(W_r\right)= \begin{cases}
\frac{1}{\alpha^{*} + \beta^{*}} +O(1) & W_r= \{\beta^{*}\}\\
O(1)& otherwise.
\end{cases}
\end{align}
This satisfies the conditions of {\bf P2}.

{\bf P3}: Now we consider the case where $\alpha^{*} \in A/D, \beta^{*} \in D$. This is similar to the previous case,  $\widetilde{Q\left(D\right)}$ is regular as $\alpha^{*} = -\beta^{*}$ and the only pole in $\widetilde{H_D(W_r)}$ is when $W_r=\{\alpha^{*}\}$.
\begin{align}
\widetilde{H_D}(W_r)=  \begin{cases}
\frac{1}{\alpha^{*} + \beta^{*}} +O(1) & W_r= \{\alpha^{*}\}\\
O(1)& otherwise.
\end{cases}
\end{align}
This shows that $\widetilde{Q\left(D\right)}$ and $\widetilde{H_D}\left(W_r\right)$ satisfy {\bf P3}.

{\bf P4}: We need to consider when $\alpha^{*}, \beta^{*} \in D$. Recall that $D^{'} = D/\{\alpha^{*},\beta^{*}\}$. We already have expansions for $H_D\left(W_r\right), A_D(W_r)$ and $A_E(D,D,D)=:A_E(D)$ in Lemmas~\ref{ADresidue}-~\ref{AEresidue}. We want to find an expansion for $\widetilde{Q\left(D\right)}$ around $\alpha^{*}=\beta^{*}$.
\begin{align}
\begin{split}
Q\left(D\right)&:= \left(-1\right)^{|D|}\left(\frac{\sqrt{M}|d|}{2\pi}\right)^{-2\sum_{\delta \in D}{\delta}} \prod_{\delta \in D}{\frac{\Gamma\left(1- \delta\right)}{\Gamma\left(1+ \delta\right)}}\\
& \qquad \times \sqrt{\frac{Z_\zeta\left(D^{-},D^{-}\right)Z_\zeta\left(D,D\right)Y_\zeta\left(D\right)}{Z_\zeta^{\dag}\left(D^{-},D\right)^2 Y_\zeta\left(D^{-}\right)}}
\end{split}\\
\begin{split}
&= Q\left(D^{'}\right)\left(\frac{\sqrt{M}|d|}{2\pi}\right)^{-2\left(\alpha^{*}+\beta^{*}\right)}\frac{\Gamma\left(1- \alpha^{*}\right)\Gamma\left(1- \beta^{*}\right)}{\Gamma\left(1+ \alpha^{*}\right)\Gamma\left(1+ \beta^{*}\right)}\\
    & \qquad \times \frac{\zeta\left(1+ 2\alpha^{*}\right)\zeta\left(1+2 \beta^{*}\right)\zeta\left(1+ \alpha^{*} + \beta^{*}\right)\zeta\left(1- \alpha^{*}- \beta^{*}\right)}{\zeta\left(1+ \alpha^{*}- \beta^{*}\right)\zeta\left(1-\alpha^{*}+ \beta^{*}\right)}\\
    &\qquad \times \prod_{\delta \in D} \frac{\zeta\left(1- \alpha^{*}-\delta\right)\zeta\left(1+ \alpha^{*}+ \delta\right)\zeta\left(1-\beta^{*}- \delta\right)\zeta\left(1+\beta^{*}+\delta\right)}{\zeta\left(1-\alpha^{*}+\delta\right)\zeta\left(1+\alpha^{*}-\delta\right)\zeta\left(1+ \beta^{*}-\delta\right)\zeta\left(1-\beta^{*}+\delta\right)}
\end{split}\\
\begin{split}
&=Q\left(D^{'}\right)\left(\frac{-1}{\left(\alpha^{*}+\beta^{*}\right)^2}- \frac{1}{12} + O\left(|\alpha^{*}+\beta^{*}|\right)\right)\\
& \qquad \times \left(1- \left(\alpha^{*}+\beta^{*}\right)\left[H_{D^{'}}\left(\{\beta^{*}\}\right)+H_{D^{'}}\left(\{-\beta^{*}\}\right)\vphantom{\sum_d}\right.\right. \\
& \qquad \qquad +2\log\left(\frac{\sqrt{M}|d|}{2\pi}\right)
\left. + \frac{\Gamma^{'}}{\Gamma}\left(1+\beta^{*}\right)+\frac{\Gamma^{'}}{\Gamma}\left(1-\beta^{*}\right)\right] \\
& \qquad \qquad \left.\vphantom{\sum_d}+ O\left(|\alpha^{*}+\beta^{*}|^2\right)\right).
\end{split}
\end{align}

Combining this with the results for $H_D\left(W_r\right), A_D(W_r)$ and $A_E(D,D,D)=A_E(D)$.

\begin{align}
\widetilde{Q(D)}&= Q(D)A_E\left(D\right)\\
\begin{split}
&= Q\left(D^{'}\right)\left(\frac{-1}{\left(\alpha^{*}+\beta^{*}\right)^2}- \frac{1}{12} + O\left(|\alpha^{*}+\beta^{*}|\right)\right)\\
& \qquad \times \left(\left[ 1- \left(\alpha^{*}+\beta^{*}\right)\left[\vphantom{\frac{\psi^{'}}{\psi}} H_{D^{'}}\left(\{\beta^{*}\}\right)+H_{D^{'}}\left(\{-\beta^{*}\}\right) \right. \right. \right.\\
& \qquad \qquad \left.\left. \left.-\frac{\psi^{'}}{\psi}\left(\frac{1}{2}+\beta^*,\chi_d\right)\right]\right] + O\left(|\alpha^{*}+\beta^{*}|^2\right)\right)\\
&\qquad \times A_E(D^{'})\left(1+ \left(\alpha^{*}+\beta^{*}\right)\left[-\widetilde{H_{D^{'}}}\left(\{\beta^{*}\}\right) + H_{D^{'}}\left(\{\beta^{*}\}\right)\right.\right.\\
& \left.\left.\qquad \qquad -\widetilde{H_{D^{'}}}\left(\{-\beta^{*}\}\right)+H_{D^{'}}\left(\{-\beta^{*}\}\right) \right]+O\left(|\alpha^{*}+\beta^{*}|^2\right)\right)
\end{split}\\
\begin{split}
&= \frac{-1}{\left(\alpha^{*}+\beta^{*}\right)^2}\widetilde{Q(D^{'})}\\
& \qquad - \frac{\widetilde{Q(D^{'})}}{\left(\alpha^{*}+\beta^{*}\right)}\left(-\widetilde{ H_{D^{'}}}( \{\beta^{*}\})-\widetilde{H_{D^{'}}}(\{-\beta^{*}\}) - f\left(\beta^{*}\right)\right)\\
& \qquad +O(1)
\end{split}
\end{align}

Combining this with Lemma~\ref{HDtilderesidue}, the conditions of {\bf P4} are met. Hence by Lemma~\ref{lem:P}, we have proved the first part of the Residue Theorem~\ref{LFresiduef}.

For the second part, we need to check that $\widetilde{Q(D)}$ and $\widetilde{H_D}(W_r)$ meet the conditions {\bf R1} and {\bf R2} from Definition~\ref{def:PropR}.

{\bf R1}: Suppose $\alpha^{*} \in A/D$,then $\widetilde{Q(D)}$ is independent of $\alpha^{*}$. The only pole comes from $W_r = \{\alpha^{*}\}$, where
\begin{align}
H_{D}(\{\alpha^{*}\}) = \sum_{d \in D}{\frac{\zeta^{'}}{\zeta}(1+\alpha^{*}-d)-\frac{\zeta^{'}}{\zeta}(1+\alpha^{*}+d)-\frac{\zeta^{'}}{\zeta}(1+2\alpha^{*})}.
\end{align}
The residue as $x \rightarrow 0$ of $\frac{\zeta^{'}}{\zeta}\left(1+x\right)$ is 1.
So
\begin{align}
\widetilde{H_D}(W_r)&=\begin{cases}
\frac{-1}{2 \alpha^{*}} + O(1) & W= \{\alpha^{*}\}\\
O(1)& otherwise
\end{cases}
\end{align}
Thus the condition {\bf R1} is met.

{\bf R2}: Now we consider the case where $\alpha^{*} \in D$. Then $H_D(A/D)$, $A_E\left(D\right)$ and $A_D\left(W_r\right)$ are regular at $\alpha^{*} =0$,and $\widetilde{H_D}(A/D)|_{\alpha^{*}=0}= \widetilde{H_{D^{'}}}(A/D)$.
\begin{align}
\begin{split}
\widetilde{Q(D)} &= \widetilde{Q(D^{'})}\left(\frac{2 \pi}{\sqrt{M}|d|}\right)^{-2\alpha^{*}}\frac{\Gamma\left(1- \alpha^{*}\right)}{\Gamma\left(1+ \alpha^{*}\right)}\\
& \qquad \times \zeta(1+2\alpha^{*})\prod_{d \in D}{\frac{\zeta(1+\alpha^{*}+d)\zeta(1-\alpha^{*}-d)}{\zeta(1+\alpha^{*}-d)\zeta(1+d-\alpha^{*})}}.
\end{split}
\end{align}
The only pole comes from the $\zeta(1+ 2\alpha^{*})$, which has a residue of $\frac{1}{2}$. This gives
\begin{align}
\widetilde{Q(D)} &= \frac{-1}{2 \alpha^{*}}\widetilde{Q(D^{'})}+O(1)
\end{align}
as required. Hence we can apply Lemma~\ref{lem:R} and conclude the proof for the second half of the Residue Lemma~\ref{LFresiduef}.

\end{proof}

\subsection{$n$-level Density of Zeros of Families of $L$-functions of Elliptic Curves}\label{sub:LFncorr}

In this section we will express the $n$-level density of zeros of families of $L$-functions of elliptic curves as contour integrals on the complex plane. We will define the $n$-level density function, $S_n^{E}(f)$, for the zeros for a family of twisted $L$-functions with even functional equation.
\begin{align}
S_n^{E}(f) = \sum_{\substack{0 < d \leq X \\ \chi_d(-M)\omega_E = +1}}{\sum_{t_1, \cdots, t_n >0}{f(\gamma_{t_1,d}, \cdots, \gamma_{t_n,d})}}
\end{align}
where $\gamma_{i,d}$ is the height of the $i^{th}$ zero of $L_E(s,\chi_d)$ on the critical line, above the real axis. Note that the sum over the zeros in $S_n^E(f)$ is not restricted to a sum over distinct indices at this stage.

\begin{theorem}\label{thm:LF1}
Let $C_{-}$ denote the path from $ - \delta - \infty i$ up to $- \delta + \infty i$ and let $C_{+}$
denote the path from $  \delta - \infty i$ up to $ \delta + \infty i$, where $\delta$ is a small positive number. Let $f$ be a
holomorphic function of $n$ variables such that
\begin{align}
f\left(x_{j_1}, \cdots, x_{j_n}\right) = f\left(\pm x_{j_1}, \cdots, \pm x_{j_n}\right).
\end{align}
Then
\begin{align}\label{LFnostar}
\begin{split}
&2^n \left(2 \pi i \right)^n S_n^E(f)= \\
    & \qquad \sum_{K \cup L \cup M = \left\{1, \cdots,n\right\}}  \int_{C_+^K} {\int_{C_{-}^{L \cup M}}{\left(-1\right)^{|M|}J_E\left(z_K \cup - z_L, z_M\right)}}\\
        & \qquad \qquad \times f\left(iz_1, \cdots,iz_n\right)dz_1 \cdots d z_n
\end{split}
\end{align}
where $z_K = \left\{z_k: k \in K\right\}$, $-z_L = \left\{-z_l: l \in L\right\}$,
$\int_{C_+^K} {\int_{C_-^{L \cup M}}}$ means we are integrating all the variables in $z_K$ along the
$C_+$ path and all others along the $C_-$ path and \\
\begin{align} \label{def:LFJ}
J_E\left(A,B\right) = \sum_{\substack{0 < d \leq X \\ \chi_d(-M)\omega_E = +1}}{\prod_{\alpha \in A}{\frac{L_E^{'}}{L_E}(\frac{1}{2}+\alpha, \chi_d)}}\prod_{\beta \in B}{\frac{\psi^{'}}{\psi}(\frac{1}{2}+\beta, \chi_d)}
\end{align}
where $\psi\left(s, \chi_d\right)=\chi_d(-M) \omega_E \left(\frac{2 \pi}{\sqrt{M}|d|}\right)^{2s-1} \frac{\Gamma(\frac{3}{2}-s)}{\Gamma(s+ \frac{1}{2})}$. Here $K,L,M$ are finite sets of integers and $A,B$ are finite sets of complex numbers. The sum over $d$ is over fundamental discriminants.
\end{theorem}

This follows immediately from Cauchy's Residue Theorem, in a similar manner to the proof of Theorem \ref{thm:RMortho1}.

Using Theorem~\ref{LFJconj} and Theorem~\ref{LFresiduef}, we can replace our $J_E(A,B)$ with $J^{*}_E(A,B)+O(X^{1/2+\epsilon})$, where we define
\begin{align}
\begin{split}
J^{*}_E\left(A,B\right) =& \sum_{\substack{0<d \leq X \\ \chi_d\left(-M\right)\omega_E = +1}}{\prod_{\beta \in B}{\frac{\psi^{'}}{\psi}(\frac{1}{2}+\beta, \chi_d)}\sum_{D \subseteq A}{{\left(\frac{\sqrt{M}|d|}{2 \pi}\right)}^{- \sum\limits_{\delta \in D}{2\delta}}}}\\
& \times \prod_{\delta \in D}{\frac{\Gamma\left(1- \delta\right)}{\Gamma\left(1+ \delta\right)}} \sqrt{\frac{Z_{\zeta}\left(D^{-},D^{-}\right)Z_{\zeta}\left(D,D\right)Y_{\zeta}\left(D\right)}{Z_{\zeta}^{\dag}\left(D^{-},D\right)^{2}Y_{\zeta}\left(D^{-}\right)}}\\
&\times \left(-1\right)^{|D|}A_E\left(D,D,D\right)\\
&\times \sum_{\substack{A \backslash D=\\W_1 \cup \cdots \cup W_R}}{\prod_{r=1}^{R}{\widetilde{H_D}\left(W_r\right)}}\\
\end{split}\label{eq:JstarE2}
\end{align}
where $\widetilde{H_D}$ etc are as described in Theorem~\ref{LFJconj}. Our previous definition of $J^{*}_E(A)$ is the equivalent of $J^{*}_E(A, \emptyset)$ in our new notation. All our theorems using the original definition have trivial extensions to encompass this new notation.

Note, however,  that this additional factor means that the residues found here may be more complicated than those in the random matrices case.  In particular we need to consider whether the $\frac{\psi^{'}}{\psi}\left(\frac{1}{2}+\beta^*, \chi_d\right)$ factors add additional poles, or affect the residues at the current ones when contemplating the multiple integrals in what follows.  Recall that
\begin{align}
\frac{\psi^{'}}{\psi}\left(\frac{1}{2}+\beta^*,\chi_d\right) = 2\log\left(\frac{2 \pi}{\sqrt{M}|d|}\right) - \frac{\Gamma^{'}}{\Gamma}\left(1- \beta^{*}\right) - \frac{\Gamma^{'}}{\Gamma}\left(1+ \beta^{*}\right)
\end{align}
Clearly the constant term is not a problem, so we only need to focus on the digamma terms. These have poles only at $0,-1,-2$ etc and are meromorphic elsewhere\cite{abramowitz1964handbook} so these are simple to deal with as long as $|\Re\left(\beta^{*}\right)| < \frac{1}{2}$. We are already assuming this in Equation~\eqref{eq:restrictionsE} so it turns out the additional factor is harmless.

We now have that
\begin{align}
\begin{split}
&\left(2\pi i\right)^n 2^n S_n^E\left(f\right)\\
    & \qquad = \sum_{\substack{K \cup L \cup M =\\ \{1, \cdots, n\}}}{\left(-1\right)^{|M|}\int_{C^{K}_+}\int_{{C}^L_{-}}\sum_{\substack{0 < d \leq X\\ \chi_d(-M) \omega_E = 1}}J^{*}_{E}\left(z_K \cup -z_L, z_M\right)} \\
        & \qquad \qquad f\left(iz_1, \cdots, iz_n\right) d z_1 \cdots d z_n + O\left(X^{1/2+\epsilon}\right)\label{eq:LFinttomove}
\end{split}
\end{align}
Note that we have pulled the error term outside of the integral as the size of the error is unaffected by integration with respect to $z_i$.

Here we have assumed that $f$ decays fast enough that the integrals are negligible outside the range given for  the imaginary parts of the parameters in the Ratios Conjecture, Conjecture \ref{thm:LFRat2}.  See the one-level density section of \cite{conrey2007applications} for more details, but since it is not at all clear over what range the Ratios Conjectures hold, we do not pursue this further here.

In line with the calculations with orthogonal matrices in Section~\ref{sub:ncorr}, we will move all the contours one by one onto the imaginary axis. We recall from Section~\ref{sec:Eresidue} that if $\alpha, \beta \in A$
\begin{align}
\begin{split}
&\Res_{\alpha \rightarrow - \beta}J^{*}_{E}(A,B) \\
& \qquad = J^{*}_{E}(A^{'}\cup\{\beta\},B) +J^{*}_{E}(A^{'}\cup\{-\beta\},B) -J^{*}_{E}(A^{'},B \cup \{\beta\})
\end{split}
\end{align}
where $A^{'} = A\backslash \{\alpha, \beta\}$.

We define, for given $n$, $0 \leq R \leq n$, \begin{align} \label{def:LFsum notation}
\sideset{}{^{n,R}}\sum = \sum^{\infty}_{\substack{j_1, \cdots, j_n =1 \\ j_i \neq j_k \forall i,k>R}}.
\end{align}
For fixed sets $K,L,M$ such that the disjoint union $K \cup L \cup M = \{1, \cdots, n\}$ let $I^{n,R}_{f,K,L,M}$ be the integral in Equation~\eqref{eq:LFinttomove}, excluding the error term, with $N-R$ of the integrals shifted onto the imaginary axis. All the integrals on the imaginary axis are principal value integrals.
\begin{align}
\begin{split}
I^{n,R}_{f,K,L,M}=& \int^{i\infty}_{-i\infty}{ \cdots \int^{i\infty}_{-i\infty}{ \int_{C^{K \cap \{1, \cdots, R\}}_+}{\int_{C^{(L  \cup M) \cap \{1, \cdots, R\}}_-}{J_{E}^{*}(z_K \cup -z_L,z_M)}}}}\\
& \qquad \qquad \times f(iz_1, \cdots, iz_n) dz_1 \cdots dz_R d z_{R+1} \cdots dz_n.
\end{split}
\end{align}

We can express Equation~\eqref{eq:LFinttomove} in the new notation.
\begin{align}
\begin{split}
(2\pi i)^n 2^n &\sum_{\substack{0 < d\leq X\\ \chi_d(-M) \omega_E = +1}}{\sideset{}{^{n,n}}\sum{f(\gamma_{j_1,d}, \cdots, \gamma_{j_n,d}) }}\\
 & = \sum_{K \cup L \cup M = \{1, \cdots, n\}}{(-1)^{|M|}I^{n,n}_{f,K,L,M}+ O\left(X^{\frac{1}{2}+\epsilon}\right)}
 \end{split}
\end{align}
We will now prove
\begin{theorem} \label{thm:LFmain}
Assume Conjecture~\ref{thm:LFRat2}. With the notation defined above, $0\leq R \leq n$, and $d$ summing over fundamental discriminants,
\begin{align}
\begin{split}
(2 \pi i)^n 2^n &\sum_{\substack{0 < d\leq X\\ \chi_d(-M) \omega_E = +1}}{\sideset{}{^{n,R}}\sum{f(\gamma_{j_1,d}, \cdots, \gamma_{j_n,d}) }} \\
&= \sum_{K \cup L \cup M = \{1, \cdots, n\}}{(-1)^{|M|}I^{n,R}_{f,K,L,M}}+ O\left(X^{\frac{1}{2}+\epsilon}\right). \label{LFstatement}
\end{split}
\end{align}
\end{theorem}

\begin{proof}

This proof follows identical steps to that of  Theorem \ref{thm:orthomain}, so we will refrain from repeating all the explanation and instead just reference the relevant point in that proof of the $SO(2N)$ case.
\\

The first step, in analogy to the steps leading up to (\ref{eq:n1R0}),  is to confirm the case $n=1$, $R=0$:
\begin{align}
 &\left(2\pi i\right)2 \sum_{\substack{0 < d\leq X\\ \chi_d(-M) \omega_E = +1}}{\sideset{}{^{1,0}}\sum{f\left(\gamma_{j_1,d}\right) }}\\
\begin{split}
& \qquad = \int_{C^+}{J_{E}^{*}(z, \emptyset)f(-iz) d z}+ \int_{C^-}{J_{E}^{*}(-z, \emptyset)f(-iz) d z}\\
& \qquad \qquad - \int_{C^-}{J_{E}^{*}(\emptyset,z)f(-iz) d z}+ O\left(X^{\frac{1}{2}+\epsilon}\right).
\end{split}\\
&=\sum_{K \cup L \cup M = \{1\}}{(-1)^{|M|}I^{1,0}_{f,K,L,M}}+ O\left(X^{\frac{1}{2}+\epsilon}\right).
\end{align}

We now consider the inductive step. Assume Equation~\eqref{LFstatement} is true for $n=p-1$ and $0 \leq R \leq p-1$ and consider the case when $n=p$, $0 \leq R \leq p$. We will proceed by induction on $R$. We have already proved that Equation~\eqref{LFstatement} holds if $R=p$ so we take that as the base case. Assume that Equation~\eqref{LFstatement} holds if $n=p$, and $R \geq S$ so that
\begin{align}
\begin{split}
&(2 \pi i)^p 2^p  \sum_{\substack{0 < d\leq X\\ \chi_d(-M) \omega_E = +1}}{\sideset{}{^{p,S}}\sum{f(\gamma_{j_1,d}, \cdots, \gamma_{j_p,d}) }} \\
& \qquad = \sum_{K \cup L \cup M = \{1, \cdots, p\}}{(-1)^{|M|}I^{p,S}_{f,K,L,M}}+ O\left(X^{\frac{1}{2}+\epsilon}\right)\\
& \qquad =\sum_{K \cup L \cup M = \{1, \cdots, p\}}{(-1)^{|M|}}\\
 &\qquad \times \int^{i\infty}_{-i\infty}{ \cdots \int^{i\infty}_{-i\infty}{ \int_{C^{K \cap \{1, \cdots, S\}}_+}{\int_{C^{(L \cup M) \cap \{1, \cdots, S\}}_-}{J_E^{*}(z_K \cup -z_L, z_M)}}}}\\
 & \qquad \qquad \times f(iz_1, \cdots, iz_p) dz_1 \cdots dz_S d z_{S+1} \cdots dz_p+ O\left(X^{\frac{1}{2}+\epsilon}\right). \label{LFinduct}
 \end{split}
\end{align}

The residue structure is identical to the case in Theorem \ref{thm:orthomain}.  Thus the only new information we need for this case is that for fixed sets $K,L,M$ such that the disjoint union $K \cup L \cup M = \{1, \cdots, p\}$, if $S \in K$, then the residue of $J_E^{*}(z_K \cup -z_L,z_M)f(iz_1, \cdots, iz_p)$  at $z_S = \pm z_t$ is
\begin{align}
\begin{split}
& i \pi \left(-J_E^{*}\left(z_{K^{'}}\cup -z_{L^{'}}, z_{M \cup \{t\}}\right)+J_E^{*}\left(z_{K^{'}}\cup -z_{L^{'} \cup \{t\}}, z_M\right)\right.\\
& \qquad \qquad \left.+J^{*}\left(z_{K^{'}\cup \{t\}}\cup -z_{L^{'} }, z_M\right)\right)\\
& \qquad \times f(iz_1, \cdots,iz_{S-1}, iz_t, iz_{S+1}, \cdots, iz_p)
\end{split}
\end{align}
where $K^{'} = K \cap (A - \{S,t\}), L^{'} = L \cap (A - \{S,t\})$.

Thus, in analogy to equation (\ref{eq:repeatf}),  we find that the main term of (\ref{LFinduct}) is equivalent to
\begin{align} \label{eq:LFrepeatf}
\begin{split}
&\sum_{K \cup L \cup M = \{1, \cdots, p\}}{(-1)^{|M|}}I^{p,S-1}_{f,K,L,M}\\
& \qquad + 4 \pi i\sum_{t=S+1}^{p}{\sum_{\substack{K^{'}\cup L^{'}\cup M=\\ \{1, \cdots, p\}-\{S,t\}}}}{(-1)^{|M|}}\\
& \qquad \qquad \int^{i\infty}_{-i\infty}{ \cdots \int^{i\infty}_{-i\infty}{ \int_{C^{K^{'} \cap \{1, \cdots, S-1\}}_+}{\int_{C^{(L^{'} \cup M) \cap \{1, \cdots, S-1\}}_-}{ \left( -J_E^{*}(z_{K^{'}}\cup -z_{L^{'}}, z_{M \cup \{t\}})\right.}}}}\\
& \qquad \qquad \qquad \qquad \left.+J_E^{*}(z_{K^{'}}\cup -z_{L^{'} \cup \{t\}},z_M)+J_E^{*}(z_{K^{'}\cup \{t\}}\cup -z_{L^{'} },z_M)\right)\\
& \qquad \qquad \qquad \times f(iz_1, \cdots,z_{S-1}, z_t, z_{S+1}, \cdots iz_p) dz_1 \cdots dz_{S-1} d z_{S+1} \cdots dz_p.
\end{split}
\end{align}
All of the unions over sets are disjoint unions.

Defining $g$ as at (\ref{eq:g}),  we can rewrite Equation~\eqref{eq:LFrepeatf}
\begin{align}
\begin{split}
&\sum_{\substack{K \cup L \cup M =\\ \{1, \cdots, p\}}}(-1)^{|M|}I^{p,S}_{f,K,L,M}\\
& \qquad = \sum_{\substack{K \cup L \cup M =\\ \{1, \cdots, p\}}}{(-1)^{|M|}}I^{p,S-1}_{f,K,L,M}\\
& \qquad \qquad + 4 \pi i\sum_{t=S+1}^{p}{\sum_{\substack{K \cup L \cup M =\\ \{1, \cdots, p-1\}}}}{(-1)^{|M|}I^{p-1,S-1}_{g_{t,S},K,L,M}}.
\end{split}
\end{align}
By our inductive hypothesis on $n$, Equation~\eqref{LFinduct} holds for $n=p-1$ and $R=S-1$.
\begin{align}
\begin{split}
&\sum_{\substack{K \cup L \cup M =\\ \{1, \cdots, p\}}}(-1)^{|M|}I^{p,S}_{f,K,L,M}\\
& \qquad = \sum_{\substack{K \cup L \cup M =\\ \{1, \cdots, p\}}}{(-1)^{|M|}}I^{p,S-1}_{f,K,L,M}\\
& \qquad \qquad + 4 \pi i\sum_{t=S+1}^{p}{(2 \pi i)^{p-1}2^{p-1} \sum_{\substack{0<d \leq X\\ \chi_d(-M)\omega_E =1 }}{\sideset{}{^{p-1,S-1}}\sum {g_{t,S}\left(\gamma_{j_1,d}, \cdots,\gamma_{j_{p-1},d} \right)}}}. \label{LFfour1}
\end{split}
\end{align}

The equivalent of (\ref{four2}) results in the left hand side of Equation~\eqref{LFinduct} being written as
\begin{align}
\begin{split}
& \left(2 \pi i\right)^p 2^p \sum_{\substack{0 < d\leq X\\ \chi_d(-M) \omega_E = +1}}{\sideset{}{^{p,S}}\sum{f\left(\gamma_{j_1,d}, \cdots, \gamma_{j_p,d}\right)}}\\
& \qquad = (2 \pi i)^p 2^p   \sum_{\substack{0 < d\leq X\\ \chi_d(-M) \omega_E = 1}}{\sideset{}{^{p,S-1}}\sum{f(\gamma_{j_1,d}, \cdots, \gamma_{j_p,d}) }} \\
& \qquad \qquad +  (2 \pi i)^p 2^p  \sum\limits_{\substack{0 < d\leq X\\ \chi_d(-M) \omega_E = +1}}{\sum\limits_{t=S+1}^{p}{\sideset{}{^{p-1,S-1}}\sum{g_{t,S}(\gamma_{j_1,d}, \cdots, \gamma_{j_{p-1},d})}}}. \label{LFfour2}
\end{split}
\end{align}
Equating Equations ~\eqref{LFfour1} and \eqref{LFfour2}, we see that
\begin{align}
\begin{split}
&\sum_{K \cup L \cup M = \{1, \cdots, p\}}{(-1)^{|M|}}I^{p,S-1}_{f,K,L,M}+ O\left(X^{\frac{1}{2}+\epsilon}\right)\\
& \qquad = (2 \pi i)^p 2^p   \sum_{\substack{0 < d\leq X\\ \chi_d(-M) \omega_E = +1}}{\sideset{}{^{p,S-1}}\sum{f(\gamma_{j_1,d}, \cdots, \gamma_{j_p,d})}}.
\end{split}
\end{align}

Thus we have completed the inductive step in $R$ by showing that Equation~\eqref{LFstatement} holds for $n=p, R=S-1$ if it holds for $n=p, R=S$. As we have already shown Equation~\eqref{LFstatement} holds for $n=p, R=p$, we have now shown that it holds for all $R$ when $n=p$. \\

This completes the inductive step in $n$, as we have shown that Equation~\eqref{LFstatement} holds for all $0 \leq R \leq n$ when $n=p$ if it holds for all $0 \leq R \leq n$ and $n=p-1$. We already proved that it holds for all $0 \leq R \leq 1$ when $n=1$, so by induction Equation~\eqref{LFstatement} is true for all $n$ and $0 \leq R \leq n$.
\end{proof}

We can state our final result for the $n$-level density of zeros of families of twisted elliptic curve $L$-functions.
\begin{theorem}[$n$-level Density of Zeros of Families of Twisted Elliptic Curve $L$-functions] \label{thm:LFNcorr}
Assume Conjecture~\ref{thm:LFRat2} and the Riemann Hypothesis. Then
\begin{align}
\begin{split}
&\sum_{\substack{0 < d\leq X\\ \chi_d(-M) \omega_E = +1}}{\sideset{}{^{n,0}}\sum{ f(\gamma_{j_1,d}, \cdots, \gamma_{j_n,d}) }}\\
& \qquad = \frac{1}{(2 \pi )^n}\sum_{\substack{K \cup L \cup M\\= \{1, \cdots, n\}}}{(-1)^{|M|}\left(\int_{0}^{\infty}\right)^n J_{E}^{*}(-iz_K \cup i z_L, -iz_M)}\\
& \qquad \qquad \times f(z_1, \cdots, z_n)dz_1 \cdots dz_n + O\left(X^{\frac{1}{2}+\epsilon}\right),
\end{split}
\end{align}
where $J_{E}^{*}(-iz_K \cup i z_L, -iz_M)$  is defined in (\ref{eq:JstarE2}) and the sum notation is defined at (\ref{def:sum notation}) and implies the sum is over zeros with distinct indices.  The sum over $d$ is over fundamental discriminants.
\end{theorem}
\begin{proof}
This follows easily from the previous theorem. We have proven that
\begin{align}
\begin{split}
&(2 \pi i)^n 2^n \sum_{\substack{0 < d\leq X\\ \chi_d(-M) \omega_E = +1}}{\sideset{}{^{n,0}}\sum{f(\gamma_{j_1,d}, \cdots, \gamma_{j_n,d})}} \\
&\qquad= \sum_{K \cup L \cup M = \{1, \cdots, n\}}{(-1)^{|M|}I^{n,0}_{f,K,L,M}}+ O\left(X^{\frac{1}{2}+\epsilon}\right)
\end{split}\\
\begin{split}
&\qquad = \sum_{K \cup L \cup M = \{1, \cdots, n\}}{(-1)^{|M|}}\\
& \qquad \qquad \times \left(\int_{-i \infty}^{i \infty}\right)^n J_E^{*}(z_K \cup -z_L, z_M)f(iz_1, \cdots, iz_n)d z_1 \cdots d z_n+ O\left(X^{\frac{1}{2}+\epsilon}\right)
\end{split}
\end{align}
where the integral here is a principal value integral.
\begin{align}
\begin{split}
&(2 \pi i)^n 2^n \sum_{\substack{0 < d\leq X\\ \chi_d(-M) \omega_E = +1}}{\sideset{}{^{n,0}}\sum{f(\gamma_{j_1,d}, \cdots, \gamma_{j_n,d})}} \\
& \qquad= i^n\sum_{\substack{K \cup L \cup M \\= \{1, \cdots, n\}}}(-1)^{|M|}\left(\int^{\infty}_{-\infty}\right)^n J_{E}^{*}(-iz_K \cup iz_L, -iz_M)f(z_1, \cdots, z_n)d z_1 \cdots d z_n+ O\left(X^{\frac{1}{2}+\epsilon}\right)
\end{split}\\
&\qquad =i^n \left(\int^{\infty}_{-\infty}\right)^n \sum_{\substack{K \cup L \cup M \\= \{1, \cdots, n\}}}(-1)^{|M|} J_{E}^{*}(-iz_K \cup iz_L, -iz_M)f(z_1, \cdots, z_n)d z_1 \cdots d z_n+ O\left(X^{\frac{1}{2}+\epsilon}\right).
\end{align}
Now
\begin{align}
\begin{split}
&\sum_{K \cup L \cup M = \{1, \cdots, n\}}(-1)^{|M|} J_{E}^{*}(-iz_K \cup iz_L, -iz_M) \\
& \qquad = \sum_{K \cup L \cup M = \{1, \cdots, n\}}(-1)^{|M|} J_{E}^{*}(iz_K \cup -iz_L, iz_M)
\end{split}
\end{align}
 due to $\frac{\psi^{'}}{\psi}(\frac{1}{2}+\alpha,\chi_d)=\frac{\psi^{'}}{\psi}(\frac{1}{2}-\alpha,\chi_d)$.  In addition, $f\left(\theta_{j_1}, \cdots, \theta_{j_n}\right) = f\left(\pm \theta_{j_1}, \cdots, \pm
\theta_{j_n}\right)$. Thus each integral from $- \infty$ to $\infty$ is double the integral from $0$ to $\infty$.
\begin{align}
\begin{split}
&(2 \pi )^n 2^n \sum\limits_{\substack{0 < d\leq X\\ \chi_d(-M) \omega_E = +1}}{\sideset{}{^{n,0}}\sum{f(\gamma_{j_1,d}, \cdots, \gamma_{j_n,d}) }} \\
& \qquad = 2^n \left(\int^{\infty}_{0}\right)^n \sum_{K \cup L \cup M = \{1, \cdots, n\}}(-1)^{|M|} J_E^{*}(-iz_K \cup iz_L, -i z_M)\\
& \qquad \qquad \times f(z_1, \cdots, z_n) d z_1 \cdots d z_n+ O\left(X^{\frac{1}{2}+\epsilon}\right).
\end{split}
\end{align}

All that is now required is to show that there are no poles on the path of integration. As $f$ is holomorphic, we just need to check that
\begin{align}
\sum_{K \cup L \cup M = \{1, \cdots, n\}}(-1)^{|M|} J_E^{*}(-iz_K \cup iz_L, -iz_M)
\end{align}
has no poles at $z_1 =z_2$ (this argument applies equally well to any other pair of $z$'s by symmetry). We do not need to check for a pole at $z_1 =-z_2$ as $z_i \geq 0$, for $ 0 \leq i \leq n$ on the path of integration. A given $J_{E}^{*}(-iz_K \cup iz_L, -iz_M)$ has a pole at $z_1 =z_2$  if $1 \in K, 2 \in L$ or $1 \in L, 2 \in K$. So
\begin{align}
\begin{split}
 \Res_{z_1 = z_2}&\left(\sum_{K \cup L \cup M = \{1, \cdots, n\}}(-1)^{|M|} J_{E}^{*}(-iz_K \cup iz_L, -iz_M)\right)\\
& = \sum_{K \cup L \cup M = \{3, \cdots, n\}}(-1)^{|M|} \Res_{z_1 = z_2}\left(J_{E}^{*}(-iz_{K \cup \{z_1\}} \cup iz_{L \cup \{z_2\}},-iz_M) \right.\\
& \qquad \qquad \qquad \left. + J_{E}^{*}(-iz_{K \cup \{z_2\}} \cup iz_{L \cup \{z_1\}}, -i z_M)\right)
\end{split}\\
&=0
\end{align}
as $Res_{x=y}f(x,y)=-Res_{x=y} f(-x,-y)$. The same argument holds for all poles at $z_i= \pm z_j$ for all $0 \leq i,j \leq n$. Therefore there are no poles on the path of integration, for the same reason as in Section \ref{sub:ncorr}
\end{proof}

\section{Zero Statistics of Quadratic Dirichlet $L$-functions} \label{sec:dir}

Let $\chi_d$ be a real primitive Dirichlet character. For any Dirichlet character of modulus $d$, $\chi_d\left(n\right)$, we can define an $L$-function \begin{align}
L\left(s, \chi_d\right) = \sum_{n=1}^{\infty}{\frac{\chi_d\left(n\right)}{n^s}}.
\end{align}
$L\left(s, \chi_d\right)$ has a Euler product
\begin{align}
L\left(s, \chi_d\right) = \prod_{p}\left(1- \frac{\chi_d(p)}{p^s}\right)^{-1}.
\end{align}
If $\chi_d$ is primitive $L\left(s, \chi_d\right)$ has an analytic continuation to the whole complex plane and has a functional equation~\cite{HBPrimenumber}.

 When $\chi_d$ is real and $d>0$, the functional equation is
\begin{align}
L(s,\chi_d)=\psi_{DL}(s, \chi_d)L(1-s,\chi_d)=\left(\frac{\pi}{d}\right)^{s-1/2} \frac{\Gamma(\tfrac{1-s}{2})} {\Gamma(\tfrac{s}{2})} L(1-s,\chi_d).
\end{align}

We can further define a family of $L$-functions $D^{+} = \left\{L(s, \chi_d) : d>0  {\rm \; a \; fundamental \;discriminant}\right\}$. All sums over $0<d\leq X$ in this section include only fundamental discriminants.  We expect that the $n$-level density of the zeros of the $L$-functions in $D^+$, averaged over the family will correspond to that of the eigenangles of $USp(2N)$.

Our plan is therefore to follow the recipe of Conrey, Farmer and Zirnbauer \cite{conrey2007autocorrelatio} to express ratios of Dirichlet $L$-functions in appropriate terms and then mimic the calculations from Section~\ref{sub:symp} to determine the $n$-level density.

\subsection{Ratios of Quadratic Dirichlet $L$-functions}
The result from ``Autocorrelation of ratios of $L$-functions"~\cite{conrey2007autocorrelatio}, can be rewritten into set notation.

\begin{conj}[Ratios of Quadratic Dirichlet $L$-functions] \label{conj:DLratios}
Let $A= \{\alpha_1, \cdots, \alpha_K\}$ and $ B= \{\gamma_1, \cdots, \gamma_Q\}$ such that
\begin{align}
\begin{split}
-\frac{1}{4} < \Re\left(\alpha_i\right) < \frac{1}{4} \qquad &1 \leq i \leq K\\
\frac{1}{\log X} \ll \Re\left(\gamma_j\right)< \frac{1}{4} \qquad &1 \leq j \leq Q\\
\forall \epsilon >0,\; Im\left(\alpha_i\right), Im\left(\gamma_j\right)\ll X^{1 - \epsilon }  \qquad &  1 \leq i \leq K, 1 \leq j \leq Q
\end{split}
\end{align}
then
\begin{align}
\begin{split}
&\sum_{0 < d \leq X}{\frac{\prod_{k=1}^{K}{L\left(\frac{1}{2}+ \alpha_k, \chi_d\right)}}{\prod_{q=1}^{Q}{L\left(\frac{1}{2}+ \gamma_q, \chi_d\right)}}}\\
& \qquad = \sum_{0 < d \leq X}\sum_{F \subset A}{\left(\frac{d}{\pi}\right)^{- \sum_{f \in F}{f}}\prod_{f \in F}{\frac{\Gamma \left(\frac{1}{4}- \frac{f}{2}\right)}{\Gamma \left(\frac{1}{4}+ \frac{f}{2}\right)}}}\\
& \qquad \qquad \times Y_{DL}\left(A,B,F\right)A_{DL}\left(A,B,F\right) +  O\left(X^{\frac{1}{2}+ \epsilon}\right)
\end{split}
\end{align}
where $d$ is summed over fundamental discriminants  and $Y_{DL}$ and $ A_{DL}$ are defined for $ D \subset A$ as follows
\begin{align}
Y_{DL}\left(A,B,D\right) &= \sqrt{\frac{Z_{\zeta}\left(\left(A \backslash D\right) \cup D^{-},\left(A \backslash D\right) \cup D^{-} \right)Y_{\zeta}\left(\left(A \backslash D\right) \cup D^{-}\right)Z_{\zeta}\left(B,B\right)}{Z_{\zeta}\left(\left(A \backslash D\right) \cup D^{-},B \right)^2 Y_{\zeta}\left(B\right)}} \label{eq:DL1def1}\\
\begin{split}
A_{DL}\left(A,B, D\right) &=Y_{DL}\left(A,B,D\right)^{-1} \prod_{p}{\frac{1}{1+ \frac{1}{p}}\left(\frac{1}{2}\frac{V_{-}\left(B\right)}{V_{-}\left(\left(A \backslash D\right) \cup D^{-}\right)}+ \right. }\\
& \qquad {\left.\frac{1}{2}\frac{V_{+}\left(B\right)}{V_{+}\left(\left(A \backslash D\right) \cup D^{-}\right)} + \frac{1}{p}\right)} \label{eq:DL1def2}
\end{split}
\end{align}
 and
 \begin{align}
V_{+}(A)&= \prod_{\alpha \in A}{\left(1+\frac{1}{p^{\frac{1}{2}+ \alpha}}\right)}\label{def:V+}\\
V_{-}(A)&= \prod_{\alpha \in A}{\left(1-\frac{1}{p^{\frac{1}{2}+ \alpha}}\right)} \label{def:V-}\\
D^{-}&= \{- \delta: \delta \in D\}\\
Z_{\zeta}(A,B)&= \prod_{\substack{a \in A\\b \in B}}{\zeta\left(1+ a +b\right)},\\
Y_{\zeta}(A) &= \prod_{a \in A}{\zeta\left(1+2a\right)}.
\end{align}
\end{conj}

\subsection{Differentiating Ratios of Quadratic Dirichlet $L$-functions}
Following the calculation in Section~\ref{sub:Dirratio}, we now wish to calculate
\begin{align} \label{def:JDL}
J_{DL}\left(A\right) &= \sum_{0 < d \leq X}{\prod_{\alpha \in A}{\frac{L^{'}}{L}\left(\frac{1}{2}+ \alpha, \chi_d\right)}}
 = \left.\prod_{\alpha \in A}{\frac{\partial}{\partial \alpha}{R_{DL}\left(A,B\right)}}\right|_{B=A}.
\end{align}
where
\begin{align}
R_{DL}\left(A,B\right) = \sum_{0< d \leq X}{\frac{\prod_{k}{L\left(\frac{1}{2}+ \alpha_k, \chi_d\right)}}{\prod_{q}{L\left(\frac{1}{2}+ \gamma_q, \chi_d\right)}}}.
\end{align}
\begin{theorem}\label{thm:DL2}
Assume Conjecture~\ref{conj:DLratios}. Let $A$ be a finite set of complex numbers where
\begin{align}
\begin{split}
\frac{1}{\log X}\ll \Re\left(\alpha\right)< \frac{1}{4} &\qquad \forall \alpha \in A\\
\forall \epsilon >0, Im\left(\alpha\right), \ll X^{1 - \epsilon }& \qquad \forall \alpha \in A.
\end{split}
\end{align}
Then we have $J_{DL}(A) = J_{DL}^{*}(A)+ O\left(X^{\frac{1}{2}+ \epsilon}\right)$ where
\begin{align}
J_{DL}\left(A\right) &= \sum_{0 < d \leq X}{\prod_{\alpha \in A}{\frac{L^{'}}{L}\left(\frac{1}{2}+ \alpha, \chi_d\right)}}\\
\begin{split}
J_{DL}^{*}\left(A\right) &= \sum_{0 < d \leq X}{\sum_{D \subset A}{\left(\frac{d}{\pi}\right)^{- \sum_{\delta \in D}{\delta}}}}\prod_{\delta \in D}{\frac{\Gamma\left(\frac{1}{4} - \frac{\delta}{2}\right)}{\Gamma\left(\frac{1}{4}+\frac{\delta}{2}\right)}}\\
&\qquad \times \sqrt{\frac{Z_{\zeta}\left(D^{-}, D^{-}\right)Z_{\zeta}\left(D,D\right)Y_{\zeta}\left(D^{-}\right)}{Z_{\zeta}^{\dagger}\left(D^{-},D\right)^2Y_{\zeta}\left(D\right)}}\\
& \qquad \times \left(-1\right)^{|D|}A_{DL}\left(D,D,D\right)\\
& \qquad \times \sum_{\substack{A\backslash D =W_1 \cup \cdots  \cup W_R\\|W_i|\leq 2}}\prod_{r=1}^{R}{\widehat{H_{D}}\left(W_r\right)}
\end{split}
\end{align}
where
\begin{align}
\widehat{H_{D}}\left(W_r\right)
&= H_{D}\left(W_r\right) + A_{D}\left(W_r\right)
\end{align}
and
\begin{align}
 H_{D}\left(W_r\right) &= \begin{cases}
 \begin{aligned}
\sum_{\beta \in D}\left(\frac{\zeta^{'}}{\zeta}\left(1+ \gamma - \beta\right)- \frac{\zeta^{'}}{\zeta}\left(1+ \gamma + \beta\right)\right)
+ \frac{\zeta^{'}}{\zeta}\left(1+ 2 \gamma\right)
\end{aligned} & W_r =\{\gamma\}\\
\left(\frac{\zeta^{'}}{\zeta}\right)^{'}\left(1+ \gamma_1 + \gamma_2\right)& W_r = \{\gamma_1, \gamma_2\}\\
1&W_r=\emptyset
\end{cases}\\
A_{D}\left(W_r\right)& = \left.\prod_{\alpha \in W_r}\frac{\partial}{\partial \alpha}A_{DL}\left(A,B,D\right)\right|_{B=A}. \label{eq:DLAD}
\end{align}
$Z_{\zeta}, Y_{\zeta}, A_{DL}(A,B,D)$ are defined in Conjecture~\ref{conj:DLratios} and the sum over $0<d\leq X$ includes only fundamental discriminants.
\end{theorem}
Note that the conditions on $A$ come from the conditions on Conjecture~\ref{conj:DLratios}.

This calculation follows exactly the same lines as similar calculations and arguments used in the proof of Theorem~\ref{LFJconj}. For full details, see the thesis of Amy Mason \cite{Mason2013}.

\subsection{Residue Theorem for Quadratic Dirichlet $L$-functions}
In this section, we will locate the poles of $J_{DL}^{*}(A)$ and calculate the residue
of these poles.

\begin{theorem} \label{thm:DLres}
Let A be a finite set of complex numbers, let $\alpha^{*},\beta^{*} \in A$ and $A^{'}=A \backslash\{a,b\}$ and let $J_{DL}^{*}(A)$ be as defined in Conjecture~\ref{thm:DL2}. Then
\begin{align}
\begin{split}
\Res_{\alpha^{*} =-\beta^{*}} \left(J_{DL}^{*}(A)\right) &= J_{DL}^{*}(A^{'} \cup \{\beta^{*}\}) + J_{DL}^{*}(A^{'} \cup \{-\beta^{*}\})\\
& \qquad -\frac{\psi_{DL}^{'}}{\psi_{DL}}\left(\frac{1}{2} +\beta^{*},\chi_d\right) J_{DL}^{*}(A^{'})
\end{split}
\end{align}
where $\psi_{DL}\left(s,\chi_d\right)= \left(\frac{\pi}{d}\right)^{s-\frac{1}{2}} \frac{\Gamma(\frac{1-s}{2})}{\Gamma(\frac{s}{2})}$ comes from the functional equation.
\begin{align}
\Res_{\alpha^{*} =0} \left(J_{DL}^{*}(A)\right) &= 0.
\end{align}
\end{theorem}
\begin{proof}
Let $D\subseteq A$,  define $P_D$ and $Q$ as
\begin{align}
J_{DL}^{*}(A) &= \sum_{D \subseteq A}{P_D(A \backslash D)}\\
    &=\sum_{D \subseteq A}{\widehat{Q(D)} \sum_{\substack{A \backslash D= \bigcup_r{W_r}\\|W_r|\leq 2 }}}{\prod_{r}{\left(\widehat{H_D}(W_r)\right)}}
\end{align}
where
\begin{align}
 \begin{split}
\widehat{Q(D)}&=  \sum_{0 < d \leq X}{\sum_{D \subset A}{\left(\frac{d}{\pi}\right)^{- \sum_{\delta \in D}{\delta}}}}\prod_{\delta \in D}{\frac{\Gamma\left(\frac{1}{4} - \frac{\delta}{2}\right)}{\Gamma\left(\frac{1}{4}+\frac{\delta}{2}\right)}}\\
    &\qquad \times \sqrt{\frac{Z_{\zeta}\left(D^{-}, D^{-}\right)Z_{\zeta}\left(D,D\right)Y_{\zeta}\left(D^{-}\right)}{Z_{\zeta}^{\dagger}\left(D^{-},D\right)^2Y_{\zeta}\left(D\right)}}\\
    & \qquad \times \left(-1\right)^{|D|}A_{DL}\left(D,D,D\right)
\end{split}
\end{align}
and
$\widehat{H_D}$ is defined in Theorem~\ref{thm:DL2} and $A_E, Z_\zeta, Y_\zeta$, etc, are defined as in Conjecture \ref{conj:DLratios}.

Replacing $Q(D)$ with $\widehat{Q(D)}$ and $H(D,W)$ with $\widehat{H_D}(W)$ in Section \ref{sub:residue}, conditions {\bf P1} to {\bf P4} and {\bf R1}  and {\bf R2} hold with
\begin{equation}
f(\beta^*)=-\frac{\psi_{DL}^{'}}{\psi_{DL}}\left(\frac{1}{2}+\beta^{*},\chi_d\right).
\end{equation}  This proof is similar to that in the elliptic curve $L$-function case and we will not reproduce them here. Full details can be found in Amy Mason's thesis\cite{Mason2013}.
\end{proof}

\subsection{$n$-level Density of Zeros of Families of Quadratic Dirichlet $L$-functions}\label{sub:dirncorr}

 We will define the $n$-level density function, $S^{DL}_n(f)$ for the zeros for the family of $L$-functions $D^{+} = \left\{L(s, \chi_d) : d>0; d \;{\rm a\;fundamental\;discriminant}\right\}$.
\begin{align}
S_n^{DL}(f) = \sum_{\substack{0 < d \leq X }}{\sum_{t_1, \cdots, t_n >0}{f(\gamma_{t_1,d}, \cdots, \gamma_{t_n,d})}}
\end{align}
where $\gamma_{i,d}$ is the $i^{th}$ zero of $L\left(s,\chi_d\right)$ on the critical line, above the real axis. Note that the sum over the zeros in $S_n^{DL}(f)$ is not restricted to a sum over distinct indices at this stage.

\begin{theorem}\label{thm:DLF1}
Let $C_{-}$ denote the path from $ - \delta - \infty i$ up to $- \delta + \infty i$ and let $C_{+}$
denote the path from $  \delta - \infty i$ up to $ \delta + \infty i$, where $\delta$ is a small positive number. Let $f$ be a
holomorphic function of $n$ variables such that
\begin{align}
f\left(\theta_{j_1}, \cdots, \theta_{j_n}\right) = f\left(\pm \theta_{j_1}, \cdots, \pm
\theta_{j_n}\right).
\end{align}

Then
\begin{align}\label{DLFnostar}
\begin{split}
2^n \left(2 \pi i \right)^n& S_n^{DL}(f)= \\
&\sum_{\substack{K \cup L \cup M =\\ \left\{1, \cdots,
n\right\}}}  \int_{C_+^K} {\int_{C_{-}^{L \cup M}}{\left(-1\right)^{|M|}J^{DL}\left(z_K \cup - z_L, z_M\right)}}\\
&\qquad \times f\left(iz_1, \cdots,
iz_n\right)dz_1 \cdots d z_n
\end{split}
\end{align}
where $z_K = \left\{z_k: k \in K\right\}$, $-z_L = \left\{-z_l: l \in L\right\}$ and
$\int_{C_+^K} {\int_{C_-^{L \cup M}}}$ means we are integrating all the variables in $z_K$ along the
$C_+$ path and all others along the $C_-$ path and define
\begin{align} \label{def:DLJ}
J_{DL}\left(A,B\right) &= \sum_{0 < d \leq X}{\prod_{\alpha \in A}{\frac{L^{'}}{L}\left(\frac{1}{2}+ \alpha, \chi_d\right)}}\prod_{\beta \in B}{\frac{\psi_{DL}^{'}}{\psi_{DL}}(\frac{1}{2}+\beta, \chi_d)}
\end{align}
where $\psi_{DL}\left(s,\chi_d\right)= \left(\frac{\pi}{d}\right)^{s-\frac{1}{2}} \frac{\Gamma(\frac{1-s}{2})}{\Gamma(\frac{s}{2})}$, with this definition being the obvious extension of $J_{DL}(A)$ defined in Equation~\eqref{def:JDL}. Here $K,L,M$ are finite, disjoint sets of integers and $A,B$ are finite sets of complex numbers. The sum over $d$ is over fundamental discriminants.
\end{theorem}
The proof of this trivially follows the same arguments as the proof of Theorem~\ref{thm:LF1}. Note that again we have an identity giving $J_{DL}(A,B) = J_{DL}^{*}(A,B)+ O\left(X^{\frac{1}{2}+ \epsilon}\right)$ where $A$ is a finite set of complex numbers where $ \frac{1}{log X} \ll \Re(\alpha) < \frac{1}{4} $ and $Im(\alpha) \ll X^{1- \epsilon}$ for $\alpha \in A$ via Theorem~\ref{thm:DL2}. Again, all preceding results translate immediately to $J_{DL}(A,B)$ from the original $J_{DL}(A)$.  Here
\begin{align}
\begin{split}\label{eq:JDLstar2}
J_{DL}^{*}\left(A,B\right) &= \sum_{0 < d \leq X}{\prod_{\beta \in B}{\frac{\psi_{DL}^{'}}{\psi_{DL}}\left(\frac{1}{2}+\beta, \chi_d\right)}\sum_{D \subset A}{\left(\frac{d}{\pi}\right)^{- \sum_{\delta \in D}{\delta}}}}\\
&\qquad \times \prod_{\delta \in D}{\frac{\Gamma\left(\frac{1}{4} - \frac{\delta}{2}\right)}{\Gamma\left(\frac{1}{4}+\frac{\delta}{2}\right)}} \sqrt{\frac{Z_{\zeta}\left(D^{-}, D^{-}\right)Z_{\zeta}\left(D,D\right)Y_{\zeta}\left(D^{-}\right)}{Z_{\zeta}^{\dagger}\left(D^{-},D\right)^2Y_{\zeta}\left(D\right)}}\\
& \qquad \times \left(-1\right)^{|D|}A_{DL}\left(D,D,D\right)\\
& \qquad \times \sum_{\substack{A/D =\\ W_1 \cup \cdots  \cup W_R}}\prod_{r=1}^{R}{\widehat{H_{D}}\left(W_r\right)}.
\end{split}
\end{align}
with all terms as defined in Theorem~\ref{thm:DL2}.

The only possible poles are, as in previous cases, when $\alpha, \beta \in A$ and $\alpha = 0$ or $\alpha = - \beta$ with residues
\begin{align}
\Res_{\alpha^{*} =0} \left(J_{DL}^{*}(A,B)\right) &= 0,\\
\begin{split}
\Res_{\alpha^{*} =-\beta^{*}} \left(J_{DL}^{*}(A,B)\right) &= J_{DL}^{*}\left(A^{'} \cup \{\beta^{*}\},B\right) \\
& \qquad + J_{DL}^{*}\left(A^{'} \cup \{-\beta^{*}\},B\right)\\
& \qquad - J_{DL}^{*}\left(A^{'}, B \cup  \{\beta^{*}\}\right)
\end{split}
\end{align}
where $\psi_{DL}\left(s,\chi_d\right)= \left(\frac{\pi}{d}\right)^{s-\frac{1}{2}} \frac{\Gamma(\frac{s}{2})}{\Gamma(\frac{1-s}{2})}$ comes from the functional equation. Reusing our earlier notation defined in Equation~\eqref{def:LFsum notation}
\begin{align}\label{eq:sumnotation2}
\sideset{}{^{n,R}}\sum = \sum^{\infty}_{\substack{j_1, \cdots, j_n =1 \\ j_i \neq j_k \forall i,k>R}}.
\end{align}

For fixed sets $K,L,M$ such that $K \cup L \cup M = \{1, \cdots, n\}$ let $I^{n,R}_{f,K,L,M}$ be the integral in Theorem~\ref{thm:DLF1} with $N-R$ of the integrals shifted onto the imaginary axis. All the integrals on the imaginary axis are principal value integrals.
\begin{align}
\begin{split}
I^{n,R}_{f,K,L,M}=& \int^{i\infty}_{-i\infty}{ \cdots \int^{i\infty}_{-i\infty}{ \int_{C^{K \cap \{1, \cdots, R\}}_+}{\int_{C^{(L \cup M) \cap \{1, \cdots, R\}}_-}{J_{DL}^{*}(z_K \cup -z_L, z_M)}}}}\\
& \qquad \qquad \times f(iz_1, \cdots, iz_n) dz_1 \cdots dz_R d z_{R+1} \cdots dz_n.
\end{split}
\end{align}
Then it can be proved, using the same ideas as the proof of Theorem~\ref{thm:LFmain} that
\begin{theorem} \label{thm:DLFmain}
Assume Conjecture~\ref{conj:DLratios}. With the notation defined above, $0\leq R \leq n$, and $d$ summed over fundamental discriminants,
\begin{align}
\begin{split}
(2 \pi i)^n 2^n &\sum_{\substack{0 < d\leq X}}{\sideset{}{^{n,R}}\sum{f(\gamma_{j_1,d}, \cdots, \gamma_{j_n,d}) }} \\
&= \sum_{K \cup L \cup M = \{1, \cdots, n\}}{(-1)^{|M|}I^{n,R}_{f,K,L,M}} + O\left(X^{1/2+\epsilon}\right). \label{DLFstatement}
\end{split}
\end{align}
\end{theorem}

This immediately gives us the result for the $n$-level density of the zeros of $D^{+} = \left\{L(s, \chi_d) : d>0\right\}$ averaged over the family.
We can state our final result for the $n$-level density of zeros of families Dirichlet $L$-functions with real quadratic characters.
\begin{theorem} ($n$-level Density of Zeros of Families of Quadratic Dirichlet $L$-functions)  \label{thm:DLNcorr}
Assume Conjecture~\ref{conj:DLratios}. Then
\begin{align}
\begin{split}
&\sum_{\substack{0 < d\leq X}}{\sideset{}{^{n,0}}\sum{f(\gamma_{j_1,d}, \cdots, \gamma_{j_n,d}) }}\\
&= \frac{1}{(2 \pi )^n}\sum_{K \cup L \cup M= \{1 \cdots, n\}}{(-1)^{|M|}\left(\int_{0}^{\infty}\right)^n J_{DL}^{*}(-iz_K \cup i z_L, -iz_M)}\\
& \qquad \qquad \times f(z_1, \cdots, z_n)dz_1 \cdots dz_n + O\left(X^{\frac{1}{2}+\epsilon}\right)
\end{split}
\end{align}
with $J_{DL}^{*}(A,B)$ as defined in (\ref{eq:JDLstar2}) and the sum notation defined at (\ref{eq:sumnotation2}) indicating a sum over zeros with distinct indices.  The sum in $d$ is over fundamental discriminants.
\end{theorem}
Using the same argument as in Section~\ref{sub:LFncorr}, it is clear there are no poles on the paths on integration.

\section{$n$-level Densities with Restricted Support}\label{sec:restricted}

Currently it is extremely difficult to verify that number theoretic results on $n$-level densities agree with theorems from random matrix theory. These calculations always involve a test function (for example the function $f$ in (\ref{eq:generaln-level})) which is sampled at positions of the zeros, and theorems can only be proven if the Fourier transform of this test function has its support restricted to a fixed interval around the origin.  Rudnick and Sarnak~\cite{rudnick1996zp} and Rubinstein~\cite{rubinstein2001low} both use ad hoc methods to compare their calculations with support restricted to $(-1,1)$, for $L$-functions with unitary and symplectic symmetry respectively. Gao~\cite{gao2008n} extended Rubinstein's calculation to $(-2,2)$, but was only able to verify his result for $n \leq 3$. Levinson and Miller~\cite{levinson2012n} extended this to $n \leq 7$ and Entin et al.\cite{kn:entrodrud} verified it for all $n$, by-passing the complicated combinatorics that made previous calculations so difficult and using a function field analogue.  Still it remains clear that there is no consistent method of verifying these connections and that in every case the task is non-trivial.

Conrey and Snaith \cite{conrey2012restricted} used their random matrix calculations from the earlier paper  \cite{conrey2008ce} to allow a more accessible $n$-level density comparison between  random matrix theory and the Riemann zeros. They demonstrate that the way they write the random matrix theory $n$-point correlation formulae in \cite{conrey2008ce} allows immediate identification of
which terms remain when restrictions are placed on the support of the Fourier transform of the test function and so the simplified expression can be used to compare with a number theory calculation where the support is restricted.

We give a sketch of the way in which the same ideas can be applied to the $n$-level density of eigenangles of orthogonal (see Section \ref{sub:restrictedortho}) or symplectic matrices (in Section \ref{sub:restrictedsymp}) when similar restrictions are placed on the support of the test function. These sections give expressions for $n$-level densities of eigenvalues of orthogonal and symplectic matrices when the Fourier transform of the test function has any range of support, reducing the number of terms under consideration when matching with number theoretic results. These results  should allow for easier and more straightforward verification of whether number theoretic results limit to random matrix expressions in the orthogonal and symplectic cases.  Sections \ref{sub:restrictedortho} and \ref{sub:restrictedsymp} describe the principle behind the calculations, while the details would depend on the particular family of $L$-functions and the set-up of the Fourier transform and any smoothing or weighting functions used;  this is the subject of future work.

\subsection{Orthogonal} \label{sub:restrictedortho}

Recall from Section~\ref{sub:ncorr}, that Lemma~\ref{thm:RMorthoJstar} tells us that for $n \leq N$ and $f$ a $2\pi$-periodic,
holomorphic function of $n$ variables such that
\begin{align}
f\left(x_1, \cdots, x_n\right) = f\left(\pm x_1, \cdots, \pm
x_n\right)
\end{align}
then
\begin{align}
\begin{split}
2^n &\int_{SO\left(2N\right)} {\sum_{j_1, \cdots, j_n =1}^N {f\left(\theta_{j_1}, \cdots,
\theta_{j_n}\right) dX} } \\
&= \frac{1}{\left(2 \pi i \right)^n} \sum_{K \cup L \cup M = \left\{1, \cdots,
n\right\}}{ (2N) ^{\left|M\right|}} \\
& \qquad \times \int_{C_+^K} {\int_{C_-^{L+M}}{J^{*}\left(z_K \cup - z_L\right)f\left(iz_1, \cdots,
iz_n\right)dz_1 \cdots d z_n}}.
\end{split}
\end{align}
with $J^{*}(A)$ etc as described in Section~\ref{sub:ncorr}.

For consistency we rewrite this in terms of the notation used by Conrey and Snaith~\cite{conrey2012restricted}. Firstly we redefine our idea of eigenangles so that we have an infinite set of points. For $X$ a $2N \times 2N$ orthogonal matrix with eigenvalues $e^{\pm i\theta_n}$, $n=1,\ldots, N$, we will consider the statistics of the set of points
\begin{align}
\cdots \theta_{-R} \leq \theta_{_R+1} \leq \cdots \leq \theta_0 \leq \theta_1 \leq \cdots \leq \theta_R \leq \theta_{R+1} \leq \cdots
\end{align}
where
\begin{align}
\theta_{r+2kN}= \theta_{r}+ 2\pi k.
\end{align}
Take $F(x_1, \cdots, x_n)$ a holomorphic function of $n$ variables such that
\begin{align}
F\left(x_1, \cdots, x_n\right) = F\left(\pm x_1, \cdots, \pm
x_n\right).
\end{align}Let $\delta >0$ and $\int_{(\delta)}$ indicate integration from $\delta -i \infty$ to $\delta + i \infty$ and assume $F(x_1, \cdots, x_n)$ decays rapidly in each variable in horizontal strips. Then a simple extension of Lemma~\ref{thm:RMorthoJstar} gives
\begin{align}
\begin{split}
 &\int_{SO\left(2N\right)} {\sum_{j_1, \cdots, j_n =-\infty}^\infty {F\left(\theta_{j_1}, \cdots,
\theta_{j_n}\right) dX} } \\
&= \frac{1}{\left(2 \pi i \right)^n} \sum_{K \cup L \cup M = \left\{1, \cdots,
n\right\}}{ (2N) ^{\left|M\right|}} \\
& \qquad \times \int_{\left(\delta\right)^{|K|}} {\int_{(-\delta)^{|L|}}{\int_{\left(0\right)^{|M|}}}{J^{*}\left(z_K \cup - z_L\right)F\left(iz_1, \cdots,
iz_n\right)dz_1 \cdots d z_n}}. \label{eq:Oworking}
\end{split}
\end{align}
where
\begin{align}
\begin{split}\label{def:OJstar}
J^*(A) &= \sum_{D\subseteq A}{e^{-2N \sum\limits_{d \in D}{d}}(-1)^{|D|} \sqrt{\frac{Z(D,D)Z(D^-,D^-)Y(D)}{Y(D^-)Z^{\dag}(D^{-},D)^2}}}\\
& \qquad \times \sum_{\substack{A/D = W_1 \cup \cdots \cup W_R \\ |W_r| \leq 2}}{\prod_{r=1}^{R}{H_D(W_r)}}
\end{split}
\end{align}
where $H_D(W_r), Y, Z$ are defined in Equations~\ref{def:HD}-~\ref{def:Z2}.

We apply this with $F(x_1, \cdots,x_n) = f\left(\frac{Nx_1}{ 2\pi},\cdots,\frac{Nx_n}{2\pi}\right)$
and rewrite $f$ in terms of its Fourier transform, $\Phi_f$, so that we can investigate the effect of limiting the support of $\Phi_f$.
\begin{align}
\begin{split}
&f\left(\frac{iNz_1}{ 2\pi},\cdots,\frac{iNz_n}{2\pi}\right) \\
& \qquad =\int_{\mathbb{R}^n}{\Phi_f\left(\xi_1, \cdots, \xi_n\right)} e\left(-\frac{iNz_1\xi_1}{ 2\pi}- \cdots - \frac{iNz_n\xi_n}{ 2\pi}\right) d \xi_1 \cdots d \xi_n \label{eq:fourier}
\end{split}
\end{align}
where
\begin{align}
e\left(-\frac{iNz_1\xi_1}{ 2\pi}- \cdots - \frac{iNz_n\xi_n}{ 2\pi}\right) = e^{Nz_1\xi_1+ \cdots + Nz_n\xi_n}.
\end{align}
Note that $|\Re(z_i)|\leq \delta$ for all $i$ and suppose that $\Phi_f\left(\xi_1, \cdots, \xi_n\right) = 0$ if $|\xi_1|+ \cdots +|\xi_n| > 2q-\epsilon$ for  some $\epsilon>0$.  That is,  $f$ has restricted support in the sense that arises in $n$-level density calculations in a number theoretical context. Therefore
\begin{align}
\left| e^{Nz_1\xi_1+ \cdots + Nz_n\xi_n}\right| \leq e^{N \delta \left(2 q -\epsilon\right)}. \label{Ofactor1}
\end{align}

Now consider sets $D$ where $|D|\geq q$. Then, as $|\Re(d)|= \delta, \;\forall d \in D$, the exponential term in $J^{*}(A)$ above is bounded by
\begin{align}
\left|e^{-2N \sum_{d \in D}{d}}\right| \leq e^{-2N\delta q} \label{Ofactor2}.
\end{align}
Clearly then the product of these two factors in Equations~\ref{Ofactor1} and \ref{Ofactor2} is also bounded
\begin{align}\label{eq:deltatoinfinity}
\left|e^{Nz_1\xi_1+ \cdots + Nz_n\xi_n}e^{-2N \sum_{d \in D}{d}}\right| &\leq e^{-N \delta \epsilon}\\
& \rightarrow 0 \text{ as } \delta \rightarrow \infty.\nonumber
\end{align}
It is significant that (\ref{eq:deltatoinfinity}) tends to zero as $\delta$ goes to infinity because we know from Section \ref{sub:residue} that we can shift the $(\delta)$ contours far to the right and the $(-\delta)$ contours far to the left without encountering any poles of the integrand $J^*(z_K\cup -z_L)\;F(iz_1, \ldots, iz_n)$ of (\ref{eq:Oworking}).  Thus if $F$ decays fast enough as any $z$ variable strays far from the real axis then (\ref{eq:Oworking}) doesn't change as the contours are shifted to infinity. It can be shown as in \cite{conrey2012restricted} that all factors in the integrand other than those considered in  (\ref{eq:deltatoinfinity}) can be bounded and in this way we see that any term with $|D|\geq q$ in (\ref{def:OJstar}) contributes nothing to the integral.
With this in mind we define
\begin{align}
\begin{split}\label{def:OJq}
J_q^*(A) &= \sum_{\substack{D\subseteq A\\ |D| < q}}{e^{-2N \sum\limits_{\delta \in D}{\delta}}(-1)^{|D|} \sqrt{\frac{Z(D,D)Z(D^-,D^-)Y(D)}{Y(D^-)Z^{\dag}(D^{-},D)^2}}}\\
& \qquad \times \sum_{\substack{A/D = W_1 \cup \cdots \cup W_R \\ |W_r| \leq 2}}{\prod_{r=1}^{R}{H_D(W_r)}}.
\end{split}
\end{align}

When the support of $f$ is limited to $2q$, we can rewrite the result in Lemma~\ref{thm:RMorthoJstar} to give
\begin{theorem} \label{thm:Jq}
Let $\delta >0$ and $\int_{(\delta)}$ indicate integration from $\delta -i \infty$ to $\delta + i \infty$. Take $F(x_1, \cdots, x_n)$ a holomorphic function of $n$ variables such that
\begin{align}
F\left(x_1, \cdots, x_n\right) = F\left(\pm x_1, \cdots, \pm
x_n\right).
 \end{align}
 and assume $F(x_1, \cdots, x_n)= f\left(\frac{Nx_1}{ 2\pi},\cdots,\frac{Nx_n}{2\pi}\right)$ decays rapidly in each variable in horizontal strips. With $\Phi_f$ defined as in Equation~\ref{eq:fourier}, let $\Phi_f\left(\xi_1, \cdots, \xi_n\right) = 0$ if $|\xi_1|+ \cdots +|\xi_n| > 2q-\epsilon$ for some $\epsilon>0$. Then
\begin{align}
\begin{split}
 &\int_{SO\left(2N\right)} {\sum_{j_1, \cdots, j_n =-\infty}^\infty {F\left(\theta_{j_1}, \cdots,
\theta_{j_n}\right) dX} } \\
&= \frac{1}{\left(2 \pi i \right)^n} \sum_{K \cup L \cup M = \left\{1, \cdots,
n\right\}}{ (2N) ^{\left|M\right|}} \\
& \qquad \times \int_{\left(\delta\right)^{K}} {\int_{(-\delta)^{L}}{\int_{\left(0\right)^{M}}}{J_q^{*}\left(z_K \cup - z_L\right)F\left(iz_1, \cdots,
iz_n\right)dz_1 \cdots d z_n}}.
\end{split}
\end{align}
where $H_D(W_r), Y, Z$ are defined in Equations~\ref{def:HD}-~\ref{def:Z2} and $J^{*}_q(A)$ defined in Equation~\ref{def:OJq} above.
\end{theorem}
\subsubsection{Special Case $q=1$}

We particularly want to consider the case $q=1$ as the terms simplify dramatically, and should allow for for immediate identification with the terms survive the restrictions in the number theory case when the support is in $(-2,2)$.

Consider Theorem~\ref{thm:Jq} in the case that $q=1$. Then in the sum within $J^{*}_q(A)$ the only valid term is that which arises from $D= \emptyset$.
\begin{align}
\begin{split}
J_1^*(A) &= \sum_{\substack{D\subseteq A\\ |D| < 1}}{e^{-2N \sum\limits_{\delta \in D}{\delta}}(-1)^{|D|} \sqrt{\frac{Z(D,D)Z(D^-,D^-)Y(D)}{Y(D^-)Z^{\dag}(D^{-},D)^2}}}\\
& \qquad \times \sum_{\substack{A/D = W_1\cup \cdots \cup W_R \\ |W_r| \leq 2}}{\prod_{r=1}^{R}{H_D(W_r)}}
\end{split}\\
&= \sum_{\substack{A = W_1\cup \cdots \cup W_R \\ |W_r| \leq 2}}{\prod_{r=1}^{R}{H_{\emptyset}(W_r)}}
\end{align}
where
\begin{align}
H_{\emptyset}\left(W_r\right) =
\begin{cases}
-\frac{z^{'}}{z}\left(2 \alpha\right) & W_r = \{\alpha\}\\
\left(\frac{z^{'}}{z}\right)^{'}\left(\alpha + \beta\right) & W_r=\{\alpha, \beta\}.
\end{cases}
\end{align}

\subsection{Symplectic} \label{sub:restrictedsymp}

We can calculate a similar result for the eigenangles of symplectic matrices.  For $X$ an $2N \times 2N$ symplectic matrix, we extend its eigenvalues as before
\begin{align}
\cdots \theta_{-R} \leq \theta_{_R+1} \leq \cdots \leq \theta_0 \leq \theta_1 \leq \cdots \leq \theta_R \leq \theta_{R+1} \leq \cdots
\end{align}
where
\begin{align}
\theta_{r+2kN}= \theta_{r}+ 2\pi k.
\end{align}
Take $F(x_1, \cdots, x_n)$ a holomorphic function of $n$ variables such that
\begin{align}
F\left(x_1, \cdots, x_n\right) = F\left(\pm x_1, \cdots, \pm
x_n\right).
\end{align}
Let $\delta >0$ and $\int_{(\delta)}$ indicate integration from $\delta -i \infty$ to $\delta + i \infty$ and assume $F(x_1, \cdots, x_n)$ decays rapidly in each variable in horizontal strips. Using Lemma~\ref{thm:RMsympJstar} and the notation of Conrey and Snaith \cite{conrey2012restricted} we can see that
\begin{align}
\begin{split}
 &\int_{USp\left(2N\right)} {\sum_{j_1, \cdots, j_n =-\infty}^\infty {F\left(\theta_{j_1}, \cdots,
\theta_{j_n}\right) dX} } \\
&= \frac{1}{\left(2 \pi i \right)^n} \sum_{K \cup L \cup M = \left\{1, \cdots,
n\right\}}{ (2N) ^{\left|M\right|}} \\
& \qquad \times \int_{\left(\delta\right)^{|K|}} {\int_{(-\delta)^{|L|}}{\int_{\left(0\right)^{|M|}}}{J_{USp(2N)}^{*}\left(z_K \cup - z_L\right)F\left(iz_1, \cdots,
iz_n\right)dz_1 \cdots d z_n}}. \label{eq:Uworking}
\end{split}
\end{align}
where
\begin{align}
\begin{split}\label{def:UJstar}
J_{USp(2N)}^*(A) &= \sum_{D\subseteq A}{e^{-2N \sum\limits_{\delta \in D}{\delta}}(-1)^{|D|} \sqrt{\frac{Z(D,D)Z(D^-,D^-)Y(D^-)}{Y(D)Z^{\dag}(D^{-},D)^2}}}\\
& \qquad \times \sum_{\substack{A/D = W_1 \cup \cdots \cup W_R \\ |W_r| \leq 2}}{\prod_{r=1}^{R}{H_D(W_r)}}
\end{split}
\end{align}
and
\begin{align}
H_D(W_r)=
\begin{cases}
\begin{aligned}
&\sum_{d \in D}{\left(\frac{z^{'}}{z}\left( \alpha-d\right)-\frac{z^{'}}{z}\left(\alpha+d\right)\right)}\\
& \qquad + \frac{z^{'}}{z}\left(2 \alpha\right)
\end{aligned} & W_r = \{\alpha\}\\
\begin{aligned}
\left(\frac{z^{'}}{z}\right)^{'}\left(\alpha + \beta\right)\end{aligned} & W_r = \{\alpha, \beta\}\\
1& W_r=\emptyset
\end{cases}
\end{align}

We apply this with $F(x_1, \cdots,x_n) = f\left(\frac{Nx_1}{ 2\pi},\cdots,\frac{Nx_n}{2\pi}\right) $
and as in the orthogonal case we write $f$ in terms of its   Fourier transform.

Suppose that $\Phi_f\left(\xi_1, \cdots, \xi_n\right) = 0$ if $|\xi_1|+ \cdots |\xi_n| > 2q-\epsilon$ for some $\epsilon>0$. Hence
\begin{align}
e\left(-\frac{iNz_1\xi_1}{ 2\pi}- \cdots - \frac{iNz_n\xi_n}{ 2\pi}\right) \leq e^{N \delta \left(2 q -\epsilon\right)}. \label{USpfactor1}
\end{align}
and considering sets $D$ where $|D|\geq q$ in a similar manner to the orthogonal case above,  we can see the exponential term in $J_{USp(2N)}^{*}(A)$ is bounded by
\begin{align}
\left|e^{-2N \sum_{d \in D}{d}}\right| \leq e^{-2N\delta q} \label{USpfactor2}.
\end{align}
Hence the product of these two factors in Equations~\ref{USpfactor1} and \ref{USpfactor2} is bounded by
\begin{align}
\left|e\left(-\frac{iNz_1\xi_1}{ 2\pi}- \cdots - \frac{iNz_n\xi_n}{ 2\pi}\right)e^{-2N \sum_{d \in D}{d}}\right| &\leq e^{-N \delta \epsilon}\\
& \rightarrow 0 \text{ as } \delta \rightarrow \infty.
\end{align}
Define
\begin{align}
\begin{split}\label{def:UJq}
J_{USp(2N),q}^*(A) &= \sum_{\substack{D\subseteq A\\ |D| < q}}{e^{-2N \sum\limits_{\delta \in D}{\delta}}(-1)^{|D|} \sqrt{\frac{Z(D,D)Z(D^-,D^-)Y(D^-)}{Y(D)Z^{\dag}(D^{-},D)^2}}}\\
& \qquad \times \sum_{\substack{A/D = W_1 \cup \cdots \cup W_R \\ |W_r| \leq 2}}{\prod_{r=1}^{R}{H_D(W_r)}}
\end{split}
\end{align}

Thus we can deduce
\begin{theorem} \label{thm:USP1}
Let $\delta >0$ and $\int_{(\delta)}$ indicate integration from $\delta -i \infty$ to $\delta + i \infty$ Take $F(x_1, \cdots, x_n)$ a holomorphic function of $n$ variables such that
\begin{align}
F\left(x_1, \cdots, x_n\right) = F\left(\pm x_1, \cdots, \pm
x_n\right).
 \end{align}
 and assume $F(x_1, \cdots, x_n)= f\left(\frac{Nx_1}{ 2\pi},\cdots,\frac{Nx_n}{2\pi}\right)$ decays rapidly in each variable in horizontal strips. With $\Phi_f$ defined as in Equation~\ref{eq:fourier}, let $\Phi_f\left(\xi_1, \cdots, \xi_n\right) = 0$ if $|\xi_1|+ \cdots +|\xi_n| > 2q-\epsilon$ for some $\epsilon>0$. Then
\begin{align}
\begin{split}
&\int_{USp\left(2N\right)} {\sum_{j_1, \cdots, j_n =-\infty}^\infty {F\left(\theta_{j_1}, \cdots,
\theta_{j_n}\right) dX} } \\
&= \frac{1}{\left(2 \pi i \right)^n} \sum_{K \cup L \cup M = \left\{1, \cdots,
n\right\}}{ (2N) ^{\left|M\right|}} \\
& \qquad \times \int_{\left(\delta\right)^{K}} {\int_{(-\delta)^{L}}{\int_{\left(0\right)^{M}}}{J_{USp(2N),q}^{*}\left(z_K \cup - z_L\right)F\left(iz_1, \cdots,
iz_n\right)dz_1 \cdots d z_n}}. \label{eq:USpworking}
\end{split}
\end{align}
where $J^{*}_{USp(2N), q}(A)$,$H_D(W_r), Y, Z$ are defined as above.
\end{theorem}
\subsubsection{Special Case $q=1$}

Again, we particularly want to consider when $q=1$. Then in the sum within $J^{*}_q(A)$ in Theorem~\ref{thm:USP1} the only valid term is that which arises from $D= \emptyset$.
\begin{align}
\begin{split}
J_1^*(A) &= \sum_{\substack{D\subseteq A\\ |D| < 1}}{e^{-2N \sum\limits_{\delta \in D}{\delta}}(-1)^{|D|} \sqrt{\frac{Z(D,D)Z(D^-,D^-)Y(D^-)}{Y(D)Z^{\dag}(D^{-},D)^2}}}\\
& \qquad \times \sum_{\substack{A/D = W_1 \cup \cdots \cup W_R \\ |W_r| \leq 2}}{\prod_{r=1}^{R}{H_D(W_r)}}
\end{split}\\
&= \sum_{\substack{A = W_1\cup \cdots \cup W_R \\ |W_r| \leq 2}}{\prod_{r=1}^{R}{H_{\emptyset}(W_r)}}
\end{align}
where
\begin{align}
H_{\emptyset}\left(W_r\right) =
\begin{cases}
\frac{z^{'}}{z}\left(2 \alpha\right) & W_r = \{\alpha\}\\
\left(\frac{z^{'}}{z}\right)^{'}\left(\alpha + \beta\right) &W_r= \{\alpha, \beta\}.
\end{cases}
\end{align}

\section{Example Calculations}\label{app:12}

The purpose of this section is to explicitly calculate the $1$ and $2$-level density functions of zeroes of families of $L$-functions of elliptic curves to illuminate the $n$-level density calculations in Section~\ref{sub:LFncorr}. We will compare our solution with a previous result when $n=1$ \cite{huynh2009lowe} and to the limiting behaviour of random matrices when $n=2$\cite{conrey2005note} to demonstrate that our results match up with current knowledge. The sums of the form $0<d\leq X$ in this section are over fundamental discriminants.

\subsection{Calculation of $S^E_1$}

Let $C_{-}$ denote the path from $ - \delta - \infty$ up to $- \delta + \infty$ and let $C_{+}$
denote the path from $  \delta - \infty$ up to $ \delta + \infty$.  Let $f$ be a holomorphic function of one variable such that
\begin{align}
f\left(\theta\right) = f\left(- \theta\right). \label{testf1}
\end{align}
Then
\begin{align}
S^E_1(f) = \sum_{\substack{0 < d \leq X\\ \chi_{d}(-M)\omega_E = 1}}{ \sum_{j>0}f(\gamma_{j,d})}\label{def:s1}
\end{align}
where $\gamma_{j,d}$ is the height of the $j$th zero of the $L$-function $L_{E}(s, \chi_d)$. Then following the calculations in Sections~\ref{sub:LFncorr},
\begin{align}
2(2\pi i) S_1^E(f) &= \left(\int_{C_{+}}-\int_{C_{-}}\right)\sum_{\substack{0 < d \leq X\\ \chi_{d}(-M)\omega_E = +1}} \frac{L_E^{'}}{L_E}(\frac{1}{2} + \alpha, \chi_d)f(-i\alpha) d \alpha\\
&= \sum_{\substack{K \cup L \cup M\\ = \{1\}}}{\int_{C_{+}^{K}}{\int_{C_{-}^{L\cup M}}}{\left(-1\right)^{|M|}J_{E}(z_K \cup -z_L, z_M) f(iz_1)d z_1}}
\\
\begin{split}
&= \sum_{\substack{K \cup L \cup M\\ = \{1\}}}{\int_{C_{+}^{K}}{\int_{C_{-}^{L\cup M}}\left(-1\right)^{|M|}J^{*}_{E}(z_K \cup -z_L, z_M)  f(iz_1)d z_1}}\\
& \qquad  +O\left(X^{\frac{1}{2}+ \epsilon}\right)
\end{split}
\end{align} The final step follows from Theorem~\ref{LFJconj}. We can move the integration onto the imaginary axis as there are no poles when we only have a single variable. Hence
\begin{align}
2(2\pi) S_1^E(f) &= \int_{-\infty}^{\infty}{f(z)\left[\vphantom{\frac{1}{3}}J^{*}_E\left(iz, \emptyset\right)+J^{*}_E\left(- iz, \emptyset\right)\right.}\\
 & {\left.\qquad -J^{*}_E\left(\emptyset,iz\right)\vphantom{\frac{1}{3}}\right] dz}+O\left(X^{\frac{1}{2}+ \epsilon}\right),
\end{align}
recalling that the definition of $J^{*}_E\left(A,B\right)$ is
\begin{align}
\begin{split}
J^{*}_E\left(A,B\right) =& \sum_{\substack{0<d \leq X \\ \chi_d\left(-M\right)\omega_E = +1}}{\prod_{\beta \in B}{\frac{\psi^{'}}{\psi}(\frac{1}{2}+\beta, \chi_d)\sum_{D \subseteq A}{{\left(\frac{\sqrt{M}|d|}{2 \pi}\right)}^{- \sum\limits_{\delta \in D}{2\delta}}}}}\\
& \qquad \times \prod_{\delta \in D}{\frac{\Gamma\left(1- \delta\right)}{\Gamma\left(1+ \delta\right)}} \sqrt{\frac{Z_\zeta \left(D^{-},D^{-}\right)Z_\zeta \left(D,D\right)Y_\zeta \left(D\right)}{Z_\zeta ^{\dag}\left(D^{-},D\right)^{2}Y_\zeta \left(D^{-}\right)}}\\
& \qquad \times \left(-1\right)^{|D|}A_E\left(D,D,D\right)\sum_{\substack{A/D=\\W_1 \cup \cdots \cup W_R}}{\prod_{r=1}^{R}{\widetilde{H_D}\left(W_E\right)}}\\
\end{split}
\end{align}
where
\begin{align}
\widetilde{H_D}\left(W_r\right) &= H_D\left(W_r\right)+ A_{D,1}\left(W_r\right)+ A_{D,2}\left(W_r\right)+ A_{D,3}\left(W_r\right)\\
H_D\left(W_r\right) &= \begin{cases}
\sum_{\delta \in D}\left(\frac{\zeta^{'}}{\zeta}\left(1+ a- \delta \right)-\frac{\zeta^{'}}{\zeta}\left(1+ a +\delta \right)\right)- \frac{\zeta^{'}}{\zeta}\left(1+ 2a\right) & W_r = \{a\}\\
\frac{\zeta^{'}}{\zeta}\left(1+ a_1 + a_2\right) & W_r= \{a_1, a_2\}\\
1 & W_r=\emptyset\\
0 & |W_r|\geq 3
\end{cases}\\
A_{D,i}\left(W_r\right)&= \left.\prod_{w \in W_r}{\frac{\partial}{\partial w}}\log A_{E,i}\left(A,B,D\right)\right|_{B=A}
\end{align}
and $A_E, Z_\zeta, Y_\zeta$ etc defined as in Equations~\ref{eq:LF1} -~\ref{eq:LF2}.

\subsubsection{Checking $S^E_1$ against Known Results}

We want to compare this with the result from \cite{huynh2009lowe} where the authors use the ratios conjecture to calculate the one-level density for a family of elliptic curves.  They show agreement numerically between their formula and the empirical one-level density of the zeros over a long range, in regimes where random matrix behaviour dominates and as the arithmetic terms take over in importance.
\begin{theorem}[Huynh, Keating and Snaith~\cite{huynh2009lowe}] \label{thm:DKKS}
The $1$-level density family of even quadratic
twists of the $L$-function of an elliptic curve with
prime conductor M is
\begin{align}
\begin{split} \label{eq:DK1}
\widetilde{S_1(f)} &= \frac{1}{2\pi}\int_{- \infty}^{\infty}{f(t)\sum_{\substack{0 <d \leq X\\ \chi_d\left(-M\right) \omega_E=1}}\left(2\log \left(\frac{\sqrt{M}|d|}{2\pi}\right)+ \frac{\Gamma^{'}}{\Gamma}\left(1+ it\right)+ \frac{\Gamma^{'}}{\Gamma}\left(1- it\right)\right.}\\
& \qquad + 2\left[- \frac{\zeta^{'}}{\zeta}\left(1+2it\right) + \frac{L^{'}_E}{L_E}\left(sym^2,1+2it\right) + B^{1}_E\left(it,it\right)\right.\\
& \qquad \left.\left. - \left(\frac{\sqrt{M}|d|}{2\pi}\right)^{-2it}\frac{\Gamma(1-it)\zeta(1+2it)L_E(sym^2,1-2it)}{\Gamma(1+it)L_E(sym^2,1)}B_E(-it,it) \right]\right)\\
& \qquad + O\left(X^{\frac{1}{2}+\epsilon}\right)
\end{split}
\end{align}
where
\begin{align}
\widetilde{S_1(f)}= \sum_{\substack{0 < d \leq X\\ \chi_{d}(-M)\omega_E = +1}}\sum_{ j=-\infty}^\infty f(\gamma_{j,d})
\end{align}
 for $f$, an even test function as described in (\ref{testf1}). $\omega_E$ is the sign from the functional equation of $L_E(s)$ as described in Section~\ref{Ellipticderivation}, $L_E(sym^2, s)$ is the symmetric square $L$-function associated to $L_E(s)$ and $B_E$ and $B_E^{1}$ are arithmetic factors. The construction  of the last three functions are described in \cite{huynh2009lowe} (Note that $B_E$, and $B_E^1$ are called $A_E$ and $A_E^1$ in Huynh et al's paper, but we have relabelled them to avoid confusion with our analytical terms.)
\end{theorem}

We notice that $\widetilde{S_1(f)}$ differs from our definition of $S^E_1(f)$ in (\ref{def:s1}), as $\widetilde{S_1(f)}$ is a sum over all zeros of the $L$-functions while $S^E_1(f)$ is  only over the positive zeros. As we are restricting to even $L$-functions ( i.e $\chi_{d}(-M)\omega_E = +1$), we expect $\widetilde{S_1(f)}= 2S^E_1(f)$.

This result matches perfectly with $S^E_1(f)$. The term containing $2\log \left(\frac{\sqrt{M}|d|}{2\pi}\right)+ \frac{\Gamma^{'}}{\Gamma}\left(1+ it\right)+ \frac{\Gamma^{'}}{\Gamma}\left(1- it\right)$ corresponds exactly to $J^{*}_E\left(\emptyset,it\right)$. If we consider the parts of $S^E_1(f)$ corresponding to $J^{*}_E\left(it, \emptyset\right)+ J^{*}_E\left(-it, \emptyset\right)$ when $D=\emptyset$ we get
\begin{align}
&\int_{-\infty}^{\infty}f(z)\sum_{\substack{0<d \leq X \\ \chi_d(-M)\omega_E=1}}A_E\left(\emptyset, \emptyset, \emptyset\right)\left[\widetilde{H_{\emptyset}}(\{iz\}) + \widetilde{H_{\emptyset}}(\{-iz\})\right]dz\\
& \qquad = \int_{-\infty}^{\infty}f(z)\sum_{\substack{0<d \leq X \\ \chi_d(-M)\omega_E=1}} - \frac{\zeta^{'}}{\zeta}\left(1+2iz\right)- \frac{\zeta^{'}}{\zeta}\left(1-2iz\right) \\
& \qquad \qquad + \sum_{j=1}^3A_{D,j}(\{iz\})+A_{D,j}(\{-iz\})dz\\
& \qquad = \int_{-\infty}^{\infty}f(z)\sum_{\substack{0<d \leq X \\ \chi_d(-M)\omega_E=1}} 2\left[- \frac{\zeta^{'}}{\zeta}\left(1+2iz\right)+ \sum_{j=1}^3A_{D,j}(\{iz\})\right]dz.
\end{align}
This matches up with the term
\begin{align}
2\left[- \frac{\zeta^{'}}{\zeta}\left(1+2iz\right) + \frac{L_E^{'}}{L_E}\left(sym^2, 1+2iz\right) + B^{1}_E(iz,iz)\right]
\end{align}
in (\ref{eq:DK1}). In our definition of the Ratios Conjecture in Section~\ref{sec:ratios of twists}, we defined the analytical term $A_E$ to contain the $\frac{L_E\left(sym^2, 1+2\alpha\right)}{L_{E}\left(sym^21 + \alpha + \gamma\right)}$ term that Huynh et al. separated out from their analytic term $B_E$. Checking the constructions of the Ratios Conjectures in their paper against ours, it is clear that $\sum_{j}A_{D,j}(\{iz\}) = \frac{L_E^{'}}{L_E}\left(sym^2, 1+2iz\right) + B^{1}_E(iz,iz)$.

All that remains is to check that the remaining terms in (\ref{eq:DK1}) match up to the terms in $J^{*}_E\left(iz, \emptyset\right)+ J^{*}_E\left(-iz, \emptyset\right)$ where $D \neq\emptyset$. This is clear once we see from the construction of the arithmetic terms that $\frac{L_E(sym^2,1-2iz)}{L_E(sym^2,1)}B_E(-iz,iz) =A_E(\{iz\},\{iz\},\{iz\})$. So
\begin{align}
\begin{split}
& \int_{-\infty}^{\infty}f(z)\sum_{\substack{0<d \leq X \\ \chi_d(-M)\omega_E=1}}{\left[-\left(\frac{\sqrt{M}|d|}{2 \pi}\right)^{-2iz}\frac{\Gamma(1-iz)}{\Gamma(1+iz)}\zeta(1+2iz) A_E(\{iz\},\{iz\},\{iz\})\right.}\\
& \qquad \qquad \left. -\left(\frac{\sqrt{M}|d|}{2 \pi}\right)^{2iz}\frac{\Gamma(1+iz)}{\Gamma(1-iz)}\zeta(1-2iz)A_E(\{-iz\},\{-iz\},\{-iz\}) \right]dz
\end{split}\\
\begin{split}
& \qquad= \int_{-\infty}^{\infty}f(t)\sum_{\substack{0<d \leq X \\ \chi_d(-M)\omega_E=1}}{2\left[ -\left(\frac{\sqrt{M}|d|}{2\pi}\right)^{-2iz}\frac{\Gamma(1-iz)}{\Gamma(1+iz)}\zeta(1+2iz)\right.}\\
& \qquad \qquad \qquad \times \left. \frac{L_E(sym^2,1-2iz)}{L_E(sym^2,1)}B_E(-iz,iz) \vphantom{\left(\frac{\sqrt{M}|d|}{2\pi}\right)^{-2iz}}\right] dz.
\end{split}
\end{align}
Hence it is clear that our result in Theorem~\ref{thm:LFNcorr} for $n=1$  agrees with the previous result in Theorem~\ref{thm:DKKS}.

\subsection{Calculation of $S^E_2$} \label{sub:2corr}

As an further example, we demonstrate the calculation of the $2$-level density of zeroes of families of quadratic twist elliptic curve $L$-functions with even functional equation: a family with orthogonal symmetry.
Let $C_{-}$ denote the path from $ - \delta - \infty$ up to $- \delta + \infty$ and let $C_{+}$
denote the path from $  \delta - \infty$ up to $ \delta + \infty$. Let $f$ be a holomorphic function of $2$ variables, such that
\begin{align}
f\left(\theta_{1}, \theta_{2}\right) = f\left(\pm \theta_{1},\pm \theta_{2}\right).
\end{align}
Then
\begin{align}
S_2^E(f) = \sum_{\substack{0 < d \leq X\\ \chi_{d}(-M)\omega_E = +1}}{ \sum_{j_1,j_2 > 0}f(\gamma_{j_1,d},\gamma_{j_2,d})}
\end{align}
where $\gamma_{j,d}$ is the $j$th zero of the $L$-function $L_{E}(s,\chi_d)$. We know from Cauchy's Residue Theorem that
\begin{align}
2^2(2\pi i)^{2} S^E_2(f) = \left(\int_{C_{+}}-\int_{C_{-}}\right)^{2}\sum_{\substack{0 < d \leq X\\ \chi_{d}(-M)\omega_E = +1}} \frac{L_E^{'}}{L_E}(\frac{1}{2} + \alpha, \chi_d)\frac{L_E^{'}}{L_E}(\frac{1}{2} + \beta, \chi_d)f(-i\alpha, -i\beta) d \alpha d \beta.
\end{align}

We note that the functional equation of $L_{E}(s, \chi_d)$ can be arranged to give $\frac{L_E^{'}}{L_E}(1-s, \chi_d) = \frac{\psi^{'}}{\psi}(s, \chi_d)-\frac{L_E^{'}}{L_E}(s, \chi_d)$.
\begin{align}
\begin{split}
&(2\pi i)^{2} S^E_2(f) \\
    &\qquad = \left(\int_{C_+}\right)^{2}\sum_{\substack{0 < d \leq X\\ \chi_{d}(-M)\omega_E = +1}} \frac{L_E^{'}}{L_E}\left(\frac{1}{2} + \alpha, \chi_d\right)\frac{L_E^{'}}{L_E}\left(\frac{1}{2} + \beta, \chi_d\right)\\
            & \qquad \qquad \qquad \times f\left(-i\alpha, -i\beta\right) d \alpha d \beta\\
        & \qquad \qquad - \int_{C_+}\int_{C_{-}}\sum_{\substack{0 < d \leq X\\ \chi_{d}(-M)\omega_E = +1}} \left(\frac{\psi^{'}}{\psi}-\frac{L_E^{'}}{L_E}\right)\left(\frac{1}{2}-\alpha,\chi_d\right)\frac{L_E^{'}}{L_E}\left(\frac{1}{2} + \beta, \chi_d\right)\\
            & \qquad \qquad \qquad \times f\left(-i\alpha, -i\beta\right) d \alpha d \beta\\
        &\qquad \qquad - \int_{C_-}\int_{C_{+}}\sum_{\substack{0 < d \leq X\\ \chi_{d}(-M)\omega_E = +1}} \frac{L_E^{'}}{L_E}\left(\frac{1}{2} + \alpha, \chi_d\right)\left(\frac{\psi^{'}}{\psi}-\frac{L_E^{'}}{L_E}\right)\left(\frac{1}{2}-\beta,\chi_d\right)\\
            & \qquad \qquad \qquad \times f\left(-i\alpha, -i\beta\right) d \alpha d \beta\\
        &\qquad \qquad + \int_{C_-}\int_{C_{-}}\sum_{\substack{0 < d \leq X\\ \chi_{d}(-M)\omega_E = +1}} \left(\frac{\psi^{'}}{\psi}-\frac{L_E^{'}}{L_E}\right)\left(\frac{1}{2}-\alpha,\chi_d\right)\\
            & \qquad \qquad \qquad \times \left(\frac{\psi^{'}}{\psi}-\frac{L_E^{'}}{L_E}\right)\left(\frac{1}{2}-\beta,\chi_d\right)f\left(-i\alpha, -i\beta\right) d \alpha d \beta.
\end{split}
\end{align}
By the definition of $J_{E}(A,B)$ in (\ref{def:LFJ}), we can write this as
\begin{align}
\begin{split}
&2^2(2\pi i)^{2} S^E_2(f) \\
    & \qquad = \left(\int_{C_+}\right)^{2}\left(J_{E}\left(\{\alpha, \beta\}, \emptyset \right)\right)f(-i\alpha, -i\beta) d \alpha d \beta\\
        &\qquad \qquad- \int_{C_+}\int_{C_{-}}\left(J_{E}\left(\{\beta,\}, \{-\alpha\}\right)-J_E(\{\beta, -\alpha\},\emptyset)\right)f(-i\alpha, -i\beta) d \alpha d \beta\\
        &\qquad \qquad - \int_{C_-}\int_{C_{+}}\left(J_E\left(\{\alpha,\}, \{-\beta\}\right)-J_E(\{\alpha, -\beta\},\emptyset)\right)f(-i\alpha, -i\beta) d \alpha d \beta\\
        &\qquad \qquad + \left(\int_{C_-}\right)^{2}\left[\vphantom{\frac{1}{3}}J_E\left(\emptyset, \{-\alpha,-\beta\}\right)-J_E\left(\{-\alpha\},\{-\beta\}\right)\right.\\
            &\left.\qquad \qquad \qquad-J_E\left(\{-\beta\},\{-\alpha\}\right)+J_E\left(\{-\alpha, -\beta\},\emptyset \right)\vphantom{\frac{1}{3}}\right]f\left(-i\alpha, -i\beta\right) d \alpha d \beta.
\end{split}
\end{align}

Using Conjecture~\ref{LFJconj}, we can replace the $J_E(A,B)$ with $J_E^{*}(A,B)$ in all of the integrals above.  In this section we regularly make use of the property that $J^*_E(A,B)=J^*_E(A,B^{-})$: because of the definition of $\tfrac{\psi'}{\psi}(1/2+\beta,\chi_d)$ in (\ref{eq:LFf}) the sign of the elements of $B$ makes no difference.
\begin{align}
\sum_{K \cup L \cup M = \{1,2\}}{(-1)^{|M|}\int_{C^K_{+}}{\int_{C^{L\cup M}_{-}}{J_E^{*}\left(z_K \cup -z_L, z_M \right)f(-iz_1, -iz_2) d z_1 d z_2}}}.
\end{align}

The next step is to move the contours onto the imaginary axis. We recall that $J_E^{*}(\{\alpha, \beta\})$ has multiple poles as proved in Theorem \ref{LFresiduef}. There is a pole at $\alpha = -\beta$ with residue $J_E^{*}(\{\beta\}, \emptyset)+J_E^{*}(\{-\beta\}, \emptyset)- J_E^{*}(\emptyset,\{\beta\})$ and poles at $\alpha=0$ and $\beta=0$ with residue $0$. We then move these onto the real axis and combine similar terms, noting that $f(x,y)=f(\pm x, \pm y)$
\begin{align}
\begin{split}
&2^2(2\pi i)^{2} S^E_2(f) \\
    & \qquad =i^{2}\left(\int_{-\infty}^{\infty}\right)^{2} \sum_{\substack{K \cup L \cup M\\ = \{1,2\}}}{(-1)^{|M|}J_E^{*}\left(-iz_K \cup iz_L, -iz_M \right)f(z_1,z_2) d z_1 d z_2}\\
        &\qquad \qquad + 4 \pi i^{2} \int_{-\infty}^{+\infty}\left(J_E^{*}(\{i\beta\}, \emptyset)+J_E^{*}(\{-i\beta\}, \emptyset)-J_E^{*}\left(\emptyset, \{i\beta\}\right)\right)f\left(\beta, \beta\right) d \beta+O(X^{1/2+\epsilon}),
\end{split}\\
\begin{split}
&2^2 (2 \pi)^2 S^E_2(f) \\
    &\qquad =\left(\int_{-\infty}^{\infty}\right)^{2} \sum_{\substack{K \cup L \cup M\\ = \{1,2\}}}{(-1)^{|M|}J_E^{*}\left(-iz_K \cup iz_L, -iz_M \right)f(z_1,z_2) d z_1 d z_2}\\
        &\qquad  \qquad + 2^{2} (2\pi)^{2} S^E_{1}(f)+O(X^{1/2+\epsilon}).
\end{split}
\end{align}

Given our definition of $S^E_2(f)$ at the beginning of this section, it is clear
\begin{align}
S^E_2(f) &= \sum_{\substack{0 < d \leq X\\ \chi_{d}(-M)\omega_E = +1}}{ \sum_{j_1,j_2> 0}f(\gamma_{j_1,d},\gamma_{j_2,d})}\\
&=\sum_{\substack{0 < d \leq X\\ \chi_{d}(-M)\omega_E = +1}}{\sideset{}{^{*}}\sum_{j_1,j_2>0}f(\gamma_{j_1,d},\gamma_{j_2,d})} + S^E_1(f)
\end{align}
where $\sum^{*}$ means terms with repeated variables are excluded. Define
\begin{align}
S^{*}_2(f) &= \sum_{\substack{0 < d \leq X\\ \chi_{d}(-M)\omega_E = +1}}{\sideset{}{^{*}} \sum_{j_1,j_2>0}f(\gamma_{j_1,d},\gamma_{j_2,d})}.
\end{align}
This allows us to conclude that
\begin{align}\label{eq:S*}
S^{*}_2(f) & = \frac{1}{2^2(2\pi)^2}\left(\int_{-\infty}^{\infty}\right)^{2} \sum_{\substack{K \cup L \cup M \\= \{1,2\}}}{(-1)^{|M|}J_E^{*}\left(-iz_K \cup iz_L, -iz_M \right)f(z_1,z_2) d z_1 d z_2}+O(X^{1/2+\epsilon})
\end{align}
as shown in Theorem~\ref{thm:LFNcorr}.

    \subsubsection{Comparing $S^E_2$ to known limiting behavior}

Finally, we want to check that the limit of $S^{*}_2(f)$ as $X \rightarrow \infty$ is as expected i.e. matches that of the orthogonal random matrices from \cite{conrey2005note}. With
\begin{equation}
L=\log\left(\frac{\sqrt{M}X}{2\pi}\right),
\end{equation}
consider
\begin{align}
\lim_{X \rightarrow \infty}\frac{1}{X^{*}} \sum_{\substack{0 < d \leq X\\ \chi_{d}(-M)\omega_E = +1}} \sideset{}{^{*}} \sum_{j_1,j_2>0}{f\left(\frac{\gamma_{j_1,d}L}{ \pi}, \frac{\gamma_{j_2,d}L}{ \pi}\right)} ,\label{eq:rescale}
\end{align}
where $\gamma_{j,d}$ is the imaginary part of the $j$th zero of $L_{E}(s, \chi_d)$ on the critical line and
\begin{align}
X^{*} =  \sum_{\substack{0 < d \leq X\\ \chi_{d}(-M)\omega_E = +1}}{1}.
\end{align}
From (\ref{eq:S*}) we have
\begin{align}\label{eq:example}
\begin{split}
 &\sum_{\substack{0 < d \leq X\\ \chi_{d}(-M)\omega_E = +1}} \sideset{}{^{*}} \sum_{j_1,j_2>0}{f\left(\frac{\gamma_{j_1,d}L}{ \pi}, \frac{\gamma_{j_2,d}L}{ \pi}\right)}\\
 & \qquad = \frac{1}{16\pi^2} \left(\int_{-\infty}^{\infty}\right)^2 \sum_{K \cup L \cup M = \{1,2\}}{(-1)^{|M|}J_E^{*}\left(-i z_K \cup iz_L, -i z_M\right)}\\
 & \qquad \qquad \times {f\left(z_1\frac{L}{ \pi},z_2\frac{L}{ \pi}\right) d z_1 dz_2}+O(X^{1/2+\epsilon})\\
 &\qquad= \frac{1}{16L^2}\left(\int_{-\infty}^{\infty}\right)^2 \sum_{K \cup L \cup M = \{1,2\}}{(-1)^{|M|}J_E^{*}\left(-\frac{i \pi z_K}{L} \cup\frac{ i\pi z_L}{L}, -\frac{i\pi z_M}{L}\right)}\\
 & \qquad \qquad \times {f\left(z_1,z_2\right) d z_1 dz_2}+O(X^{1/2+\epsilon}) .
 \end{split}
\end{align}
We evaluate the resulting sum making the following two large $X$ approximations.  Firstly that
\begin{align}
\frac{1}{X} \sum_{0<d \leq X}f(d) \sim \frac{1}{X^{*}}\sum_{\substack{0 < d \leq X\\ \chi_{d}(-M)\omega_E = +1}}{f(d)}.
\end{align}
Secondly, by the Euler-Maclaurin formula~\cite{jeffrey2008handbook}
\begin{align}
\frac{1}{X} \sum_{0<d \leq X}f(d) \sim \frac{1}{X}\int_{0}^{X}{f(t)} dt.
\end{align}

After the change of variables we will have a $1/L^2$ outside the integral so we are only interested in the terms of order $L^2$ and higher, as the other terms will become negligible as $X \rightarrow \infty$.
\begin{align}
\begin{split}
&J_E^{*}\left(\emptyset, \left\{\frac{-\pi i a}{L}, \frac{- \pi i b}{L}\right\}\right) \\
&\qquad =  \sum_{\substack{0 < d \leq X\\ \chi_{d}(-M)\omega_E = +1}}{\frac{\psi^{'}}{\psi}\left(\frac{1}{2}- \frac{ \pi i a}{L}, \chi_d\right)\frac{\psi^{'}}{\psi}\left(\frac{1}{2}- \frac{\pi i b}{L}, \chi_d\right)}
\end{split}
\end{align}
where $\psi\left(s, \chi_d\right)= \chi_d(-M)\omega_E \left(\frac{2 \pi}{\sqrt{M}|d|}\right)^{2s-1}\frac{\Gamma\left(\frac{3}{2}-s\right)}{\Gamma\left(s+\frac{1}{2}\right)}$ is from the functional equation of $L_{E}(s, \chi_d)$ and so we recall
\begin{align}
\begin{split}
&
\frac{\psi^{'}}{\psi}\left(\frac{1}{2}-s, \chi_d\right) = -2\log\left(\frac{\sqrt{M}|d|}{2 \pi}\right) - \frac{\Gamma^{'}}{\Gamma}\left(1+ s\right) - \frac{\Gamma^{'}}{\Gamma}\left(1-s\right).
\end{split}
\end{align}

We have
\begin{align}
\begin{split}
&\frac{1}{X^{*}}J_E^{*}\left(\emptyset, \left\{\frac{- \pi i a}{L}, \frac{- \pi i b}{L}\right\}\right) \\
 & \qquad \sim \frac{1}{X} \int_{0}^{X} \left(-2\log\left(\frac{\sqrt{M}t}{2\pi}\right)- \frac{\Gamma^{'}}{\Gamma}\left(1+ a\right) - \frac{\Gamma^{'}}{\Gamma}\left(1-a\right)\right)\\
&\qquad \qquad \times \left(-2\log\left(\frac{\sqrt{M}t}{2\pi}\right) - \frac{\Gamma^{'}}{\Gamma}\left(1+ b\right) - \frac{\Gamma^{'}}{\Gamma}\left(1-b\right)\right) dt\\
&\qquad\sim 4L^2.
\end{split}
\end{align}
Next we note that for large $X$
\begin{equation}
\zeta\left(1+\frac{2\pi ia}{L}\right) \sim \frac{L}{2\pi i a}
\end{equation}
and
\begin{equation}
\frac{\zeta'}{\zeta}\left(1+\frac{2\pi ia}{L}\right) \sim -\frac{L}{2\pi i a},
\end{equation}
and so consider terms of the form
\vspace{-5pt}
\begin{align}
\begin{split}
& \frac{1}{X^*}J_E^{*}\left(\left\{\frac{ \pi i a}{L}\right\},\left\{ \frac{-\pi i b}{L}\right\}\right)  \\
    & \qquad =\frac{1}{X^*}\sum_{\substack{0 < d \leq X \\ \chi_{d}(-M)\omega_E = 1}}{\left(-\left(\frac{\sqrt{M}|d|}{2 \pi}\right)^{\frac{-2\pi i a}{L}}\frac{\Gamma(1-\frac{ \pi ia}{L})}{\Gamma(1+\frac{ \pi ia}{L})}\zeta\left(1+\frac{2\pi i a}{L}\right)A_E\left(\left\{\frac{ \pi i a}{L}\right\}\right)\right.}\\
        & {\left.\qquad \qquad - \frac{\zeta^{'}}{\zeta}\left(1+ \frac{2 \pi i a}{L}\right)+A_{\emptyset}\left(\left\{\frac{ \pi i a}{L}\right\}\right)\vphantom{\left(\frac{\sqrt{M}|d|}{2 \pi}\right)^{\frac{4\pi i a}{\log X}}}\right)\frac{\psi^{'}}{\psi}\left(\frac{1}{2}- \frac{ \pi i b}{L}, \chi_d\right)}
 \end{split}\\
& \qquad \sim \left( \frac{1}{X}\int_0^X \left(\frac{\sqrt{M}t}{2\pi} \right)^{\frac{-2\pi i a}{L}}2\log\left(\frac{\sqrt{M}t}{2\pi}\right)\frac{L} {2\pi i a}  dt -\frac{1}{X}\int_0^X 2\log\left(\frac{\sqrt{M}t}{2\pi}\right)\frac{L}{2\pi i a}dt\right)\\
&\qquad \sim \frac{L^2}{\pi i a} \left(\frac{e^{-2\pi i a}} {1-\frac{2\pi i a}{L}} -1\right)\sim\frac{L^2}{\pi i a} \left(e^{-2\pi i a} -1\right) ,
\end{align}
where we recall that $A_E(D)=A_E(D,D,D)$ and we have introduced $A_D(W_r)=A_{D,1}(W_r)+A_{D,2}(W_r)+A_{D,3}(W_r)$ (See Theorem \ref{LFJconj} for definitions).  We also use the fact that $A_E(D)=1$ if all the elements of $D$ have value zero.

So
\begin{equation}
 \frac{1}{X^*}J_E^{*}\left(\left\{\frac{ \pi i a}{L}\right\},\left\{ \frac{-\pi i b}{L}\right\}\right)+ \frac{1}{X^*}J_E^{*}\left(\left\{\frac{- \pi i a}{L}\right\},\left\{ \frac{-\pi i b}{L}\right\}\right)=-4L^2 \frac{\sin(2\pi a)} {2\pi a}
 \end{equation}

Proceeding similarly with the other $J_E^*$ terms, and with the help of computer algebra, we can show that
\begin{align}
\begin{split}
&\lim_{X \rightarrow \infty}\frac{1}{X^{*}}S^{*}_2(f) \\
& \qquad = \frac{1}{4}\left(\int_{-\infty}^{\infty}\right)^{2} f(\theta_1,\theta_2)\det_{2 \times 2}K_{SO,even}\left(\theta_i,\theta_j\right) d \theta_1 d \theta_2
\end{split}
\end{align}
where
\begin{align}
K_{SO,even}(x,y) = \frac{\sin(\pi(y-x))}{\pi(y-x)} + \frac{\sin(\pi(y+x))}{\pi(y+x)}.
\end{align}

This is exactly the form of the kernel for correlations of eigenvalues of  random matrices from $SO(2N)$ \cite{conrey2005note} and hence we have confirmed our result agrees with the conjectured limiting behavior of the family of zeroes\cite{katz1999zz}.


\end{document}